\documentclass[a4paper,10pt,leqno]{amsart}
        \title[A twisted Bass-Heller-Swan decomposition \dots]
{A twisted Bass-Heller-Swan decomposition for the \change{algebraic} $K$-theory of additive categories}

\author{Wolfgang L\"uck}
\email{wolfgang.lueck@him.uni-bonn.de}
       \urladdr{http://www.him.uni-bonn.de/lueck/}
\author{Wolfgang Steimle}
\email{steimle@math.uni-bonn.de}
       \urladdr{http://www.math.uni-bonn.de/people/steimle}   
\address{Universit\"at Bonn\\
               Mathematisches Institut\\
               Endenicher Allee 60,
               D-53115 Bonn, Germany}

       \date{May 2014}
  

\usepackage{pdfsync}
\usepackage{color}
\usepackage{hyperref}
\usepackage{calc}
\usepackage{enumerate,amssymb}
\usepackage{graphicx}
\usepackage[arrow,curve,matrix,tips,2cell]{xy}
  \SelectTips{eu}{10} \UseTips
  \UseAllTwocells
\usepackage{tikz}
\usetikzlibrary{decorations.pathmorphing,snakes}

\DeclareMathAlphabet{\matheurm}{U}{eur}{m}{n}

\newcommand{\change}[1]{#1}

\newcommand{\addcat}{\matheurm{Add\text{-}Cat}}
\newcommand{\addcatic}{\matheurm{Add\text{-}Cat}_{ic}}

\newcommand{\Chcat}{\matheurm{Ch}}
\newcommand{\Chcathf}{\matheurm{Ch}^{\operatorname{hf}}}

\newcommand{\Spectra}{\matheurm{Spectra}}

\DeclareMathOperator{\Ch}{\Chcat}
\DeclareMathOperator{\Chhf}{\Chcat^{hf}}

\DeclareMathOperator{\colim}{colim}

\DeclareMathOperator{\cone}{cone}
\DeclareMathOperator{\cyl}{cyl}
\DeclareMathOperator{\End}{End}
\DeclareMathOperator{\el}{el}
\DeclareMathOperator{\ev}{ev}

\DeclareMathOperator{\HNil}{HNil}

\DeclareMathOperator{\hofib}{hofib}

\DeclareMathOperator{\id}{id}

\DeclareMathOperator{\Idem}{Idem}

\DeclareMathOperator{\mor}{mor}
\DeclareMathOperator{\Nil}{Nil}

\DeclareMathOperator{\pt}{pt}
\DeclareMathOperator{\pr}{pr}

\DeclareMathOperator{\sh}{sh}

  \newcommand{\IZ}{\mathbb{Z}}

   \newcommand{\Ab}{\matheurm{Ab}}

  \newcommand{\cala}{\mathcal{A}}
  \newcommand{\calb}{\mathcal{B}}
  \newcommand{\calc}{\mathcal{C}}
  
  \newcommand{\cale}{\mathcal{E}}
  
  \newcommand{\calg}{\mathcal{G}}

  \newcommand{\calp}{\mathcal{P}}
  
  \newcommand{\calr}{\mathcal{R}}

  \newcommand{\calu}{\mathcal{U}}

  \newcommand{\calx}{\mathcal{X}}
  \newcommand{\caly}{\mathcal{Y}}

  \newcommand{\bfa}{{\mathbf a}}
  
  \newcommand{\bfb}{{\mathbf b}}

  \newcommand{\bfE}{{\mathbf E}}
  \newcommand{\bfe}{{\mathbf e}}
  \newcommand{\bfEinfty}{{\mathbf E}^{\infty}}
  
  \newcommand{\bff}{{\mathbf f}}
  
  \newcommand{\bfg}{{\mathbf g}}

  \newcommand{\bfi}{{\mathbf i}}

  \newcommand{\bfK}{{\mathbf K}}
  \newcommand{\bfKinfty}{{\mathbf K}^{\infty}}

  \newcommand{\bfKNilinfty}{{\mathbf K}_{\Nil}^{\infty}}

  \newcommand{\bfm}{{\mathbf m}}
  \newcommand{\bfN}{{\mathbf N}}
  \newcommand{\bfn}{{\mathbf n}}

  \newcommand{\bfT}{{\mathbf T}}

  \newcommand{\bfu}{{\mathbf u}}

  \newcommand{\bfw}{{\mathbf w}}
  
  \newcommand{\bfx}{{\mathbf x}}
  
  \newcommand{\bfy}{{\mathbf y}}
  
  \newcommand{\bfz}{{\mathbf z}}

\newcommand{\NK}{N\!K}

\newcommand{\bfNK}{{\mathbf N}{\mathbf K}}
\newcommand{\bfNKinfty}{{\mathbf N}{\mathbf K}^{\infty}}

\newcommand{\inv}{^{-1}}

\newcommand{\bfev}{{\mathbf e}{\mathbf v}}

\newcommand{\trun}{[\,]}
\newcommand{\hatxf}{\widehat{\mathcal{X}}^{\operatorname{f}}}
\newcommand{\hatxhf}{\widehat{\mathcal{X}}^{\operatorname{hf}}}

\theoremstyle{plain}
\newtheorem{theorem}{Theorem}[section]

\newtheorem{lemma}[theorem]{Lemma}

\newtheorem{proposition}[theorem]{Proposition}

\theoremstyle{definition}
\newtheorem{definition}[theorem]{Definition}

\newtheorem{example}[theorem]{Example}

\newtheorem{remark}[theorem]{Remark}
\newtheorem{notation}[theorem]{Notation}

\theoremstyle{remark}

\makeatletter\let\c@equation=\c@theorem\makeatother

\newcommand{\version}[1]                       
{\begin{center} Last edited on #1\\
Last compiled on \today
\\
file name: \jobname
\end{center}
}

\newcommand{\squarematrix}[4]
{\left( \begin{array}{cc} #1 & #2 \\ #3 &
#4
\end{array} \right)
}

\newcommand{\xycomsquare}[8]                
{$$\xymatrix{#1 \ar[r]^-{#2} \ar[d]^{#4} &
    #3 \ar[d]^{#5}  \\
    #6\ar[r]^-{#7} & #8 }$$}

\begin{document}

\begin{abstract}
We prove a twisted Bass-Heller-Swan decomposition for both the connective
and  the non-connective  $K$-theory spectrum of  additive categories. 
\end{abstract}

\maketitle


\section*{Introduction}

\subsection*{Statement of the main results}

Let $\cala$ be a (small) additive category together with an automorphism 
$\Phi \colon \cala \xrightarrow{\cong} \cala$ of additive categories. Let $\cala_{\Phi}[t,t^{-1}]$ be
the associated \emph{twisted finite Laurent category} (see Definition~\ref{def:A_phi[t,t(-1)]}) and denote by
$\cala_{\Phi}[t]$ and $\cala_\Phi[t\inv]$ the obvious additive subcategories of
$\cala_{\Phi}[t,t^{-1}]$ (see Definition~\ref{def:A_phi(t_and_t(-1))}). 
Denote by $\bfKinfty(\cala)$ the \emph{non-connective $K$-theory spectrum} 
of the additive category $\cala$.  Denote by $\bfT_{\bfKinfty(\Phi^{-1})}$ the
\emph{mapping torus} of the map of spectra 
$\bfKinfty(\Phi^{-1}) \colon \bfKinfty(\cala) \to \bfKinfty(\cala)$.  
Define $\bfNKinfty(\cala_{\Phi}[t^{\pm 1}])$ to be the homotopy fiber of the map of spectra 
$\bfKinfty(\ev_0^{\pm}) \colon \bfKinfty(\cala_{\Phi}[t^{\pm 1}]) \to \bfKinfty(\cala)$
induced by the functor of additive categories $\ev_0^{\pm} \colon \cala_{\Phi}[t^{\pm 1}] \to \cala$ 
obtained by evaluating at $t=0$. There is a certain \emph{Nil-category} $\Nil(\cala,\Phi)$ for which its
\emph{non-connective $K$-theory} $\bfKinfty_{\Nil}(\cala,\Phi)$ is a certain delooping
of the  connective $K$-theory  $\bfK\bigl(\Nil(\cala,\Phi)\bigr)$.

To talk about functoriality, denote by $\addcat$ the category of small additive categories
and additive functors. Let us consider the group $\IZ$ as a category and denote by
$\addcat^\IZ$ the category of functors $\IZ\to\addcat$, with natural transformations as
morphisms. Note that an object of this category is precisely described by a pair
$(\cala,\Phi)$ as above.

The main theorem of this paper is:

\begin{theorem}[The Bass-Heller-Swan decomposition for non-connective $K$-theory of
  additive categories]
  \label{the:BHS_decomposition_for_non-connective_K-theory}
  Let $\cala$ be an additive category. Let $\Phi \colon \cala \to \cala$ be an
  automorphism of additive categories.

  \begin{enumerate}
  \item \label{the:BHS_decomposition_for_non-connective_K-theory:BHS-iso}

  There exists a weak homotopy equivalence of
  spectra, natural in $(\cala,\Phi)$, 
  \[
  \bfa^{\infty} \vee \bfb_+^{\infty}\vee \bfb_-^{\infty} \colon \bfT_{\bfKinfty(\Phi^{-1})} \vee
  \bfNKinfty(\cala_{\Phi}[t]) \vee \bfNKinfty(\cala_{\Phi}[t^{-1}]) \xrightarrow{\simeq}
  \bfKinfty(\cala_{\Phi}[t,t^{-1}]);
  \]
  
\item \label{the:BHS_decomposition_for_non-connective_K-theory:Nil} There exist a functor
  $\bfEinfty\colon \addcat^\IZ\to\Spectra$ and weak homotopy equivalences of spectra,
  natural in $(\cala,\Phi)$,
  \begin{eqnarray*}
   \Omega \bfNKinfty(\cala_{\Phi}[t]) & \xleftarrow{\simeq} & \bfE^{\infty}(\cala,\Phi);
   \\
   \bfKinfty(\cala) \vee  \bfE^{\infty} (\cala,\Phi)  & \xrightarrow{\simeq} & \bfKNilinfty(\cala,\Phi).
 \end{eqnarray*}
  \end{enumerate}
\end{theorem}

Next we state what we get  after applying homotopy groups. 

\begin{remark}[Wang sequence] \label{rem:Wang_sequence}
We obtain for all $n \in \IZ$ a natural splitting
\[
K_n(\cala_{\Phi}[t,t^{-1}]) \xrightarrow{\cong} C_n(\cala_{\Phi}[t,t^{-1}]) \oplus \NK_n(\cala) \oplus \NK_n(\cala),
\]
if we define $C_n(\cala_{\Phi}[t,t^{-1}])$ to be the cokernel of the split injective homomorphism
$ K_n(\bfb^{\infty}_+) \oplus K_n(\bfb^{\infty}_-) \colon \NK_n(\cala) \oplus \NK_n(\cala) \to K_n(\cala_{\Phi}[t,t^{-1}])$, 
and get a long exact Wang sequence, infinite to both sides,
\begin{multline*}
\ldots \xrightarrow{\partial_{n+1}} K_n(\cala) \xrightarrow{K_n(\Phi) - \id} K_n(\cala) \xrightarrow{K_n(i_0)} C_n(\cala_{\Phi}[t,t^{-1}])
\\
\xrightarrow{\partial_n} K_{n-1}(\cala) \xrightarrow{K_{n-1}(\Phi) - \id} K_{n-1}(\cala)  \xrightarrow{K_{n-1}(i_0)} \cdots.
\end{multline*}
\end{remark}

\begin{example}[Finitely generated free $R$-modules]\label{exa:Finitely_generated_free_R-modules}    
  Let $R$ be an (associative) ring (with unit).
  Let $\calr$ be the category whose objects consist of natural numbers $m =
  0,1,2 \ldots$ and whose morphisms from $m$ to $n$ are given by the abelian
  group of \change{$n$-by-$m$}-matrices with entries in $R$.  The composition is given by
  matrix multiplication. The (categorical) direct sum of $m$ and $n$ is $m+n$ and on morphisms
  given by taking block matrices. 

  Then $\calr$ is a skeleton of the category of  finitely generated free right $R$-modules,
  and $\calr[t,t^{-1}] = \calr_{\id}[t,t^{-1}]$ is a skeleton for the category
  of finitely generated free  modules over the group ring $R[t,t^{-1}]$. 
  In this situation Theorem~\ref{the:BHS_decomposition_for_non-connective_K-theory}~\ref{the:BHS_decomposition_for_non-connective_K-theory:BHS-iso} 
  reduces for $\cala = \calr$ to the classical Bass-Heller Swan isomorphism
  \[
  K_n(R) \oplus K_{n-1}(R) \oplus \NK_n(R) \oplus \NK_n(R) \xrightarrow{\cong} K_n(R[t,t^{-1}]) \quad \text{for}\; n \in \IZ,
  \]
  and Theorem~\ref{the:BHS_decomposition_for_non-connective_K-theory}~\ref{the:BHS_decomposition_for_non-connective_K-theory:Nil}
  reduces for $\cala = \calr$ and $n \ge 0$ to the classical isomorphism
   \[
   K_{n}(\Nil(R)) \xrightarrow{\cong} K_n(R) \oplus \NK_{n+1}(R).
   \]
If $R$ comes with a ring automorphism $\phi \colon R \to R$ and we equip $\calr$
with
the induced automorphism $\Phi \colon \calr \xrightarrow{\cong} \calr$, then
$\calr_\Phi[t,t\inv]$ is equivalent to the category of finitely generated free
modules over the twisted group ring $R_\phi[t,t\inv]$. Hence
Theorem~\ref{the:BHS_decomposition_for_non-connective_K-theory}~\ref{the:BHS_decomposition_for_non-connective_K-theory:BHS-iso}
   provides, after applying $\pi_n$, a twisted
   Bass-Heller-Swan decomposition of the twisted group ring $R_\phi[t,t\inv]$.

 \end{example}

 There is also a version for the \emph{connective $K$-theory spectrum} $\bfK$. Denote by
 $\addcatic\subset\addcat$ the full subcategory on idempotent complete categories.

\begin{theorem}[The Bass-Heller-Swan decomposition for connective $K$-theory of
  additive categories]
  \label{the:BHS_decomposition_for_connective_K-theory}
  Let $\cala$ be an additive category which is idempotent complete. Let $\Phi \colon \cala \to \cala$ be an
  automorphism of additive categories. 

  \begin{enumerate}
  \item \label{the:BHS_decomposition_for_connective_K-theory:BHS-iso}

  Then there is a weak equivalence of  spectra, natural in $(\cala,\Phi)$,
  \[
  \bfa \vee \bfb_+\vee \bfb_- \colon \bfT_{\bfK(\Phi^{-1})} \vee
  \bfNK(\cala_{\Phi}[t]) \vee \bfN\bfK(\cala_{\Phi}[t^{-1}]) \xrightarrow{\simeq}
  \bfK(\cala_{\Phi}[t,t^{-1}]);
  \]

  \item \label{the:BHS_decomposition_for_connective_K-theory:Nil}
  There exist a functor $\bfE\colon (\addcatic)^\IZ\to\Spectra$ and weak homotopy equivalences of spectra, natural in $(\cala,\Phi)$,
  \begin{eqnarray*}
   \Omega \bfN\bfK(\cala_{\Phi}[t]) & \xleftarrow{\simeq} & \bfE(\cala,\Phi);
   \\
   \bfK(\cala) \vee  \bfE (\cala,\Phi)  & \xrightarrow{\simeq} & \bfK(\Nil(\cala;\Phi)).
 \end{eqnarray*}
   \end{enumerate}
\end{theorem}

We emphasize that for the connective version some care is necessary concerning the
interpretation after applying $\pi_n$ in the case $n = 0$, since in contrast to the
non-connective $K$-theory spectrum the passage from an additive category to its idempotent
completion does change the zeroth $K$-group and the assumption that $\cala$ is idempotent
complete does not imply that $\cala_{\Phi}[t]$, $\cala_{\Phi}[t^{-1}]$, or
$\cala_{\Phi}[t,t^{-1}]$ is idempotent complete.  At least the canonical inclusion of an
additive category in its idempotent completion yields a map on the connected $K$-theory
spectra which induces isomorphisms on $\pi_n$ for $n \ge 1$. 

Recall that $K_0(\cala)$ is obtained as the Grothendieck construction of the abelian monoid
of  stable isomorphism classes of objects in $\cala$ under direct sum. 
(Two objects $A_0$ and $A_1$ are stably isomorphic if there exists an object $B$ such that
$A_0 \oplus B$ and $A_1 \oplus B$ are isomorphic.) We get for the connective version in degree $0$
\[
\pi_0\bigl( \bfNK(\cala_{\Phi}[t]) \bigr) =  \pi_0\bigl( \bfNK(\cala_{\Phi}[t^{-1}]) \bigr) = 0,
\]
since $i_{\pm} \colon \cala \to \cala_{\Phi}[t^{\pm 1}]$ is  bijective on objects and
$\ev_0^{\pm} \circ i_{\pm} = \id_{\cala}$, 
and therefore $\pi_0(\bfK(\ev_0^{\pm})) \colon  K_0(\cala_{\Phi}[t^{\pm 1}]) \to K_0(\cala)$ is bijective.
The Wang sequence associated to 
Theorem~\ref{the:BHS_decomposition_for_connective_K-theory}~%
\ref{the:BHS_decomposition_for_connective_K-theory:BHS-iso}
agrees with the one in Remark~\ref{rem:Wang_sequence} in degree $n \ge 1$ and ends in degree zero by
\[
\ldots \xrightarrow{\partial_{1}} K_0(\cala) \xrightarrow{K_0(\Phi) - \id} K_0(\cala) 
\xrightarrow{K_0(i_0)} K_0(\cala_{\Phi}[t,t^{-1}]) \to 0.
\]

\subsection*{Relation to other work}
We start by giving a (incomplete) list of previous work on the Bass-Heller-Swan decomposition.
In~\cite{Bass-Heller-Swan(1964)}, Bass-Heller-Swan proved a decomposition of
$K_1(R[t,t\inv])$ for regular rings $R$. Original sources for the Bass-Heller-Swan
decomposition of $K_1(R[t,t^{-1}])$ for an arbitrary ring $R$ are
Bass~\cite[Chapter~XII]{Bass(1968)} and Swan~\cite[Chapter~16]{Swan(1968)}. Bass
used the decomposition to define negative $K$-groups and to extend the
Bass-Heller-Swan decomposition
in this range. Ranicki~\cite[Chapter~10]{Ranicki(1992a)} extended this decomposition
of middle and lower $K$-groups to additive categories.
Farrell-Hsiang~\cite{Farrell-Hsiang(1970)} gave a decomposition of $K_1$ of twisted
group rings. More treatments of the classical Bass-Heller-Swan decomposition can be
found e.g.,~in~\cite[Theorem~3.2.22 on page~149, Theorem~3.3.3 on page~155,
Theorem~5.3.30 on page~295]{Rosenberg(1994)} and~\cite[Theorem~9.8 on
page~207]{Srinivas(1991)}.

Grayson~\cite{Grayson(1976)} proved a Bass-Heller-Swan decomposition on the
level of higher algebraic $K$-groups, restricting to the case of a ring. In later
work~\cite{Grayson(1988)} he generalized this result to the case of a twisted group
ring. The connective $K$-theory of generalized Laurent extensions of rings is treated
in Waldhausen~\cite{Waldhausen(1978genfreeI+II),Waldhausen(1978genfreeIII+IV)}.
H\"uttemann-Klein-Vogell-Waldhausen-Williams~\cite{Huettemann-Klein-Waldhausen-Williams(2001)}
proved a Bass-Heller-Swan decomposition for connective algebraic $K$-theory of
spaces on the spectrum level; Klein-Williams~\cite{Klein-Williams(2008)}
identified the relative terms with the $K$-theory spectrum of homotopy-nilpotent
endomorphisms.

In a companion paper~\cite{Lueck-Steimle(2013delooping)} to the present work we will, 
building on Bass's approach, develop a non-connective delooping
machine for functors from additive categories to spectra which are $n$-contracting, which
roughly speaking means that the (untwisted) Bass-Heller-Swan map is bijective on $\pi_i$
for $i \ge n+1$ and its reduced version is split injective on $\pi_i$ for $i \le n$.  It
will come with a universal property. This will enable us to make sense of
$\bfK^{\infty}(\cala)$ and $\bfKinfty(\Nil(\cala;\Phi))$ and to deduce
Theorem~\ref{the:BHS_decomposition_for_non-connective_K-theory} from
Theorem~\ref{the:BHS_decomposition_for_connective_K-theory}.  This delooping machine is of
interest in its own right since it is rather elementary and comes with a universal
property.

A definition of $\bfK^{\infty}(\cala)$ for an additive category $\cala$ has also been
given by Pedersen-Weibel using controlled topology in~\cite{Pedersen-Weibel(1985)}.  It
can be identified with our approach using the universal property. It is also obvious from
the construction of our approach that $K_n(\cala)$ agrees with the original definition of
Bass using contracting functors.  Notice that we cannot define
$\bfKinfty(\Nil(\cala;\Phi))$ using Pedersen-Weibel~\cite{Pedersen-Weibel(1985)} since we
use a different exact structure than the one coming from split exact sequences. There is a
definition of negative $K$-groups for exact categories presented by
Schlichting~\cite{Schlichting(2006)} \change{which has not} yet been identified with our
approach for $\Nil(\cala,\Phi)$.

The result of this paper will play a key role in a forthcoming paper by the same 
authors~\cite{Lueck-Steimle(2013splitasmb)}
where an explicit splitting on spectrum level of the relative Farrell-Jones assembly map
from the family of finite subgroups to the family of virtually subgroups is given and the
involution on the relative term is analyzed.  Such a splitting, but without identifying
the relative term, has already been constructed by Bartels~\cite{Bartels(2003b)} using
controlled topology.

We try to keep the presentation of the present paper as self-contained as possible,
relying just on some fundamental results in algebraic $K$-theory
\cite{Cisinki(2010inv),Thomason-Trobaugh(1990),Waldhausen(1985)}, the companion
paper~\cite{Lueck-Steimle(2013delooping)}, and some very basic stable homotopy
theory and category theory.
While Quillen's setting for algebraic $K$-theory is very
well adapted to proving the Bass-Heller-Swan decomposition for rings, it is not for
the more general setup of additive categories, as the necessary localization
sequences are not available. The remedy is to pass to the category of chain
complexes over $\cala$, which is a Waldhausen category, i.e., categories with weak
equivalences and cofibrations in the sense of Waldhausen~\cite{Waldhausen(1985)},
and to use Waldhausen's approach to algebraic $K$-theory. Thus it is not surprising
that our proof of Theorem~\ref{the:BHS_decomposition_for_connective_K-theory}
follows the same global pattern as the one given in the non-linear setting by
H\"uttemann-Klein-Vogell-Waldhausen-Williams~\cite{Huettemann-Klein-Waldhausen-Williams(2001)}
and Klein-Williams~\cite{Klein-Williams(2008)}. Some extra work is necessary to pass back from
the category of homotopy-nilpotent endomorphisms of chain complexes over $\cala$
(``Waldhausen setting'') to the category of nilpotent endomorphisms in $\cala$
(``Quillen setting''). Such a reduction was carried out by
Ranicki~\cite[Chapter~9]{Ranicki(1992a)} on the level of path-components.

\subsection*{Acknowledgement}
This paper has been financially supported by the Leibniz-Award, granted by the Deutsche
Forschungsgemeinschaft, of the first author. The second author was also supported by the
ERC Advanced Grant 288082 and by the Danish National Research Foundation through the
Centre for Symmetry and Deformation (DNRF92). We thank the referee for his detailed report
and helpful suggestions.

\setcounter{tocdepth}{1}
\tableofcontents


\section{Preliminaries  about additive categories}
\label{sec:Preliminaries_about_additive_categories}

In this section we present some basics about additive categories


\subsection{The twisted finite Laurent category}
\label{subsec:The_twisted_finite_Laurent_category}

Let $\cala$ be an additive category. Let $\Phi \colon \cala \to \cala$ be an
automorphism of additive categories. 

\begin{definition}[{Twisted finite Laurent category $\cala_{\Phi}[t,t^{-1}]$}] \label{def:A_phi[t,t(-1)]}
  Define the \emph{$\Phi$-twisted finite Laurent category}  $\cala_{\Phi}[t,t^{-1}]$ as follows. It has the
  same objects as $\cala$.  Given two objects $A$ and $B$, a morphism 
  $f  \colon A \to B$ in $\cala_{\Phi}[t,t^{-1}]$ is a formal sum 
  $f = \sum_{i \in    \IZ} f_i \cdot t^i$, where $f_i \colon \Phi^i(A) \to B$ is a morphism in
  $\cala$ from $\Phi^i(A)$ to $B$ and only finitely many of the morphisms $f_i$ are non-trivial. If 
  $g   = \sum_{j \in \IZ} g_j \cdot t^j$ is a morphism in $\cala_{\Phi}[t,t^{-1}]$
  from $B$ to $C$, we define the composite $g \circ f \colon A \to C$ by
  \[
  g \circ f := \sum_{k \in \IZ} \left( \sum_{\substack{i,j \in \IZ,\\i+j = k}}
    g_j \circ \Phi^j(f_i)\right) \cdot t^k.
  \]  
  The  direct sum and the structure of an abelian
  group on the set of morphism from $A$ to $B$ in $\cala_{\Phi}[t,t^{-1}]$ are
  defined in the obvious way using the corresponding structures in
  $\cala$.
\end{definition}

  So the decisive relation is for a morphism $f \colon A \to B$ in $\cala$
  \[
   (\id_{\Phi(B)} \cdot t) \circ (f \cdot t^0) = \Phi(f) \cdot t.
  \]

  We have already explained in
  Example~\ref{exa:Finitely_generated_free_R-modules}     that for a ring
  $R$ with automorphism $\phi$ the passage from $\calr$ to $\calr_{\Phi} [t,t^{-1}]$
  corresponds to the passage of finitely generated free modules over
  $R$ to finitely generated free modules over the twisted group ring
  $R_{\phi}[t,t^{-1}]$.

  \begin{definition}[{$\cala_{\Phi}[t]$} and {$\cala_{\Phi}[t^{-1}]$}] 
   \label{def:A_phi(t_and_t(-1))} 
    Let $\cala_{\Phi}[t]$
    and $\cala_{\Phi}[t^{-1}]$ respectively be the additive subcategory of
    $\cala_{\Phi}[t,t^{-1}]$ whose set of objects is the set of objects in
    $\cala$ and whose morphisms from $A$ to $B$ are the formal sums 
    $\sum_{i \in \IZ} f_i \cdot t^i$ with $f_i = 0$ for $i < 0$
    and $i > 0$ respectively (in other words, polynomials in $t$ and $t\inv$, 
    respectively).
  \end{definition}
  
  In the setting of Example~\ref{exa:Finitely_generated_free_R-modules} the additive
  subcategories $\calr_{\Phi}[t]$ and $\calr_{\Phi}[t^{-1}]$ of
  $\calr_{\Phi}[t,t^{-1}]$ correspond to the category of finitely
  generated free modules over the subrings $R_{\phi}[t]$ and
  $R_{\phi}[t^{-1}]$ of $R_{\phi}[t,t^{-1}]$.


\subsection{Idempotent completion}
\label{subsec:Idempotent_completion}

Given an additive category $\cala$, its \emph{idempotent completion}
$\Idem(\cala)$ is defined to be the following additive
category. Objects are morphisms $p \colon A \to A$ in $\cala$
satisfying $p \circ p = p$.  A morphism $f$ from $p_1 \colon A_1 \to
A_1$ to $p_2 \colon A_2 \to A_2$ is a morphism $f \colon A_1 \to A_2$
in $\cala$ satisfying $p_2 \circ f \circ p_1 = f$.  If $\cala$ has the
structure of an additive category then $\Idem(\cala)$ inherits such a
structure, and if $\cala$ has a preferred choice of finite or
countable direct sums then so does $\Idem(\cala)$. Obviously a functor of additive categories $F \colon \cala\to \calb$ 
induces a functor $\Idem(F) \colon \Idem(\cala) \to \Idem(\calb)$ 
of additive categories. There is a obvious embedding
\[
\eta(\cala)\colon\cala \to \Idem(\cala)
\]
 sending an objects $A$ to
$\id_A \colon A \to A$ and a morphism $f \colon A \to B$ to the
morphisms given by $f$ again.  An additive category $\cala$ is called
\emph{idempotent complete} if $\eta(\cala)\colon \cala \to \Idem(\cala)$ is an
equivalence of additive categories, or, equivalently, 
if for every idempotent $p \colon A \to A$ in $A$
there exists objects $B$ and $C$ and an isomorphism 
$f \colon A \xrightarrow{\cong} B \oplus C$ in $\cala$
such that $f \circ p \circ f^{-1} \colon B \oplus C \to B \oplus C$ is given by 
$\squarematrix{\id_B}{0}{0}{0}$.  The idempotent completion $\Idem(\cala)$ 
is idempotent complete.

Given a ring $R$, then $\Idem(\calr)$ of the additive category $\calr$
defined in Example~\ref{exa:Finitely_generated_free_R-modules} is a
skeleton of the additive category of finitely generated projective
$R$-modules.

\begin{theorem}[Passage to the idempotent completion]
\label{the:Passage_to_the_idempotent_completion}
Let $\cala$ be an additive category and let $\eta(\cala) \colon \cala \to \Idem(\cala)$
be the canonical embedding into its idempotent completion.

\begin{enumerate}

\item \label{the:Passage_to_the_idempotent_completion:connective}
The map of connective spectra  $\bfK(\eta(\cala)) \colon \bfK(\cala) \to \bfK(\Idem(\cala))$
induces an isomorphism on $\pi_n$ for $n \ge 1$ and an injection for $n  = 0$;

\item \label{the:Passage_to_the_idempotent_completion:non-connective}
The map of non-connective spectra  $\bfKinfty(\eta(\cala)) \colon \bfKinfty(\cala) \to \bfKinfty(\Idem(\cala))$
is a weak homotopy equivalence.

\end{enumerate}
\end{theorem}
\begin{proof}~\ref{the:Passage_to_the_idempotent_completion:connective}
This is proved in~\cite[Theorem~A.9.1.]{Thomason-Trobaugh(1990)}. 
\\[2mm]~\ref{the:Passage_to_the_idempotent_completion:non-connective}
This follows from assertion~\ref{the:Passage_to_the_idempotent_completion:connective}
and~\cite[Corollary~3.7]{Lueck-Steimle(2013delooping)}.
\end{proof}


\subsection{Infinite direct sums}
\label{subsec:finite_direct_sum}

There is a functorial way of adjoining countable direct sums to a $\IZ$-category
$\cala$ (compare~\cite[Lemma~9.2]{Bartels-Reich(2007coeff)}).  We go carefully through
such a  construction. In order to avoid set theoretic problems, we fix for once and all a
\emph{universe}, i.e., an infinite set $\calu$ with base point $u \in \calu$ and a
bijection 
\begin{align}
\tau \colon \calu \times \calu & \xrightarrow{\cong} \calu
\label{bijection_tau}
\end{align}
with $\tau(u,u) = u$.   In the sequel all index sets will be subsets of $\calu$.  Since we will essentially only deal 
with countable index sets, we could take $\calu$ to be the set of integers $\IZ$ and $u$ to be $0 \in \IZ$.  
Given a subset $J \subseteq \calu$ and for every $j \in J$ a subset $I_j \subseteq \calu$,
it is not clear how the disjoint union of the $I_j$-s can be considered as subset of $\calu$, but we can consider
instead $\tau(\calg(I_j, j \in J))$ for the \emph{graph}
$\calg(I_j, j \in J) := \{(i,j) \in \calu \times \calu \mid j \in J, i \in I_j\} \subseteq \calu \times \calu$.

Choose an infinite  cardinal $\kappa$ such that the cardinality of $\calu$ is greater or equal to
$\kappa$.  Consider a $\IZ$-category $\cala$, i.e, a small category $\cala$ enriched over abelian
groups.  Next we define another $\IZ$-category $\cala^{\kappa}$ with preferred $\kappa$-direct sums,
i.e., for each subset
$I\subseteq \calu$ of cardinality less or equal to $\kappa$ and a collection $(A_i)_{i \in  I}$ 
of objects in $\cala^{\kappa}$, there is a preferred object $\bigoplus_{i \in I} A_i$ which is a direct sum of the
collection $(A_i)_{i \in I}$, i.e., for each $i \in I$ there exists a preferred morphism 
$\iota_i \colon A_i \to\bigoplus_{i \in I} A_i$ such that for any object $B$ the map
\[
\mor_{\cala^{\kappa}}\left(\bigoplus_{i \in I} A_i,B\right) \xrightarrow{\cong} 
\prod_{i \in I} \mor_{\cala^{\kappa}}(A_i,B), \quad f \mapsto (f \circ \iota_i)_{i \in I}
\]
is an isomorphism of abelian groups.

An object $A$ in $\cala^{\kappa}$ is given by a subset $I \subseteq \calu$ of cardinality
less or equal to $\kappa$ and a map from $I$ to the set of objects of $\cala$, in other
words, by a collection of objects $(A_i)_{i \in I}$.  A morphism $f = (f_{i,j})_{(i,j) \in
  I \times J} \colon A = (A_i)_{i \in I} \to B = (B_j)_{j \in J}$ is given by a collection
of morphisms $f_{i,j} \colon A_i \to B_j$ indexed by $(i,j) \in I \times J$ such that for
every $i \in I$ the set $\{j \in J \mid f_{i,j} \not= 0\}$ is finite. The composite of the
morphism above with the morphism $g = (g_{j,k})_{(j,k) \in J \times K} \colon B = (B_j)_{j
  \in J} \to C = (C_k)_{k \in K}$ is the morphism $g \circ f = (g \circ f)_{i,k})_{(i,k)
  \in I \times K} \colon A = (A_i)_{i \in I} \to C = (C_k)_{k \in K}$, where
\[
(g \circ f)_{i,k} := \sum_{j \in J} g_{j,k} \circ f_{i,j} \colon A_i \to C_k.
\]
The composition is well-defined and for given $i \in I$ the set $\{k \in K \mid (g \circ f)_{i,k}\}$ 
is finite, since the set $\{(j,k) \in J \times K \mid f_{i,j} \not= 0\;  \text{and} \; g_{j,k} \not= 0\}$ is finite 
for each $i \in I$.  For  two morphism 
$f =(f_{i,j})_{(i,j) \in I \times J}$ and $(f'_{i,j})_{(i,j) \in I \times J}$ from $A = (A_i)_{i  \in I}$ to $B = (B_j)_{j \in J}$ 
define $f + f' \colon A \to B$ by $(f_{i,j} + f'_{i,j})_{(i,j) \in   I \times J}$.

Consider a subset $J \subseteq \calu$ of cardinality less or equal to $\kappa$ and a
collection of objects $(A[j])_{j \in J}$ in $\cala^{\kappa}$. We want to define a
model for the direct sum $\bigoplus_{j \in J} A[j]$.  Each object $A[j]$ is given by a
subset $I[j] \subseteq \calu$ of cardinality less or equal to $\kappa$ and a collection
$(A[j])_{i[j]})_{i[j] \in I[j]}$ of objects in $\cala$. The object $\bigoplus_{j \in J} A[j]$ 
is defined by the subset $\tau(\calg(I_j,j \in J)) \subseteq \calu$, which indeed
has cardinality less or equal to $\kappa$, and the collection of objects
$(A[\tau^{-1}_2(k)]_{\tau^{-1}_1(k)})_{k \in \tau(\calg(I_j,j \in J))}$, where
$\tau^{-1}_i(k) $ denotes the $i$-th component of $\tau^{-1}(k)$ in $\calu \times \calu$
for $i = 1,2$. Consider an object $B = (B_l)_{l \in L}$ in $\cala^{\kappa}$ and a
collection of morphism $(f[j] \colon A[j] \to B)_{j \in J}$ of morphisms in
$\cala^{\kappa}$. We want to define their direct sum
\[
\bigoplus_{j \in J} f[j] \colon \bigoplus_{j \in J} A[j] \to B.
\]
Hence for every $k \in \tau(\calg(I_j,j \in J))$ and $l \in L$ we have to specify a
morphism from $A[\tau^{-1}_2(k)]_{\tau^{-1}_1(k)}$ to $B_l$. We just take
$f[\tau^{-1}_2(k)]_{\tau^{-1}_1(k),l} \colon A[\tau^{-1}_2(k)]_{\tau^{-1}_1(k)} \to B_l$.
One easily checks that for given $k \in \tau(\calg(I_j,j \in J))$ the set 
$\{l \in L \mid f[\tau^{-1}_2(k)]_{\tau^{-1}_1(k),l} \not= 0\}$ is finite and that we just 
have defined a preferred $\kappa$-direct sum for $\cala^{\kappa}$.

Consider a (small) $\IZ$-category $\calb$ with a preferred $\kappa$-direct sum.  
Then the forgetful functor sending a $\IZ$-category $\calb$ with preferred $\kappa$-direct sum
to the \change{underlying} $\IZ$-category $\widehat{\calb}$ has a left adjoint, namely, $\cala \mapsto \cala^{\kappa}$.
Hence any functor of $\IZ$ categories $F \colon \cala \to \widehat{\calb}$ extends in unique way to a functor
$F^{\kappa} \colon \cala^{\kappa} \to \calb$ respecting the preferred $\kappa$-direct
sums.  In particular there is a well-defined functor of $\IZ$-categories with preferred
$\kappa$-direct sums extending $\id \colon \widehat{\cala^{\kappa}} \to \widehat{\cala^{\kappa}}$
\begin{eqnarray}
\zeta^{\kappa}   \colon (\widehat{\cala^{\kappa}})^{\kappa} & \to & \cala^{\kappa}.
\label{zetakappa}
\end{eqnarray}

All the constructions above go through also in the case, where one replaces the condition
of cardinality less or equal $\kappa$ by the condition being finite. We denote the
resulting $\IZ$-category with preferred finite  direct sums (over finite index sets of
$\calu$) by $\cala^f$. Thus one can extend a $\IZ$-category $\cala$ to an additive
category with preferred finite  sums. The construction of
$\zeta^{\kappa}$ carries over to $\zeta^f$ in the obvious way. The analogue of \eqref{zetakappa} in this case is an isomorphism of categories
\begin{eqnarray}
\zeta^{f}   \colon (\widehat{\cala^{f}})^{f} & \xrightarrow{\cong} & \cala^{f}.
\label{zetaf}
\end{eqnarray}
There is a canonical inclusion $\cala^f \to \cala^{\kappa}$ respecting the preferred finite direct sums.

Let $\cala$ be an additive category. Recall that this is a small $\IZ$-category with the
property that for two objects their direct sum exists, but there is no preferred model for
the direct sum required.  Let $\widehat{\cala}$ be the $\IZ$-category obtained from
$\cala$ by forgetting the existence of the direct sums.  Then $\widehat{\cala}^f$ is an
additive category, but now with preferred finite sums.  There is an equivalence of
additive categories
\[
F \colon \cala \xrightarrow{\simeq} \widehat{\cala}^f.
\]
It sends an object $A$ in $\cala$ to the object in $\widehat{\cala}^f$ given by the subset
$\{u\}$ of $\calu$ and the collection of objects indexed by $\{u\}$ whose only member is
$A$. The definition of $F$ on morphisms is now obvious.  If we choose the structure of a preferred 
finite direct sums on $\cala$, 
we obtain in the obvious way an equivalence
of additive categories with preferred finite direct sums
\[
\widehat{\cala}^f \xrightarrow{\sim} \cala.
\]
So if $\cala$ is already an additive category, we can replace $\cala$ by $\widehat{\cala}^f$ without harm.

\begin{example}
\label{exa:kappa-generated_free_modules}
Let $R$ be an (associative) ring (with unit). Let $\cala_R$ be the $\IZ$-category with one object
$\ast$ and set of morphisms $\mor_{\cala_R}(\ast,\ast) = R$. Composition is given by the multiplication
in $R$ and the $\IZ$-structure comes from the addition in $R$. Then is $\cala_R^f$ is another model
for the additive category $\calr$ of Example~\ref{exa:Finitely_generated_free_R-modules}
and in particular a skeleton for the additive category of finitely generated free $R$-modules.
The category $\cala_R^{\kappa}$ is a skeleton of the category of free $R$-modules
which have $R$-basis of cardinality less or equal to $\kappa$.
\end{example}

\begin{notation} In the sequel we will often write for  $\widehat{\cala}$ just $\cala$ again.
In particular $(\widehat{\cala})^{\kappa}$ will be written as $\cala^{\kappa}$. Moreover we think of
$\cala$ as sitting in $\cala^{\kappa}$ by interpreting $\cala$ as $\widehat{\cala}^f$.
\end{notation}


\subsection{Induction}
\label{subsec:induction}

Define functors of additive categories
\begin{align}
i_0 \colon  \cala  & \to   \cala_{\Phi}[t,t^{-1}];
\label{def_of_i_0_Lueck}
\\
i_{\pm}  \colon  \cala  & \to    \cala_{\Phi}[t^{\pm 1}];
\label{def_of_i_pm0_Lueck}
\\
 j_{\pm}  \colon  \cala_{\Phi}[t^{\pm 1}]  &\to   \cala_{\Phi}[t,t^{-1}];
\label{def_of_j_pm}
\\
 \ev_0^{\pm} \colon  \cala_{\Phi}[t^{\pm 1}]  & \to  \cala
\label{def_of_ev_opm0_Lueck}
\end{align}
as follows. The functors $i_0$, $i_+$ and $i_-$ send a morphism $f \colon A \to B$ in
$\cala$ to the morphism $f \cdot t^0 \colon A \to B$.  The functors
$j_{\pm}$ are just the inclusions.  The functor $\ev_0^{\pm} \colon\cala_{\Phi}[t^{\pm 1}] \to \cala$ 
is given by evaluation at $t^0$, i.e., it sends a morphism $\sum_{i \ge 0} f_i \cdot t^i$ in $\cala_{\Phi}[t]$
or $\sum_{i \le 0} f_i \cdot t^i$ in $\cala_{\Phi}[t^{-1}]$
respectively to $f_0$.  Notice that $\ev_0^{\pm} \circ i_{\pm}$ is the
identity $\id_{\cala}$ and $i_0 = j_+ \circ i_+ = j_- \circ i_-$.

These functors extend (by applying the functor $\cala \mapsto \cala^{\kappa}$) to the functors
denoted by the same symbols

\begin{align}
i_0 \colon  \cala^{\kappa}  & \to   \cala_{\Phi}[t,t^{-1}]^{\kappa};
\label{def_of_i_0_kappa}
\\
i_{\pm}  \colon  \cala^{\kappa} & \to    \cala_{\Phi}[t^{\pm 1}]^{\kappa};
\label{def_of_i_pm_kappa}
\\
 j_{\pm} \colon  \cala_{\Phi}[t^{\pm 1}]^{\kappa}  &\to   \cala_{\Phi}[t,t^{-1}]^{\kappa};
\label{def_of_j_pm_kappa}
\\
\ev_0^{\pm} \colon  \cala_{\Phi}[t^{\pm 1}]^{\kappa}  & \to  \cala^{\kappa}.
\label{def_of_ev_opm_kappa}
\end{align}


\subsection{Restriction}
\label{subsec:restriction}

In the setting of Example~\ref{exa:Finitely_generated_free_R-modules} the additive
subcategories $\calr_{\Phi}[t]$ and $\calr_{\Phi}[t^{-1}]$ of $\calr_{\Phi}[t,t^{-1}]$
correspond to the categories of finitely generated free modules over the subrings
$R_{\phi}[t]$ and $R_{\phi}[t^{-1}]$ of $R_{\phi}[t,t^{-1}]$, respectively, and the
functors $i_0$, $i_+$ and $i_-$ corresponds to induction. If we allow countably generated free
modules, it is well known that all the three functors have right adjoints, given by restriction.
Next we  extend this construction to additive categories.

To define restriction, we need to fix an embedding 
\begin{align}
\sigma \colon \IZ &\to \calu
\label{injection_sigma}
\end{align}
satisfying $\sigma(0) = u$. Actually we will suppress in the sequel $\sigma$ in the notation 
and think of $\IZ$ as a subset of $\calu$ with $0 = u$.

Define functors
\begin{align}
i^0\colon \cala_\Phi[t,t\inv]^{\kappa} & \to \cala^\kappa;
\label{def_of_i_upper_0_kappa_Lueck}
\\
i^\pm\colon \cala_\Phi[t^{\pm1}]^{\kappa} & \to \cala^\kappa,
\label{def_of_i_upper_pm_kappa_Lueck}
\end{align}
as follows: Consider an object $B$ in $\cala_\Phi[t,t\inv]^{\kappa}$. It is given by a
subset $J \subseteq \calu$ and a collection $(B_j)_{j \in J}$ of objects in
$\cala_\Phi[t^{\pm1}]$. Since $\cala_\Phi[t^{\pm1}]$ and $\cala$ have the same set of
objects, this is the same as a collection $(B_j)_{j \in J}$ of objects in $\cala$ indexed
by $J$. The image $i^0(B)$ is the object in $\cala^{\kappa}$ given by the set
$\tau(\IZ \times J) \subseteq \calu$ and the collection of objects in $\cala$ given
by $\bigl(\phi^{\tau_1^{-1}(k)} B_{\tau_2^{-1}(k)}\bigr)_{k \in \tau(\IZ \times  J)}$.  
Consider another object $B'$ in $\cala_\Phi[t,t\inv]^{\kappa}$ given by a subset
$J' \subseteq \calu$ and a collection $(B_{j'})_{j' \in J'}$ of objects in
$\cala_\Phi[t^{\pm1}]$. Let $f \colon B \to B'$ be a morphisms in
$\cala_\Phi[t,t\inv]^{\kappa}$ which is given by a collection $\bigl(f_{j,j'}
\colon B_j \to B'_{j'}\bigr)_{(j,j') \in J \times J'}$ of morphisms in
$\cala_\Phi[t^{\pm1}]$ such that for every $j \in J$ the set $\{j' \in J' \mid f_{j,j'}
\not= 0\}$ is finite.  Each $f_{j,j'} $ is given by a finite formal sum 
$\sum_{k[j,j'] \in \IZ} f_{j,j',k[j,j']} \cdot t^{k[j,j',k]}$, where $f_{j,j',k[j,j']} \colon \Phi^{k[j,j']}(B_j) \to B'_{j'}$ 
is a morphism in $\cala$.  Then $j^0(f) \colon j^0(B) \to j^0(B')$
is given by the collection of morphisms $\bigl(j^0(f)_{k,k'} \bigr)_{(k,k') \in
  \tau(\IZ \times J) \times\tau(\IZ \times J')}$ in $\cala$, where
$j^0(f)_{k,k'}$ is the morphism $\Phi^{ \tau_1^{-1}(k')} f_{\tau_2^{-1}(k),\tau_2^{-1}(k'),\tau_1^{-1}(k) -
  \tau_1^{-1}(k')} \colon \Phi^{\tau_1^{-1}(k)} B_{\tau_2^{-1}(k)} \to
\Phi^{\tau_1^{-1}(k')} B'_{\tau_2^{-1}(k')}$.  We have to check that for each  $k  \in \tau(\IZ \times J)$
the set $\{k' \in \tau(\IZ \times J') \mid j^0(f)_{k,k'} \not= 0\}$ is finite.
This follows from the fact that the set 
$\{j'\in  J' \mid   f_{\tau_2^{-1}(k),j'} \not= 0\}$ and hence the set 
$\{(l,j') \in \IZ \times J' \mid   f_{\tau_2^{-1}(k),j',l} \not= 0\}$
are finite. One easily checks that $i^0$ respects composition, the abelian group 
structure on the set of morphisms and is compatible with the preferred $\kappa$-direct sums.

Here is a second description of $i^0$.  Define a functor of
$\IZ$-categories $\widehat{i^0} \colon \cala_{\Phi}[t,t^{-1}] \to \cala^{\kappa}$ by
sending an object $A$ in $\cala_{\Phi}[t,t^{-1}]$, which is just an object in $\cala$, to
the object in $\cala^{\kappa}$ given by the subset $\IZ \subseteq \calu$ and the
collection of objects $(\Phi^k (B)_{k \in \IZ}$. This is the preferred direct sum
$\bigoplus_{k \in \IZ} \Phi^k(B)$, if we denote by abuse of notation the object in
$\cala^{\kappa}$ given the set $\{k\}$ and the collection of objects indexed by $\{k\}
\subseteq \calu$ whose only member is $\Phi^k(B)$, just by $\Phi^k(B)$ again.  A morphism
in $\cala_{\Phi}[t,t^{-1}]$ of the shape $f \cdot t^0 \colon A\to B$ for a morphism $f
\colon A \to B$ in $\cala$ is sent to
\[
\bigoplus_{k } \Phi^{-k}(f) \colon \bigoplus_{k = -\infty}^{\infty} \Phi^{-k}(A) 
\to \bigoplus_{k = -\infty}^{\infty} \Phi^{-k}(A).
\]
A morphism in $\cala_{\Phi}[t,t^{-1}]$ of the shape $\id_{A} \cdot t \colon \Phi^{-1}(A)
\to A $ is sent to the shift automorphism
\[
\sh \colon  \bigoplus_{k = -\infty}^{\infty} \Phi^{-k}(\Phi^{-1}(A))
\to \bigoplus_{k = -\infty}^{\infty} \Phi^{-k}(A)
\]
which sends the $k$-th summand $\Phi^{-k}(\Phi^{-1}(A)) = \Phi^{-(k+1)}(A)$ of the source
identically to the $(k+1)$-summand of the target.  Since any morphism in
$\cala_{\Phi}[t,t^{-1}]$ is a finite sum of composites of such morphisms, this specifies
the desired functor $\widehat{i^0} \colon \cala_{\Phi}[t,t^{-1}] \to \cala^{\kappa}$.
Then $i^0$ is the composite
\[
\cala_{\Phi}[t,t^{-1}]^{\kappa} \xrightarrow{\widehat{i^0}^{\kappa}} (\cala^{\kappa})^{\kappa} 
\xrightarrow{\zeta^{\kappa}} \cala^{\kappa}.
\]
where the functor $\zeta^{\kappa}$ has been defined in~\eqref{zetakappa}.

The construction of $i^\pm$ is analogous and left to the reader.


\subsection{Adjunction between induction and restriction}
\label{subsec:Adjunction_between_induction_and_restriction}

\begin{lemma}\label{lem:Adjunction_between_induction_and_restriction}
The pairs $(i_0,i^0)$, $(i_+,i^+)$ and $(i_-,i^-)$ are adjoint pairs, i.e., 
for objects $A$ in $\cala^\kappa$, $B^\pm$ in $\cala_\Phi[t^{\pm1}]^{\kappa}$ and $B$ in $\cala_\Phi[t,t\inv]^\kappa$
there are isomorphisms, natural in $A$ and $B$ and compatible with the preferred $\kappa$-direct sum in the first variable
\begin{eqnarray*}
\mor_{\cala_{\Phi}[t,t^{-1}]^\kappa}(i_0A,B) 
& \xrightarrow{\cong} &
\mor_{\cala^\kappa}(A,i^0B);
\\
\mor_{\cala_{\Phi}[t]^\kappa}(i_+A,B) 
& \xrightarrow{\cong} &
\mor_{\cala^\kappa}(A,i^+B);
\\
\mor_{\cala≈_{\Phi}[t^{-1}]^\kappa}(i_-A,B) 
& \xrightarrow{\cong} &
\mor_{\cala^\kappa}(A,i^-B).
\end{eqnarray*}
\end{lemma}

\begin{proof}
  We only treat the first isomorphism, the proof for the other ones is analogous.  The
  object $A$ in $\cala^{\kappa}$ is given by a subset $I \subseteq \calu$ and a collection
  of objects $(A_i)_{i \in I}$ of $\cala$. The object $B$ in
  $\cala_\Phi[t,t\inv]^{\kappa}$ is given by a subset $J \subseteq \calu$ and a collection
  of objects $(B_j)_{j\in IJ}$ of $\cala$.  Then $i_0 A$ is the object in
  $\cala_\Phi[t,t\inv]^{\kappa}$ given again by a subset $I \subseteq \calu$ and a
  collection of objects $(A_i)_{i \in I}$ of $\cala$.  The object $i^0 B$ in
  $\cala^{\kappa}$ is given by the set $\tau(\IZ \times J)$ and the collection of objects
  $(\Phi^{\tau^{-1}_1{j'}}(B_{\tau^{-1}_2(j')})_{j' \in \tau(\IZ \times J)}$.

  A morphism $f \colon i_0 (A) \to B$ is given by a collection $\bigl(f_{i,j}
  \colon A_i \to B_j\bigr)_{(i,j) \in I \times J}$, where $f_{i,j} \colon A_i \to B_j$ is
  a morphism in $\cala_{\Phi}[t,t^{-1}]$ such that for every $i \in I$ the set $\{j \in J
  \mid f_{i,j} \not= 0\}$ is finite.  Each $f_{i,j}$ is a finite  sum 
  $\sum_{k[i,j]   \in \IZ} f_{i,j,k[i,j]} \cdot t^{k[i,j]}$, where $f_{j,j,k[i,j]} \colon \phi^{k[i,j}(A_j) \to  B_j$ 
  is a morphism in $\cala$.  So $f$ is given by a collection of morphisms $f_{i,j,k}
  \colon \phi^k(A_j) \to B_j$ in $\cala$ indexed by $(i,j,k) \in I \times J \times \IZ$
  which satisfies condition (C'): For each $i \in I$ the set $\{j \in J \mid \exists k \in
  \IZ\; \text{with} f_{i,j,k} \not= 0\}$ is finite and for each $(i,j) \in I \times J$ the
  set $\{k \in \IZ \mid f_{i,j,k} \not= 0\}$ is finite.

  A morphism $g \colon A \to j^0 B$ in $\cala^{\kappa}$ is given by a collection
  of morphisms $\bigl(g_{i,j'} \colon A_i \to
  \Phi^{\tau^{-1}(j')}(B_{\tau_2^{-1}(j')})\bigr)_{(i,j') \in I \times \tau(\IZ \times
    J)}$ such that for each $i \in I$ the set $\{j' \in \tau(\IZ \times J) \mid g_{i,j'}
  \not= 0\}$ is finite.  This is the same as a collection of morphisms $\bigl(g_{i,j,k}
  \colon A_i \to \Phi^k(B_j)\bigr)_{(i,j,k) \in I \times J \times \IZ}$ in $\cala$ which
  satisfies condition (C''): For each $i \in I$ the set $\{(j,k) \in J \times \IZ \mid
  g_{i,j,k} \not= 0\}$ is finite.

  Now we can define the desired isomorphism of abelian groups by sending a collection
  $\bigl(f_{i,j} \colon A_i \to B_j\bigr)_{(i,j) \in I \times J}$ to the same collection
  $\bigl(f_{i,j} \colon A_i \to B_j\bigr)_{(i,j) \in I \times J}$ since the conditions
  (C') and (C'') are equivalent.

  One easily checks that this isomorphism is natural in $A$ and $B$.
\end{proof}


\section{Strategy of proof for 
Theorem~\ref{the:BHS_decomposition_for_connective_K-theory}%
~\textnormal{\ref{the:BHS_decomposition_for_connective_K-theory:BHS-iso}}}
\label{sec:Strategy_of_proof_for_the_connective}

In this section we present the details of the formulation and then the basic strategy of
proof of Theorem~\ref{the:BHS_decomposition_for_connective_K-theory}%
~\ref{the:BHS_decomposition_for_connective_K-theory:BHS-iso}.

In the sequel $\bfK(\calc)$ denotes the connective $K$-theory spectrum of a \emph{Waldhausen category}
$\calc$, i.e.,  a category with cofibrations and weak equivalences $\calc$, in the sense of 
Waldhausen~\cite{Waldhausen(1985)}. 

\begin{remark}[Exact categories as Waldhausen categories]
\label{rem:Additive_categories_as_Waldhausen_categories}
Any additive (in fact, any exact) category has a canonical Waldhausen
 structure where the cofibrations are the admissible monomorphisms and the weak equivalences are the isomorphisms.

In the situation of Example~\ref{exa:Finitely_generated_free_R-modules} we get
that $\pi_n(\bfK(\calr)) = K_n(R)$ for $n \ge 1$, 
the map $\IZ \to K_0(\calr)$ sending $n$ to $[R^n]$ is surjective and even bijective if
$R^n \cong R^m$ implies $m = n$, and $\pi_n(\bfK(\calr)) = 0$ for
$n \le -1$.  If we pass to the idempotent completion $\Idem(\calr)$, then we obtain
$\pi_n(\bfK(\Idem(\calr)) = K_n(R)$ for $n \ge 0$, where $K_0(R)$ is the projective class group, 
and \change{$\pi_n(\bfK(\Idem(\calr)) = 0$} for $n \le -1$.
\end{remark}


\subsection{The NK-terms and the maps $\bfa$ and $\bfb$}
\label{subsec:The_NK-terms_and_the_maps_bfa_and_bfb}

\begin{definition}[{$\bfNK(\cala_{\Phi}[t])$} and
  {$\bfNK(\cala_{\Phi}[t^{-1}])$}]
  \label{def:NK(cala_Phi[t(pm1]}
  Define $\bfNK(\cala_{\Phi}[t^{\pm 1}])$ to be the homotopy fiber of the map of
  spectra $\bfK(\ev_0^{\pm}) \colon \bfK(\cala_{\Phi}[t^{\pm 1}]) \to \bfK(\cala)$.

  Let $\bfb^{\pm} \colon \bfNK(\cala_{\Phi}[t^{\pm 1}]) \to
  \bfK(\cala_{\Phi}[t^{\pm 1}])$ be the canonical map of spectra.
\end{definition}

Let $S \colon i_0 \circ \Phi^{-1} \to i_0$ be the natural transformation of functors
of additive categories $\cala \to \cala_{\Phi}[t,t^{-1}]$ which is given on an
object $A$ in $\cala$ by the isomorphism $\id_A \cdot t \colon \Phi^{-1}(A) \to A$. 
It induces a (preferred) homotopy 
\begin{eqnarray}
\bfK(S) \colon \bfK(\cala) \wedge I_+ & \to & \bfK(\cala_{\Phi}[t,t^{-1}])
\label{homotopy_bfK(S)}
\end{eqnarray}
from $\bfK(i_0) \circ \bfK(\Phi^{-1})$ to $\bfK(i_0)$.
Recall that the mapping torus of $\bfK(\Phi^{-1})$ is by definition the pushout
\[
\xymatrix{\bfK(\cala) \vee \bfK(\cala) = \bfK(\cala) \wedge \partial I_+ 
\ar[r]^-{\bfn} \ar[d]^{\bfK(\Phi^{-1}) \vee \id_{\bfK(\cala)}}
& \bfK(\cala)\wedge I_+ \ar[d]
\\
\bfK(\cala) \ar[r]
& \bfT_{\bfK(\Phi^{-1})}
}
\]
where the upper horizontal map $\bfn$ is given by the inclusion $\partial I \to I$.
Hence $S$ yields a map of spectra
\[
\bfa \colon \bfT_{\bfK(\Phi^{-1})} \to \bfK(\cala_{\Phi}[t,t^{-1}]).
\]

Thus we have explained all terms appearing Theorem~\ref{the:BHS_decomposition_for_connective_K-theory}%
~\ref{the:BHS_decomposition_for_connective_K-theory:BHS-iso}.
Next  we explain the strategy of its proof.


\subsection{The twisted projective line}
\label{subsec:The_twisted_projective_line}

We define the \emph{twisted projective line} to be the following 
additive category $\calx= \calx(\cala,\Phi)$. Objects are
triples $(A^+,f,A^-)$ consisting of objects $A^+$ in $\cala_\Phi[t]$ and $A^-$ in $\cala_\Phi[t\inv]$ and an
isomorphism $f \colon j_+A^+ \to j_-A^-$ in $\cala_{\Phi}[t,t^{-1}]$. A morphism 
$(u^+,u^-) \colon (A^+,f,A^-) \to (B^+,g,B^-)$ in $\calx$ consists of morphisms 
$u^+ \colon A^+ \to B^+$ in $\cala_{\Phi}[t]$ and a morphism $u^- \colon A^- \to B^-$ in
$\cala_{\Phi}[t^{-1}]$ such that the following diagram commutes in $\cala_{\Phi}[t,t^{-1}]$
\[
\xymatrix{j_+A^+ \ar[r]^{f} \ar[d]_{u^+} 
& j_-A^- \ar[d]^{u^-}
\\
j_+B^+ \ar[r]_{g} 
& j_-B^-
}
\]
Let 
\begin{eqnarray}
k^{\pm} \colon \calx & \to & \cala_{\Phi}[t^{\pm 1}]
\label{def_f_kpm}
\end{eqnarray}
be the functor sending $(A^+,f,A^-)$ to $A^{\pm}$.

The category $\calx$ is naturally an exact category by declaring a sequence to be 
exact if and only if becomes (split) exact both after applying $k^+$ and $k^-$.

The proof of the next result is deferred to 
Section~\ref{sec:Proof_of_Theorem_ref(the:homotopy_pull_back_of_calx)}.

\begin{theorem} \label{the:homotopy_pull_back_of_calx}
Consider the following (not necessarily commutative) diagram of spectra
\[
\xymatrix@!C=8em{\bfK(\calx) \ar[r]^-{\bfK(k^-)} 
\ar[d]_{\bf\bfK(k^+)} 
& \bfK(\cala_{\Phi}[t^{-1}])\ar[d]^{\bfK(j_-)}
\\
\bfK(\cala_{\Phi}[t]) \ar[r]_-{\bfK(j_+)}
&
\bfK(\cala_{\Phi}[t,t^{-1}]) 
}
\]
There is a natural equivalence of functors
$T \colon j_+ \circ k^+ \xrightarrow{\cong} j_- \circ k^-$ which is given on an object $(A^+,f,A^-)$ by $f$.
It induces a preferred homotopy 
$\bfK(j_+) \circ \bfK(k^+) \simeq \bfK(j_-)  \circ \bfK(k^-)$.

If $\cala$ is idempotent complete, then the diagram above is a weak homotopy pullback, i.e.,
the canonical map from $\bfK(\calx)$ to the homotopy pullback of
\[\bfK(\cala_{\Phi}[t]) \xrightarrow{\bfK(j_+)} \bfK(\cala_{\Phi}[t,t^{-1}])  \xleftarrow{\bfK(j_-)}  \bfK(\cala_{\Phi}[t^{-1}])
\]
is a weak homotopy equivalence.
\end{theorem}

Let
\begin{eqnarray}
l_i \colon \cala & \to & \calx \quad \text{for} \; i = 0,1
\label{def_of_l_i}
\end{eqnarray}
be the functor which sends an object $A$ to $(A,\id,A)$ for $i = 0$ and to the object  $(\Phi^{-1}(A),
\id_A \cdot t, A)$ for $i = 1$, and a morphism $f \colon A \to B$ in $\cala$ to
the morphism $\bigl(i_+(f),i_-(f)\bigr)$ for $i = 0$ and
$\bigl(i_+(\Phi^{-1}(f)),i_-(f)\bigr)$ for $i = 1$.

The proof of the next result is deferred to  Section~\ref{sec:Proof_of_Theorem_ref(computing_K(calx)}

\begin{theorem} \label{the:computing_K(calx)}
Suppose that $\cala$ is idempotent complete.  Then the map of spectra
\[
\bfK(l_0) \vee \bfK(l_1) \colon \bfK(\cala) \vee \bfK(\cala)  \xrightarrow{\simeq} \bfK(\calx).
\]
is a  weak homotopy equivalence.
\end{theorem}


\subsection{Proof of
Theorem~\ref{the:BHS_decomposition_for_connective_K-theory}%
~\ref{the:BHS_decomposition_for_connective_K-theory:BHS-iso}}
\label{subsec:Proof_of_Theorem_ref(the:BHS_decomposition_for_connective_K-theory)_(i)}
In this subsection we finish the proof of 
Theorem~\ref{the:BHS_decomposition_for_connective_K-theory}%
~\ref{the:BHS_decomposition_for_connective_K-theory:BHS-iso}
assuming that Theorem~\ref{the:homotopy_pull_back_of_calx} 
and Theorem~\ref{the:computing_K(calx)} are true.

There is a not necessarily commutative diagram
\begin{eqnarray}
& 
\xymatrix@!C=12em{\bfK(\cala) \vee \bfK(\cala)  
\ar[r]^-{\bfK(i_-) \vee \bfK(i_-)} 
\ar[d]_{\bfK(i_+ \circ \Phi^{-1})\vee \bfK(i_+)} 
& \bfK(\cala_{\Phi}[t^{-1}]) \ar[d]^{\bfK(j_-)}
\\
\bfK(\cala_{\Phi}[t]) \ar[r]_-{\bfK(j_+)}
&
\bfK(\cala_{\Phi}[t,t^{-1}])
}
&
\label{ho-pullback_auxiliary}
\end{eqnarray}
The homotopy 
$\bfK(S) \colon \bfK(\cala) \wedge I_+ \to \bfK(\cala_{\Phi}[t,t^{-1}])$ of~\eqref{homotopy_bfK(S)}
induces a preferred homotopy 
$\bfK(j_+) \circ \bigl(\bfK((i_+ \circ \Phi\inv)\vee \bfK( i_+)\bigr)  
\simeq \bfK(j_-) \circ  \bigl(\bfK(i_-) \vee \bfK(i_-)\bigr)$.

\begin{theorem} \label{the:new_homotopy_pushout_above_degree_ge_1}
Suppose that $\cala$ is idempotent complete. 
With respect to this choice of homotopy, the 
diagram~\eqref{ho-pullback_auxiliary} is a weak homotopy pushout,
i.e., the canonical map from the homotopy pushout of
\[
\bfK(\cala_{\Phi}[t])  \xleftarrow{\bfK(i_+ \circ \Phi^{-1})\vee \bfK(i_+)}  
\bfK(\cala) \vee \bfK(\cala)  
\xrightarrow{\bfK(i_-) \vee \bfK(i_-)} \bfK(\cala_{\Phi}[t^{-1}]) 
\]
to  $\bfK(\cala_{\Phi}[t,t^{-1}])$ is a weak homotopy equivalence.
\end{theorem}
\begin{proof} Combining Theorem~\ref{the:homotopy_pull_back_of_calx} 
and Theorem~\ref{the:computing_K(calx)} shows that the diagram of 
spectra~\eqref{ho-pullback_auxiliary} is a weak homotopy pullback.
This implies that~\eqref{ho-pullback_auxiliary} is a weak homotopy pushout.
The latter claim follows for commutative squares of spectra from~\cite[Lemma~2.6]{Lueck-Reich-Varisco(2003)}
and then follows easily for squares commuting up to a preferred homotopy.
\end{proof}

Consider the following commutative diagram
{\small \begin{eqnarray}
\label{comm_diagram}
& & 
\\
\xymatrix@!C= 11em{
\bfK(\cala) \vee \bfNK(\cala_{\Phi}[t]) \ar[d]_{\bfK(i_+) \vee \bfb_+}
& 
\bfK(\cala) \vee \bfK(\cala) \ar[d]_{\id} 
\ar[l]_-{\bfm_1 \circ (\bfK(\Phi^{-1}) \vee \id)}
\ar[r]^{\bfm_1\circ (\id \vee \id)}
& 
\bfK(\cala)  \vee \bfNK(\cala_{\Phi}[t])  \ar[d]_{\bfK(i^-) \vee \bfb_-}  
\\
\bfK(\cala_{\Phi}[t]) 
&
\bfK(\cala) \vee \bfK(\cala) \ar[l]^-{\bfK((i_+ \circ \Phi^{-1})  \vee \bfK(i_+)} \ar[r]_{\bfK(i_-) \vee \bfK(i_-)} 
&
\bfK(\cala_{\Phi}[t^{-1}])
}
& & 
\nonumber
\end{eqnarray}}
where $\bfm_1$ here and in the sequel denotes the 
inclusion of the first summand.  Let $\bfE_t$ and $\bfE_b$
respectively be the homotopy pushout of the top and of the bottom row of the
diagram~\eqref{comm_diagram} respectively.  One easily checks using the fact
that the composite 
$\bfK(\cala) \xrightarrow{\bfK(i^{\pm})} \bfK(\cala_{\Phi}[t^{\pm 1}]) \xrightarrow{\bfK(\ev_0^{\pm})} \bfK(\cala)$ 
is the identity that all vertical arrows in the diagram~\eqref{comm_diagram} are weak
equivalences. Hence the diagram~\eqref{comm_diagram}  induces a weak homotopy equivalence 
$\bfe   \colon \bfE_t \to \bfE_b$. 

Let $\bff \colon \bfE_b \to \bfK(\cala_{\Phi}[t,t^{-1}])$ be homotopy equivalence coming
from~\eqref{ho-pullback_auxiliary} and
Theorem~\ref{the:new_homotopy_pushout_above_degree_ge_1}.

Next we construct a weak homotopy equivalence 
\[
\bfg \colon \bfE_t \to \bfT_{\bfK(\Phi^{-1})} \vee \bfNK(\cala_{\Phi}[t]) \vee \bfNK(\cala_{\Phi}[t^{-1}]).
\]
Consider the following not necessarily commutative diagram
{\small\[
\xymatrix@!C= 3.5em{
\bfK(\cala) \vee \bfNK(\cala_{\Phi}[t]) \ar[d]_-{\id \vee \id}
&&&
\bfK(\cala) \vee \bfK(\cala) \ar[d]_-{\id} 
\ar[lll]_-{\bfm_1 \circ (\bfK(\Phi^{-1}) \vee \id)}
\ar[rr]^(.44){\bfm_1\circ (\id \vee \id)}
& &
\bfK(\cala)  \vee \bfNK(\cala_{\Phi}[t])  \ar[d]^-{\bfn_0 \vee \id} 
\\
\bfK(\cala) \vee \bfNK(\cala_{\Phi}[t^{-1}]) 
&&&
{\begin{array}{c}
\bfK(\cala) \vee \bfK(\cala) 
\\
= 
\\
\bfK(\cala) \wedge \partial I_+
\end{array}}
\ar[lll]_-{\bfm_1 \circ (\bfK(\Phi^{-1}) \vee \bfK(\id_{\cala}))}
\ar[rr]^{\bfn}
&&
{\begin{array}{ccc}
\bfK(\cala) \wedge I_+  
\\
\vee 
\\
\bfNK(\cala_{\Phi}[t^{-1}])
\end{array}}
}
\]}
where $\bfn_0$ comes from the inclusion $\{0\} \to I$, and $\bfn$ comes from the inclusion
$\partial I \to I$.  The left square commutes. The right square commutes up to a preferred
homotopy coming from the standard homotopy from the inclusion $\partial I \to I$ to the
constant map $\partial I \to I$ with value $0$.  Since the pushout of the lower row is
$\bfT_{\bfK(\Phi^{-1})} \vee \bfNK(\cala_{\Phi}[t]) \vee \bfNK(\cala_{\Phi}[t^{-1}])$, we
obtain a map $\bfg \colon \bfE_t \to \bfT_{\bfK(\Phi^{-1})} \vee \bfNK(\cala_{\Phi}[t]) \vee \bfNK(\cala_{\Phi}[t^{-1}])$.  
Since the horizontal right
arrow in the diagram above is a cofibration and all vertical arrows are weak homotopy
equivalences, the map $\bfg$ is a weak homotopy
equivalence. One easily checks that it fits into the following commutative diagram
\[
\xymatrix@!C= 14em{\bfE_t \ar[r]^-{\bfg}_-{\simeq} \ar[d]_{\bfe}^{\simeq}
& 
\bfT_{\bfK(\Phi^{-1})}  \vee \bfNK(\cala_{\Phi}[t]) \vee \bfNK(\cala_{\Phi}[t^{-1}])
\ar[d]^{\bfa \vee \bfb_+ \vee \bfb_-}
\\
\bfE_b \ar[r]^-{\bff}_-{\simeq}
&
\bfK(\cala_{\Phi}[t,t^{-1}])
}
\]
This finishes the proof of 
Theorem~\ref{the:BHS_decomposition_for_connective_K-theory}%
~\ref{the:BHS_decomposition_for_connective_K-theory:BHS-iso}, i.e., that
the right vertical arrow in the diagram above is a weak homotopy equivalence, provided
that Theorem~\ref{the:homotopy_pull_back_of_calx} and Theorem~\ref{the:computing_K(calx)} are true.


\section{Preliminaries about chain complexes}
\label{sec:Preliminaries_about_chain_complexes}

Consider an additive category $\cala$. The notions of chain complexes over $\cala$, chain
maps, chain homotopies, chain contractions of chain complexes are defined in the same way
as in the category of $R$-modules. A short exact sequence of chain complexes in $\cala$ is
a sequence which is level-wise split exact.

We write all chain complexes homologically. If $C$ is a chain complex in $\cala$, 
we denote its $n$-th object by $C_n$ and its $n$-differential by $c_n\colon C_n\to C_{n-1}$.


\subsection{Mapping cylinders and mapping cones}
\label{subsec:mapping_cylinders_and_mapping_cones}

Let $f \colon C \rightarrow D$ be a chain map. Define
its mapping cylinder $\cyl(f)$ to be the chain complex with $n$-th
differential
\[
C_{n-1} \oplus C_n  \oplus D_n
\xrightarrow{\left(\begin{array}{ccc} -c_{n-1} & 0 & 0 \\ - \id & c_n & 0
\\ f_{n-1} & 0 & d_n \end{array}\right)}
C_{n-2} \oplus C_{n-1} \oplus D_{n-1}.
\]
There are obvious inclusions $i_C \colon C \to \cyl(f)$ and 
$i_D \colon D \to \cyl(f)$ and an obvious projection $p_D \colon \cyl(f) \to D$
such that $p_D \circ i_C = f$, $p_D \circ i_D = \id_D$ and both $p_D$ and
$i_D$ are chain homotopy equivalences.  Define the mapping cone
$\cone(f)$ of $f$ to be the cokernel of $i_C \colon C \to \cyl(f)$.
Hence the $n$-th differential of $\cone(f)$ is
\[
C_{n-1} \oplus D_n
\xrightarrow{\left(\begin{array}{cc} -c_{n-1} & 0 \\
f_{n-1} &  d_n \end{array}\right)}
 C_{n-2}  \oplus D_{n-1}.
\]
We write $\cone(C) := \cone(\id_C)$.  Given a chain complex $C$,
define its suspension $\Sigma C$ to be the cokernel of the obvious
embedding $C\to \cone(C)$, i.e., to be the chain complex with $n$-th
differential
\[
C_{n-1} \xrightarrow{-c_{n-1}} C_{n-2}.
\]

We will call a chain complex  \emph{elementary}
if it is the finite direct sum of chain complexes $\el(X,d)$ for objects $X$ and 
integers $d$, where $\el(X,d)$ is concentrated in dimension $d$ and $d+1$ 
and has as $(d+1)$-th differential $\id_X \colon X \to X$.
Notice that elementary chain complexes are contractible.

We call a chain complex $C$ \emph{concentrated in degrees $[a,b]$} if $C_n = 0$ for $n < a$ 
and for $n > b$. The minimal possible nonnegative number $b-a$ is the \emph{length} of $C$. We call $C$
\emph{bounded} if there are natural numbers $a, b$ such that $C$ is concentrated in
degrees $[a,b]$. For an object $A$ of $\cala$ we denote by $A[n]$ the chain complex
concentrated in degrees $[n,n]$ whose single object is $A$.

We collect the following elementary statements about chain complexes.

\begin{lemma} \label{lem:elementary_facts}
Let $f \colon C \to D$ be a chain map and $E$ be a chain complex.

\begin{enumerate}

\item \label{lem:elementary_facts:short_exact_sequences}

There are obvious short exact sequences of chain complexes
\[
\begin{array}{lclclclcl}
0 & \to & C &\xrightarrow{i(C)} & \cyl(f) &\to &\cone(f) & \to & 0; 
\\
0 & \to & D &\xrightarrow{i(D)}  & \cyl(f) & \to &\cone(C) & \to & 0;
\\
0 & \to & D & \to & \cone(f) & \to & \Sigma C & \to & 0;
\end{array}
\]

\item \label{lem:elementary_facts:projection_cyl_to_D}
The natural projection $\pr(D) \colon \cyl(f) \to D$ is the chain map given by
$\pr(D)_n = (0, f_n, id_{D_n}) \colon C_{n-1} \oplus C_n \oplus D_n \to  D_n$.
Then $\pr(D) \circ i(D) = \id_D$ and there is a chain homotopy 
$h(D) \colon \id_{\cyl(f)} \simeq i(D) \circ \pr(D)$ given by
\[
h(D)_n = \begin{pmatrix} 0 & \id_{C_n}  & 0 \\  0 & 0 & 0 \\ 0 & 0 & 0 \end{pmatrix} 
\colon C_{n-1} \oplus C_n \oplus D_n \to C_{n} \oplus C_{n+1} \oplus D_{n+1};
\]

\item \label{lem:elementary_facts:maps_between_mapping_cones} 
Consider the following (not necessarily commutative) diagram of chain complexes
\xycomsquare{C}{f}{D}{u}{v}{C'}{f'}{D'}
Consider a  chain homotopy $h \colon v \circ f' \simeq f' \circ u$.

Then we obtain a chain map $g \colon \cone(f) \to \cone(f')$ by
\[
g_n = \begin{pmatrix} u_{n-1} & 0 \\ h_{n-1} & v_n \end{pmatrix} \colon C_{n-1} \oplus D_n \to C'_{n-1} \oplus D'_n.
\]

Conversely, a chain map $g \colon \cone(f) \to \cone(f')$ given by
\[
g_n = \begin{pmatrix} u_{n-1} & w_n \\ h_{n-1} & v_n \end{pmatrix} \colon C_{n-1} \oplus D_n \to C'_{n-1} \oplus D'_n
\]
yields such a  diagram and homotopy;

\item \label{lem:elementary_facts:maps_between_mapping_cylinders}
Let $f \colon C \to D$, $u \colon C \to E$, and $v \colon D \to E$ 
be chain maps and let $h \colon v \circ f \simeq  u$ be
  a chain homotopy.  Then we obtain a chain map $F \colon \cyl(f) \to E$ by
\[
F_n:=  (h_{n-1},u_n,v_n)  \colon C_{n-1} \oplus C_n \oplus D_n \to E_n
\]
such that the composite of $F$ with the canonical inclusions of $C$ and $D$ into $\cyl(f)$
are $u$ and $v$.

The converse is also true, i.e., a chain map $F$ yields chain maps $u$, $v$ and a chain homotopy 
$h \colon v \circ f \simeq   u$;

\item \label{lem:elementary_facts:chain_homotopy_equivalence_cone} 
A chain map is a   chain homotopy equivalence if and only if its mapping cone 
is contractible;

\item \label{lem:elementary_facts:exact_sequence_and_contraction} Let
  $0 \to C \xrightarrow{i} D \xrightarrow{p}  E \to 0$ be an exact
  sequence of chain complexes. Suppose that $E$ is contractible. Then
  there exists a chain map $s \colon E \to D$ with $p \circ s =
  \id_C$.  In particular we get a chain isomorphism
  \[
  i \oplus s \colon C \oplus E \xrightarrow{\cong} D;
  \]

\item \label{lem:elementary_facts:chain_homotopy_equivalence_2_of_3}
Consider the following commutative diagram of  chain complexes 
\[
\xymatrix{0 \ar[r] & C \ar[r] \ar[d]_{f} & D \ar[r] \ar[d]_{g} & E \ar[r] \ar[d]_{h} & 0
\\
0 \ar[r] & C' \ar[r] & D' \ar[r] & E' \ar[r] & 0
}
\]
with exact rows. If two of the chain maps $f$, $g$ and $h$ are chain homotopy equivalences,
then all three are;

\item \label{lem:elementary_facts:induction_step_argument} Let $C$ be a chain complex
  concentrated in degrees $[a,b]$ (where $a<b+1$) such that the last differential
  $c_{a+1}$ is split surjective. Then, for any split $\gamma$ of $c_{a+1}$ there is a
  short exact sequence
 \[
  0 \to \el(C_a, a) \xrightarrow{i} C\oplus \el(C_a, a+1) \xrightarrow{p} D \to 0
  \]
 with a chain complex $D$ concentrated in degrees $[a,b+1]$. It is uniquely split and natural in $(C,\gamma)$.

\item \label{lem:elementary_facts:elementary}
Let $f\colon C\to D$ be a map of bounded chain complexes in an additive category
$\cala$.  Then the following statements are equivalent:
\begin{enumerate}
\item \label{lem:elementary_facts:elementary:(a)}
  $f$ is a chain homotopy equivalence;
\item \label{lem:elementary_facts:elementary:(b)}
 There are elementary chain complexes $E$, $E'$ in $\cala$ and a
commutative diagram
\[
\xymatrix{C \ar[d]^f \ar[r] & C\oplus E \ar[d]_\cong
\\
D & D\oplus E' \ar[l]
}
\]
where the horizontal maps are the canonical inclusion and projection and
the right vertical arrow is a chain isomorphism.
\end{enumerate}

\end{enumerate}
\end{lemma}
\begin{proof}~\ref{lem:elementary_facts:short_exact_sequences} This is obvious.
  \\[1mm]~\ref{lem:elementary_facts:projection_cyl_to_D} This follows from a direct calculation.
  \\[1mm]~\ref{lem:elementary_facts:maps_between_mapping_cones} This
  is obvious.
   \\[1mm]~\ref{lem:elementary_facts:maps_between_mapping_cylinders}  This is 
   obvious.
  \\[1mm]~\ref{lem:elementary_facts:chain_homotopy_equivalence_cone}
  See for instance~\cite[Lemma~11.5~a) on page~214]{Lueck(1989)}.
  \\[1mm]~\ref{lem:elementary_facts:exact_sequence_and_contraction}
  For each $n$ there exists a morphism $t_n \colon E_n \to D_n$
  with $p_n \circ t_n = \id_{D_n}$.  Let $\gamma$ be a chain
  contraction for $E$. Define $s_n \colon E_n \to D_n$ by 
  $d_{n+1} \circ  t_{n+1} \circ \gamma_n  + t_n \circ \gamma_{n-1} \circ e_n$.  Then the
  collection $s = (s_n)$ is a chain map $s \colon E \to D$ with $p \circ s = \id_E$.  
   \\[1mm]~\ref{lem:elementary_facts:chain_homotopy_equivalence_2_of_3} 
  The commutative diagram induces a short exact
  sequence of chain complexes $0 \to \cone(f) \to \cone(g) \to \cone(h) \to 0$. Because
  of
  assertion~\ref{lem:elementary_facts:chain_homotopy_equivalence_cone}
  it remains to show for any short exact sequence $0 \to C
  \xrightarrow{i} D \xrightarrow{p} E \to 0$ that all three chain
  complexes are contractible if two of them are.

   If $C$ and  $E$ are contractible, then $D$ is contractible by 
   assertion~\ref{lem:elementary_facts:exact_sequence_and_contraction}.
   In the sequel we will use that we have already taken care of this case.
  
   Now suppose that $C$ and $D$ are known to be contractible. Because
   of the short exact sequence $0 \to \Sigma C \to \cone(p) \to
   \cone(E) \to 0$ and the conclusion from
   assertion~\ref{lem:elementary_facts:chain_homotopy_equivalence_cone}
   that $\cone(E)$ is contractible, we see that $\cone(p)$ is
   contractible.  Because of the short exact sequence $0 \to D \to
   \cyl(p) \to \cone(p) \to 0$, the mapping cylinder $\cyl(p)$ is
   contractible.  Since $E$ is chain homotopy equivalent to $\cyl(p)$,
   we conclude that $E$ is contractible.

   If $D$ and $E$ are contractible, we get from
   assertion~\ref{lem:elementary_facts:exact_sequence_and_contraction} a short exact
   sequence $0 \to E \to D \to C \to 0$ and conclude from the previous case that $C$ is
   contractible. 
   \\[1mm]~\ref{lem:elementary_facts:induction_step_argument}
   Again we assume that $C$ is concentrated in degrees $[0,d]$. The splitting of the last
   differential induces a chain map $\Gamma\colon \el(C_0,0)\to C$.

The commutative diagram
\[
\xymatrix{C_0[0] \ar@{>->}[rr] \ar[d]^\id && \el(C_0,0) \ar[d]^\Gamma \\
 C_0[0]   \ar@{>->}[rr] && C
}\]
induces a map
\[
i\colon \el(C_0,0)\to \cone(\Gamma)
\] 
on the vertical cones. Here the symbol
``$\rightarrowtail$'' denotes the inclusion into a direct summand. It follows that $i$ is
also the inclusion into a direct summand, so it extends to a short exact sequence
\[
0\to \el(C_0,0)\xrightarrow{i} \cone(\Gamma) \xrightarrow{p} D \to 0
\] 
in $\cala$. But the 0-th object of $\cone(\Gamma)$ is just $C_0$, so $D$ concentrated in degrees
$[1,d]$. Moreover the map $i$ is (uniquely) split on the 0-th level; as the domain of $i$
is elementary, it follows $i$ has a (unique) splitting.

Finally, as $\el(C_0,0)$ is canonically contractible, the map $\Gamma$ is canonically
null-homotopic. It follows that
\[
\cone(\Gamma)\cong \cone\bigl(0\colon \el(C_0,0)\to C\bigr) \cong \el(C_0,1)\oplus C.
\]
(An explicit isomorphism is given by
\[
\begin{pmatrix}-1 & 0\\  \gamma & 1 \end{pmatrix}\colon C_1\oplus C_0\to C_1\oplus C_0
\]
in degree 1 and by the identity in all other degrees.)
\\[1mm]~\ref{lem:elementary_facts:elementary} The
   implication~$(b)\implies (a)$ is obvious, it remains to prove the implication~$(a)
   \implies (b)$. We have the exact sequences $0 \to C \to \cyl(f) \to \cone(f) \to 0$ and
   $0 \to D \to \cyl(f) \to \cone(C) \to 0$.  The chain complexes $\cone(f)$ and
   $\cone(C)$ are contractible by
   assertion~\ref{lem:elementary_facts:chain_homotopy_equivalence_cone}. Because of
   assertion~\ref{lem:elementary_facts:exact_sequence_and_contraction} it suffices to show
   for a bounded contractible chain complex $C$ that there are elementary chain complexes
   $X$ and $X'$ together with chain isomorphisms $C \oplus X' \xrightarrow{\cong} X$.  We
   use induction over the length of $C$. The induction beginning $d = 1$ is obvious since
   then $C$ looks like $\cdots \to 0 \to C_{n+1} \xrightarrow{c_{n+1}} C_n \to 0 \to
   \cdots$ and $c_{n+1}$ is an isomorphism.  The induction step from $(d-1)$ to $d \ge 2$
   is done as follows.

   We assume for simplicity that $C$ is concentrated in degrees $[0,d]$. Choose a chain
   contraction $\gamma$ for $C$. Now by
   part~\ref{lem:elementary_facts:induction_step_argument}, there is an isomorphism
   \[
    \el(C_0,0)\oplus D \cong C\oplus \el(C_0, 1)
    \]
   where $D$ is concentrated in degrees $[1, d]$. Since the induction hypothesis applies
   to $D$, the claim follows.

This completes the proof of Lemma~\ref{lem:elementary_facts}.
\end{proof}
 

\subsection{Homotopy fiber sequences}
\label{subsec:Homotopy_fiber_sequences}

 A sequence $A\xrightarrow{f} B\xrightarrow{g} C$ of chain complexes together with a
 null-homotopy $g\circ f\simeq 0$ is called a \emph{homotopy fiber sequence} if
 the induced map $\cone(f)\to C$ (see
 Lemma~\ref{lem:elementary_facts}~\ref{lem:elementary_facts:maps_between_mapping_cones})
 is a chain homotopy equivalence. In particular any short exact
 sequence of chain complexes $0\to C\xrightarrow{i} D\to E \to 0$ is a homotopy fiber sequence since it induces a
 short exact sequence $0 \to \cone(C) \to \cone(i) \to E \to 0$ and we
 can apply
 Lemma~\ref{lem:elementary_facts}~\ref{lem:elementary_facts:chain_homotopy_equivalence_2_of_3}.
 
Let 
\xycomsquare{A}{f}{B}{k}{l}{C}{g}{D}
be a square in $\Chcat(\cala)$ which commutes up to a homotopy $h\colon g\circ k\simeq l\circ f$. 
We call this square \emph{homotopy cartesian} if one of the following equivalent conditions \change{holds:}
\begin{enumerate}
\item The induced map $\cone(f)\to \cone(g)$ is a homotopy equivalence.
\item The induced map $\cone(k)\to\cone(l)$ is a homotopy equivalence.
\item The induced sequence
\[
A\xrightarrow{(f,k)}  B\oplus C \xrightarrow{g-l} D
\]
together with the null-homotopy induced by $h$ is a fiber sequence.
\end{enumerate}

We conclude from
Lemma~\ref{lem:elementary_facts}~\ref{lem:elementary_facts:chain_homotopy_equivalence_cone}
and the fact that the mapping cones of the maps $\cone(f) \to \cone(g)$,
$\cone(k) \to \cone(l)$ and $\cone(f,k) \to D$ are isomorphic
that these three conditions above are indeed equivalent.


\subsection{Detecting contractibility by restriction}
\label{subsec:Detecting_contractibility_by_restriction}

If $S$ is a subring of $R$ and $C$ is a bounded $R$-chain complex, such that each
$R$-module $C_n$ is of the shape $R \otimes_S C_n'$ for some $S$-module $C_n'$ and $C$
considered as $S$-chain complex is contractible, then $C$ is contractible as $R$-chain
complex.  We will later need the following version of this fact for $\cala \subseteq
\cala_{\Phi}[t,t^{-1}]$. (The proof of the fact for rings follows the same lines but will
not be needed in this paper.)

\begin{lemma} \label{lem:contractible_over_cala_Phi[t,t(-1)]_versus_over_cala}
  Let $f \colon C \to D$ be an  $\cala_{\Phi}[t]^\kappa$-chain map of
  bounded $\cala_{\Phi}[t]^\kappa$-chain complexes. Then $f$ is an
  $\cala_{\Phi}[t]^\kappa$-chain homotopy equivalence if and only if
  its restriction $i^+ f \colon i^+ C \to i^+ D$ is an
  $\cala^{\kappa}$-chain homotopy equivalence.
\end{lemma}

The proof of this Lemma builds on the following result of category theory:

\begin{lemma}\label{lem:abstract_retraction_from_adjunction}
  Let
  \[
   i_+\colon \cala\leftrightarrows \calb\colon i^+
   \] 
   be an adjunction between categories,
  such that the right adjoint $i^+$ is faithful. Then, for any two objects $A$ of $\cala$
  and $B$ of $\calb$, the injection
  \[
  i^+\colon \calb(i_+A,B)\to\cala(i^+i_+A, i^+B)
   \] 
   has a splitting $r$ which is natural
  in $A$ and $B$.

  If $\cala$ and $\calb$ are additive and $i^+$ and $i_+$ are additive, then so is the
  splitting.
\end{lemma}

\begin{proof}[Proof of Lemma~\ref{lem:abstract_retraction_from_adjunction}]
  Denote by $\eta_A\colon A\to i^+i_+A$ and $\varepsilon_B\colon i_+i^+B\to B$ the unit
  and the co-unit of the adjunction. The retraction sends a morphism $f\colon i^+i_+A\to
  i^+B$ to the composite
  \[
  r(f)\colon i_+A\xrightarrow{i_+\eta_A} i_+i^+i_+A \xrightarrow{i_+ f} i_+ i^+ B
  \xrightarrow{\varepsilon_B} B.
  \] 
  This is clearly natural in $A$ and $B$. Moreover it is
  an elementary property of adjunctions that $i^+ r(\id_{i^+i_+ A})=\id_{i^+i_+A}$. As
  $i^+$ was assumed to be faithful, we conclude $r(\id)=\id$.

If $f$ is of the form $i^+g$, then by naturality we have
\[
r(i^+g) = r(g_*\id) = g_* r(\id) = g_*\id = g
\]
so $r$ is indeed a retraction.
\end{proof}

\begin{proof}[Proof of Lemma~\ref{lem:contractible_over_cala_Phi[t,t(-1)]_versus_over_cala}]
  We conclude from the fact that  $i^+(\cone(f)) = \cone(i^+f)$ and 
  Lemma~\ref{lem:elementary_facts}~\ref{lem:elementary_facts:chain_homotopy_equivalence_cone}
  that it suffices to show for a bounded $\cala_{\Phi}[t]$-chain complex $C$ that
  $C$ is contractible as $\cala_{\Phi}[t]$-chain complex if and only if $i^+C$ is
  contractible as $\cala^{\kappa}$-chain complex.

 We argue by induction on the length $d$ of $C$. The induction beginning $d =
  0$ is trivial; the induction step from $d-1 \ge 0$ to $d$ is done as follows.

  We assume for simplicity that $C$ is concentrated in degrees $[0,d]$.  Since $i^+C$ is
  contractible, there exists a morphism $s_0 \colon i^+ C_0 \to i^+C_1$ in
  $\cala^{\kappa}$ such that the composite $i^+ c_1 \circ s_0 \colon i^+ C_0 \to i^+C_0$
  is the identity. Let $\gamma_0:= r(s_0)$ for a splitting $r$ as in Lemma~\ref{lem:abstract_retraction_from_adjunction}. Then, by naturality,
  \[
  c_1\circ\gamma_0 = (c_1)_*r(s_0) = r((c_1)_* s_0) = r(i^+c_1\circ s_0)= r(\id) =
  \id.
  \]

By Lemma~\ref{lem:elementary_facts}~\ref{lem:elementary_facts:induction_step_argument} 
it follows that there are an $\cala_{\Phi}[t]^\kappa$-chain complex $D$ and elementary $\cala_{\Phi}[t]^\kappa$-chain complexes
$E$ and $E'$  such that 
\[
C \oplus E\cong D \oplus E'
\]
and $D$ is concentrated in degrees $[1,d]$.
Since $i^+C$ is contractible, $i^+D$  is contractible by 
Lemma~\ref{lem:elementary_facts}~\ref{lem:elementary_facts:chain_homotopy_equivalence_2_of_3}.
By the induction hypothesis  $D$ is a contractible
$\cala_{\Phi}[t]^\kappa$-chain complex. Therefore $C$ is a contractible $\cala_{\Phi}[t]^\kappa$-chain complex, again
by Lemma~\ref{lem:elementary_facts}~\ref{lem:elementary_facts:chain_homotopy_equivalence_2_of_3}.
\end{proof}

 
 \subsection{Finitely dominated chain complexes}
 \label{subsec:Finitely_dominated_chain_complexes}
 
 Let $C$ be an $\cala^{\kappa}$-chain complex. Recall that we view $\cala$ as a full
 additive subcategory of $\cala^{\kappa}$.  We call $C$ \emph{finitely dominated} if there
 exists a bounded $\cala$-chain complex $D$ and $\cala^{\kappa}$-chain maps $i \colon C \to D$ and
 $r \colon D \to C$ such that $r \circ i$ is $\cala^{\kappa}$-chain homotopic to the
 identity. A proof of the next result can be found in~\cite[Proposition~3.2~(ii)]{Ranicki(1985)}.
 
 \begin{lemma} \label{lem:criterion_for_finitely_dominated} Suppose that $\cala$ is
   idempotent complete. Then an $\cala^{\kappa}$-chain complex is finitely dominated if and
   only if it is $\cala^{\kappa}$-chain homotopy equivalent to a bounded $\cala$-chain
   complex.
 \end{lemma}

 
 \subsection{Homotopy finite  chain complexes}
 \label{subsec:Homotopy_finite_chain_complexes}

Let $\calb$ be an additive category with a full additive  subcategory $\cala \subseteq \calb$.
We call a $\calb$-chain complex $C$ \emph{homotopy $\cala$-finite} if $C$ is
$\calb$-chain homotopy equivalent to a bounded $\cala$-chain complex.

\begin{lemma}\label{lem:homotopy_finite_two_of_three}

  Let $0 \to C \xrightarrow{i} D \xrightarrow{p} E \to 0$ be an exact sequence of
  $\calb$-chain complexes.  Suppose that two of the three $\calb$-chain complexes $C$,
  $D$, and $E$ are $\cala$-homotopy finite. Then all three are  homotopy $\cala$-finite.
\end{lemma}
\begin{proof}
  We begin with the case where $C$ and $D$ are homotopy $\cala$-finite. We have to show
  that $E$ is homotopy $\cala$-finite.
  
  Choose bounded $\cala$-chain complexes $P$ and $Q$ together with $\calb$-chain homotopy
  equivalences $v \colon P \to C$ and $w \colon Q \to D$.  Then there is a $\cala$-chain
  map $f \colon P \to Q$ such that $w \circ f \simeq u \circ v$ holds as $\calb$-chain
  maps.  From
  Lemma~\ref{lem:elementary_facts}~\ref{lem:elementary_facts:maps_between_mapping_cylinders}
  we obtain a $\calb$-chain map $F \colon \cyl(f) \to D$ satisfying $F \circ i = u \circ
  v$ for the canonical inclusion $i \colon P \to \cyl(f)$. We obtain a commutative diagram
  of $\calb$-chain complexes whose rows are short exact sequences
\[
\xymatrix{0 \ar[r] 
&  C \ar[r]^{i}
& D \ar[r]^{p}
& E \ar[r]
& 0
\\
0 \ar[r] 
& P \ar[r]^{i} \ar[u]^{v}
& \cyl(f) \ar[r] \ar[u]^{F}
& \cone(f)  \ar[r] \ar[u]^{\overline{F}}
& 0
}
\]
Since $v$ and $F$ are $\calb$-chain homotopy equivalences,
$\overline{F}$ is a $\calb$-chain homotopy equivalence by
Lemma~\ref{lem:elementary_facts}~\ref{lem:elementary_facts:chain_homotopy_equivalence_2_of_3}.
Since $P$ and $Q$ are bounded $\cala$-chain complexes,
$\cone(f)$ is a bounded $\cala$-chain complexes. This proves that
$E$  is homotopy $\cala$-finite.

Next we deal with the second case, where $D$ and $E$ are homotopy $\cala$-finite.
We have to show that $C$ is homotopy $\cala$-finite.
The exact sequence $0 \to C \xrightarrow{i} D \xrightarrow{p} E \to 0$ induces an exact
sequence of $\calb$-chain complexes $0 \to D \to \cyl(p) \to \Sigma C \to 0$ and $\cyl(p)$
is $\calb$-chain homotopy equivalent to $E$. 
Since $D$ and $\cyl(p)$ are homotopy $\cala$-finite, the first case applied to 
$0 \to D \to \cyl(p) \to \Sigma C \to 0$ implies that $\Sigma C$ and hence $C$ are
homotopy $\cala$-finite.  

If $C$ and $E$ are homotopy $\cala$-finite, then $\cyl(p)$ and
$\Sigma C$ are homotopy $\cala$-finite and by the second case applied to 
$0 \to D \to \cyl(p) \to \Sigma C \to 0$ we conclude that $D$ is homotopy $\cala$-finite.
\end{proof}

 
 \subsection{Chain homotopy equivalences and cofibrations}
 \label{subsec:Chain_homotopy_equivalences_and_cofibrations}
 
 For the purpose of this paper, we define a cofibration of chain complexes in an additive category to be a chain
 map $i\colon C\to D$ which is level-wise split-injective.  The next lemma is well-known
 for cofibrations of spaces, see for instance~\cite[Proposition~5.2.5 on
 page~108]{Dieck(2008)}.
 
 \begin{lemma} \label{lem:Chain_homotopy_equivalences_and_cofibrations}
 Let $j(D) \colon C \to D$ and $j(E) \colon C \to E$ be  cofibrations of $\cala$-chain complexes.
 Suppose that there exists an $\cala$-chain homotopy equivalence $v \colon E \to D$
 such that $v \circ j(E) = j(D)$. 
 
 Then there exists an $\cala$-chain map $w \colon D \to E$ with $w \circ j(D) = j(E)$ 
 together with a chain homotopy $h \colon v \circ w \simeq \id_D $  satisfying $h \circ j(D) = 0$.
 \end{lemma}
 \begin{proof} 
   Choose a chain map $w' \colon D \to E$ together with a chain homotopy
   $h' \colon w' \circ v \simeq \id_E$.  Since $j(E) \colon C \to E$ is a cofibration,
   we may choose for each $n$ a morphism $r_n \colon D_n \to C_n$ with $r_n \circ j(D)_n =
   \id_{C_n}$. Letting $H_n' \colon D_n \to E_{n+1}$ be the composite $h'_n\circ j(E)_n\circ r_n$, we see that
   \[
   H_n' \circ j(D)_n = h_n' \circ j(E)_n.
   \]
   Define a new chain map $w'' \colon D \to
   E$ by putting $w_n'' = w'_n + e_{n+1} \circ H_n' + H_{n-1}' \circ d_n$.  Then $w''$ is
   homotopic to $w'$ and hence still a chain homotopy inverse of $v$ and satisfies
 \begin{eqnarray*}
 w'' \circ j(D) 
 & = & 
 \bigl(w' + e \circ H' + H' \circ d\bigr) \circ j(D)
 \\
 & = & 
 w' \circ j(D) +  e \circ H'  \circ j(D) + H' \circ d \circ j(D)
 \\
 & = & 
 w' \circ j(D) +  e \circ H'  \circ j(D) + H' \circ  j(D) \circ c
 \\
 & = & 
 w' \circ j(D) +  e \circ h'  \circ j(E) + h' \circ  j(E) \circ c
 \\
 & = & 
 w' \circ j(D)  + e \circ h'  \circ j(E) + h' \circ  e \circ j(E) 
 \\
 & = & 
 w' \circ j(D) +  \bigl(e \circ h'  +  h' \circ  e \bigr) \circ j(E) 
 \\
 & = & 
 w' \circ j(D) +  \bigl(\id_D - w' \circ v \bigr) \circ j(E) 
 \\
 & = & 
 w' \circ j(D) +  j(E) -  w' \circ v \circ j(E) 
 \\
 & = & 
 w' \circ j(D) +  j(E) -  w' \circ j(D) 
 \\
 & = & 
 j(E).
 \end{eqnarray*}
 Let $[D,E]_C$ be the set of chain homotopy classes relative $C$ of chain maps 
 $f \colon D \to E$ satisfying $f \circ j(D) = j(C)$, where a chain homotopy $h$ relative $C$ is a chain
 homotopy satisfying $h \circ j(D) = 0$. Define $[D,D]_C$ analogously. We
 obtain maps $w''_* \colon [D,D]_C \to [D,E]_C$ and $v_* \colon [D,E]_C \to [D,D]_C$ by
 taking composites.  Fix a chain homotopy $h \colon v \circ w'' \simeq \id_D$. Define a map
 $h_{\sharp} \colon [D,D]_C \to [D,D]_C$ by sending the class of $f \colon D \to D$ to the
 class of $f + d \circ h + h \circ d$. This is well-defined since
 \begin{eqnarray*}
 (d \circ h + h \circ d) \circ j(D) 
 & = & 
 (v \circ w'' - \id) \circ j(D) 
 \\
 & = & 
 v \circ w'' \circ j(D) - j(D)
 \\
 &= &
 j(D) - j(D) 
 \\
 & = & 
 0.
 \end{eqnarray*}
 Next we prove 
 \begin{eqnarray}
 h_{\sharp} \circ v_* \circ w''_* 
 & = & 
 \id_{[D,D]_C}.
 \label{h_sharp_circ-v_circ_w}
 \end{eqnarray}
 Consider a chain map $f \colon D \to D$ with $f \circ j(D) = j(D)$.
 Then $h_{\sharp} \circ v_*\circ  w''_* ([f])$ is the chain homotopy class relative $C$ of
 $v \circ w'' \circ f +  d \circ h + h \circ d$.
 We compute
 \begin{eqnarray*}
 \lefteqn{v \circ w'' \circ f +  d \circ h + h \circ d - f}
 & & 
 \\
 & = &
 v \circ w'' \circ f +  d \circ h \circ f + h \circ d \circ f - d \circ h \circ f - h \circ d \circ f +  d \circ h + h \circ d -f 
 \\
 & = &
 \bigl(v \circ w''  +  d \circ h + h \circ d\bigr) \circ f - d \circ h \circ f - h \circ f \circ d +  d \circ h + h \circ d - f
 \\
 & = &
 \id_D \circ f + d \circ (h - h \circ f) + (h - h \circ f)\circ d- f
 \\
 & = &
 d \circ (h - h \circ f) + (h - h \circ f)\circ d.
 \end{eqnarray*}
 This implies~\eqref{h_sharp_circ-v_circ_w} since 
 $(h - h \circ f) \circ j(D) = h  \circ j(D) - h \circ f  \circ j(D) =h \circ j(D) - h \circ j(D) = 0$ 
 holds.  Obviously $h_{\sharp}$
 is a bijection, an inverse is given by $(-h)_{\sharp}$.  We conclude
 from~\eqref{h_sharp_circ-v_circ_w} that $v_* \colon [D,E]_C \to [D,D]_C$ is surjective.
 Let $w \colon D \to E$ be any chain map with $w \circ j(E) = j(D)$ such that the class of
 $[w]$ is mapped under $v_*$ to the class of the identity.
 \end{proof}


\section{Some basic tools for connective $K$-theory}
\label{sec:Some_basic_tools_for_connective_K-theory}

We collect some basic tools about connective $K$-theory of exact categories and
Waldhausen categories.


\subsection{The Gillet-Waldhausen Theorem}
\label{sec:The_Gillet_Waldhausen_theorem}

Throughout this subsection, let $\cale$ be an exact category. The Gillet-Waldhausen
theorem compares the $K$-theory of $\cale$ with the $K$-theory of the category
$\Ch(\cale)$ of bounded chain complexes over $\cale$. This is a slightly subtle problem
since on $\Ch(\cale)$ there might be different notions of weak equivalences which give
potentially different $K$-theories.

In this section we define a notion of a \emph{canonical Waldhausen structure} 
on $\Ch(\cale)$ such that the following holds:

\begin{theorem}[Gillet-Waldhausen Theorem]\label{the:Gillet_Waldhausen}
  The inclusion functor $\cale\to\Ch(\cale)$ which considers an object as a 0-dimensional
  chain complex induces a homotopy equivalence
  \[
   \bfK(\cale)\xrightarrow{\simeq} \bfK(\Ch(\cale)),
   \]
  provided $\Ch(\cale)$ carries the canonical Waldhausen structure, 
  see Definition~\ref{def:canonical_Waldhausen_structure}.
\end{theorem}

We denote admissible monomorphisms and admissible epimorphisms in $\cale$ 
by the symbols `$\rightarrowtail$' and `$\twoheadrightarrow$', respectively. 

\begin{definition}[Exact structure on chain complexes]
\label{def:Exact_structure_on_chain_complexes}
A chain complex $C$ in $\cale$ is called \emph{exact} if each differential factors as
\[
c_n\colon C_n \overset{p_n}{\twoheadrightarrow} Z_{n-1} \overset{i_{n-1}}{\rightarrowtail} C_{n-1}
\]
such that all the sequences
\[
Z_{n} \overset{i_{n}}{\rightarrowtail} C_n\overset{p_n}{\twoheadrightarrow} Z_{n-1}
\]
are exact.
\end{definition}

Notice that if $\cale$ is abelian then this is just the usual notion of exactness in the
sense that the chain complex has no homology.

\begin{definition}[Property (P)]
  \label{def:property_(P)}
  We say that $\cale$ \emph{satisfies property (P)} if any split surjection $p\colon
    A\to B$ in $\cale$ is an admissible epimorphism, i.e., is part of an exact sequence
    $K\rightarrowtail A\twoheadrightarrow B$.
\end{definition}

In the case where $\cale$ satisfies property (P)
the canonical Waldhausen structure agrees with the one naturally defined by the exact structure.
In this case the Gillet-Waldhausen Theorem is well-known and can be found for instance 
in~\cite[1.11.7]{Thomason-Trobaugh(1990)}. 

We first discuss the choice of Waldhausen structure provided that property (P) may not hold.
It is not a good idea to declare a chain map to be a weak equivalence by demanding
that its mapping cone is exact. Namely, the following example shows that a 
contractible chain complex $C$ may not be  exact and
that a direct summand of an exact chain complex may not be exact either.

\begin{example}
  Let $M$ be a module over some ring $R$ such that $M$ is not free but stably free, i.e.,
  such that $M\oplus R^m$ is \change{finitely generate free} for some $m$. Then, the chain complex
  \[
  R^m \xrightarrow{\begin{pmatrix} 0 & 1 \end{pmatrix}} M\oplus R^m
  \xrightarrow{\begin{pmatrix} 1 & 0 \\ 0 & 0 \end{pmatrix}} M\oplus R^m
  \xrightarrow{\begin{pmatrix} 0 & 1 \end{pmatrix}} R^m
  \]
  is a chain complex in the exact category $\calr$ of finitely generated free
  $R$-modules. As such it is not exact, for the last map has no kernel in $\calr$. On the
  other hand, it is chain contractible and acyclic as a chain complex of $R$-modules.

  Moreover, if we take direct sum with the exact chain complex $R^m\xrightarrow{\id} R^m$
  to the middle degrees, then the resulting chain complex is exact in the category of \change{finitely generated free}
  $R$-modules. 
\end{example}

First we explain how property (P) ensures that this pathology does not arise.

\begin{definition}[Waldhausen structure for an exact category with Property (P)]
  \label{def:waldhausen_structure_on_category_with_P}
  If $\cale$ satisfies property (P), a chain map $f\colon C\to D$ is called
    \begin{enumerate}
    \item a \emph{cofibration} if it is degree-wise an admissible monomorphism;
    \item a \emph{weak equivalence} if its mapping cone is an exact chain complex.
    \end{enumerate}
\end{definition}

For the following, we say that an exact functor $F\colon \cale\to \cale'$ between exact
categories \emph{reflects exactness} provided
\[
E_0\to E_1\to E_2 \mathrm{~is~exact~in~}\cale \: \Longleftrightarrow \: F(E_0)\to
F(E_1)\to F(E_2) \mathrm{~is~exact~in~}\cale'.
\] 
We say that $F$ \emph{reflects admissible epimorphisms} provided
\begin{multline*}
E_1\to E_2 \mathrm{~admissible~epimorphism~in~}\cale \: \Longleftrightarrow \\
 F(E_1)\to F(E_2)
\mathrm{~admissible~epimorphism~in~}\cale'.
\end{multline*}

\begin{lemma}\label{lem:functor_reflecting_exactness_and_admissible_epis}
  If an exact functor $F\colon \cale\to \cale'$ reflects exactness and admissible
  epimorphisms between categories with property (P), then $\Ch(F)\colon
  \Ch(\cale)\to\Ch(\cale')$ reflects weak equivalences.
\end{lemma}

\begin{proof}
  It suffices to show that if $F(C)$ is exact in $\Ch(\cale')$ then so is $C$ in
  $\Ch(\cale)$. The argument is by induction on the length of $C$. If is has length at
  most 2, then the claim holds by assumption.

  For the inductive step, note that a chain complex $C$ concentrated in $[0,n]$ is exact
  if and only if $c_1\colon C_1\to C_0$ is an admissible epimorphism with kernel $Z_1$ so
  that the induced chain complex
\begin{equation}\label{eq:induced_chain_complex}
  \dots C_2\xrightarrow{c_3} C_2\xrightarrow{c_2} Z_1 \to 0 
\end{equation}
is exact. Thus if $F(C)$ is exact in $\Ch(\cale')$, then $F(c_1)$ is an admissible
epimorphism in $\cale'$. Hence by assumption, $c_1$ is an
admissible epimorphism in $\cale$, with a kernel $Z_1$. Applying the inductive hypothesis
to the chain complex~\eqref{eq:induced_chain_complex} concludes the proof.
\end{proof}

\begin{lemma}\label{lem:waldhausen_structure_under_property_P}
  Suppose that $\cale$ satisfies property (P). Then:
  \begin{enumerate}
  \item \label{lem:waldhausen_structure_under_property_P:(1)}
  With the above choice of cofibration and weak equivalence, $\Ch(\cale)$ is a
    Waldhausen category that satisfies the saturation, extension, and cylinder axioms;
  \item \label{lem:waldhausen_structure_under_property_P:(2)}
  If the exact structure on $\cale$ is the split-exact structure of the underlying
    additive category, then a chain map is a weak equivalence if and only if it is a chain
    homotopy equivalence.
  \end{enumerate}
\end{lemma}
\begin{proof} 
  In the case where $\cale$ is an abelian category (so that weak equivalences are
  homology equivalences by the long exact homology sequence), the conclusion is
  well-known.

  In the general case, denote by $i\colon \cale \to \cale'$ the Gabriel-Quillen
  embedding~\cite[A.7.1]{Thomason-Trobaugh(1990)}, where $\cale'$ is the abelian category
  of contravariant left exact functors $\cale\to \Ab$ to the abelian category $\Ab$ of abelian groups. 
  (Note that surjections in $\cale'$
  are not the objectwise surjective transformations.)
  By~\cite[A.7.16]{Thomason-Trobaugh(1990)}, the image if $i$ is a full subcategory,
  and $i$ reflects exactness and admissible
  epimorphisms. Hence Lemma~\ref{lem:functor_reflecting_exactness_and_admissible_epis}
  implies that $\Ch(\cale)$ is a full Waldhausen subcategory of $\Ch(\cale')$ which is
  closed under taking mapping cylinders; in particular it is a Waldhausen category
  satisfying the three extra axioms.
  \\[2mm]~\ref{lem:waldhausen_structure_under_property_P:(2)}
  If the exact structure is the split exact one, then any additive contravariant
  functor $\cale\to \Ab$ is automatically left exact. This implies that $\cale'$ coincides
  with the abelian category
  of contravariant  functors $\cale\to \Ab$ where the abelian structure on $\cale'$ 
  is given objectwise by the one on $\Ab$,
  and the Gabriel-Quillen embedding $i$ 
  is just the additive Yoneda embedding. Therefore the image of any object of $\cale$ under 
 $i$  is projective. So for a chain
  map $f$ in $\cale$ its image $i(f)$ is exact (in other words, a homology equivalence) if and only
  if $i(f)$ is a chain homotopy equivalence in $\cale'$. But this is equivalent to $f$
  being a chain homotopy equivalence as the embedding $i$ is full.
\end{proof}

By~\cite[A.9.1]{Thomason-Trobaugh(1990)}, the category $\Idem(\cale)$ becomes an exact 
category if we call a sequence exact if it is a direct summand of an exact sequence of $\cale$; 
moreover the image of the inclusion $\eta\colon \cale\to\Idem(\cale)$ is an exact subcategory.

\begin{definition}
Let $\calp\cale\subset\Idem(\cale)$ be the full exact subcategory of all objects $A$
that are stably in $\cale$, i.e., for which $A \oplus A'$ is isomorphic to an object of
$\cale$, for some $A'\in\cale$. 
\end{definition}

This clearly define an endofunctor $\calp$ on the category of exact categories; 
moreover the full embedding $\eta\colon \cale\to\Idem(\cale)$ factors through a 
full embedding $I\colon \cale\to \calp\cale$ which reflects exactness as $\eta$ does.

\begin{lemma}\label{lem:property_P_completion}
  \begin{enumerate}
  \item \label{lem:property_P_completion:endo} 
  For any exact category $\cale$, the exact category $\calp\cale$ satisfies property (P).

   \item \label{lem:property_P_completion:embedding} 
   The embedding $I\colon \cale\to\calp\cale$ is an equivalence if $\cale$
    already satisfies property \change{(P).}
  \item \label{lem:property_P_completion:K-theory}
  The embedding $I$ induces a homotopy equivalence on $K$-theory.
  \end{enumerate}
\end{lemma}

\begin{proof}
To show~\ref{lem:property_P_completion:endo}, suppose that 
$p\colon A\to B$ is a split surjection in $\calp\cale$, with split $s$. Then $K=(A, 1-s\circ p)$ is
a kernel of $p$ in $\Idem(\cale)$ and $K\rightarrowtail A\twoheadrightarrow B$ is
exact. We need to show that $K\in \calp\cale$.

To do that, let $A', B'\in\cale$ such that $A\oplus A'$ and $B\oplus B'$ are isomorphic to
objects in $\cale$. Let
\[
 p':=p\oplus\id\colon A\oplus A'\oplus B' \to B\oplus A'\oplus B'.
 \] 
Obviously $p$ and $p'$ obviously have isomorphic kernels. As $p'$ is isomorphic to a split surjection in
$\cale$, its kernel lies in $\calp\cale$.  

Part~\ref{lem:property_P_completion:embedding} is
clear. Part~\ref{lem:property_P_completion:K-theory} follows from Waldhausen's Cofinality
Theorem, see~\cite[1.5.9]{Waldhausen(1985)}, as $\cale\subset\calp\cale$ is strictly
cofinal.
\end{proof}

Now we proceed to define the canonical Waldhausen structure on an arbitrary exact category
$\cale$.

\begin{definition}[Canonical Waldhausen
  structure]\label{def:canonical_Waldhausen_structure}
  The \emph{canonical Waldhausen structure} on $\Ch(\cale)$ is defined as follows: A
  morphism $f \colon C \to D$ in $\Ch(\cale)$ is a canonical cofibration if and only if it
  is an admissible monomorphism in each degree. A morphism $f$ is a canonical weak
  equivalence if and only if $I(f)$ has an exact mapping cone in $\Ch(\calp\cale)$.
\end{definition}

\begin{lemma}\label{lem:canonical_Waldhausen_structure_is_Waldhausen_structure}
  For any exact category $\cale$, the category $\Ch(\cale)$, when endowed with its
  canonical Waldhausen structure, we have:

  \begin{enumerate}
  
 \item \label{lem:canonical_Waldhausen_structure_is_Waldhausen_structure:axioms}
  With the above choice of cofibrations and weak equivalences, $\Ch(\cale)$ is a
    Waldhausen category that satisfies the saturation, extension, and cylinder axioms;
  
 \item \label{lem:canonical_Waldhausen_structure_is_Waldhausen_structure:split}
  If the exact structure on $\cale$ is the split-exact structure of the underlying
    additive category, then a chain map is a weak equivalence if and only if it is a chain
    homotopy equivalence;
 
   \item \label{lem:canonical_Waldhausen_structure_is_Waldhausen_structure:K)}
  The inclusion
  \[
  \Ch(\cale)\to \Ch(\calp\cale)
   \]
  induces a homotopy equivalence on $K$-theory.
\end{enumerate}
\end{lemma}

\begin{proof}\ref{lem:canonical_Waldhausen_structure_is_Waldhausen_structure:axioms}
  This follows from 
  Lemma~\ref{lem:waldhausen_structure_under_property_P}~\ref{lem:waldhausen_structure_under_property_P:(1)} 
  applied to $\calp\cale$
  and Lemma~\ref{lem:property_P_completion}~\ref{lem:property_P_completion:endo}  
  since $\Ch(\cale)$ is a full Waldhausen subcategory, closed under
  taking mapping cylinders, of $\Ch(\calp\cale)$ with the Waldhausen structure defined in
  Definition~\ref{def:waldhausen_structure_on_category_with_P}. 
  \\[2mm]~\ref{lem:canonical_Waldhausen_structure_is_Waldhausen_structure:split}
  If the exact structure on
  $\cale$ is the split exact one, then the same is true for $\Idem(\cale)$ and hence for
  $\calp\cale$. Hence assertion~\ref{lem:waldhausen_structure_under_property_P:(2)} 
   follows from  
  Lemma~\ref{lem:waldhausen_structure_under_property_P}~\ref{lem:waldhausen_structure_under_property_P:(2)} 
  applied to $\calp\cale$.
  \\[2mm]~\ref{lem:canonical_Waldhausen_structure_is_Waldhausen_structure:K)}
  This follows from Waldhausen's Cofinality Theorem, 
  see~\cite[1.5.9]{Waldhausen(1985)}, as
  $\Ch(\cale)$ is strictly cofinal in $\Ch(\calp\cale)$.
\end{proof}

\begin{proof}[Proof of the Gillet-Waldhausen Theorem~\ref{the:Gillet_Waldhausen}]
  If $\cale$ satisfies property (P), then the Gillet-Waldhausen Theorem is proved
  in~\cite[1.11]{Thomason-Trobaugh(1990)}. In the general case there is a commutative
  diagram
  \[
   \xymatrix{\bfK(\cale) \ar[rr] \ar[d]^\simeq && \bfK(\Ch(\cale)) \ar[d]^\simeq\\
    \bfK(\calp\cale) \ar[rr]^\simeq && \bfK(\Ch(\calp\cale)) }
   \]
  where the lower
  horizontal map is a homotopy equivalence since $\calp\cale$ satisfies property (P)
  by Lemma~\ref{lem:property_P_completion}~\ref{lem:property_P_completion:endo},
  and the vertical maps are homotopy equivalences by
  Lemma~\ref{lem:property_P_completion}~\ref{lem:property_P_completion:K-theory} and
  Lemma~\ref{lem:canonical_Waldhausen_structure_is_Waldhausen_structure}~%
\ref{lem:canonical_Waldhausen_structure_is_Waldhausen_structure:K)}. 
\end{proof}

To identify the canonical weak equivalences on the exact categories we need to consider,
we use the following generalization of
Lemma~\ref{lem:functor_reflecting_exactness_and_admissible_epis}:

\begin{lemma}\label{lem:F_reflects_then_CH(F)_reflects}
  Let $F\colon \cale\to \cale'$ be an exact functor between exact categories which
  reflects exactness and admissible epimorphisms. Then
  $\Ch(f)\colon\Ch(\cale)\to\Ch(\cale')$ reflects canonical weak equivalences.
\end{lemma}

\begin{proof} Obviously $\calp F$ is exact.
  Let $E_0\to E_1\to E_2$ be a sequence in $\calp\cale$ which becomes exact after applying
  $\calp F$. For suitable $Y,Z\in\cale$, the direct sum
  \begin{equation}\label{eq:direct_sum_of_exact_sequences}
    (E_0\to E_1\to E_2) \oplus (Y\rightarrowtail Y\oplus Z\twoheadrightarrow Z)  
  \end{equation}
  is a sequence in $\cale$ which becomes exact after applying $F$. As $F$ reflects
  exactness,~\eqref{eq:direct_sum_of_exact_sequences} is an exact sequence in
  $\cale$. This implies that the first summand on~\eqref{eq:direct_sum_of_exact_sequences}
  is an exact sequence in $\calp\cale$.

  This argument shows that $\calp F$ reflects exactness. A similar argument shows that
  $\calp F$ reflects admissible epimorphisms. Now apply
  Lemma~\ref{lem:functor_reflecting_exactness_and_admissible_epis}.
\end{proof}

\begin{example}[Twisted Nil category] \label{exa:twisted_Nil_category}
  Given an additive category $\cala$ with automorphism $\Phi$, consider the twisted Nil
  category $\Nil(\cala,\Phi)$ from
  Section~\ref{sec:Strategy_of_proof_for_the_connective_version_NIL}. It comes with an
  additive functor $F\colon \Nil(\cala,\Phi) \to\cala$, sending $(A,f)$ to its underlying
  object $A$. This functor reflects exactness by definition, and it is not hard to see
  that is reflects admissible epimorphisms. Hence a chain map $\varphi$ in
  $\Nil(\cala,\Phi)$ is a canonical weak equivalence if and only if $F(\varphi)$ is
  one. By conclusion (ii) of
  Lemma~\ref{lem:canonical_Waldhausen_structure_is_Waldhausen_structure} the latter
  statement is equivalent to $F(\varphi)$ being a chain homotopy equivalence.
\end{example}

\begin{example}[Projective line]\label{exa:projective_line}
  A similar statement holds for the twisted projective line category $\calx$ from
  Section~\ref{subsec:The_twisted_projective_line}: By definition, the functor
  \[
   F=(k^+, k^-)\colon \calx\to\cala_\Phi[t]\times \cala_\Phi[t\inv]
  \] 
  reflects  exactness. It is not hard either to see that it reflects admissible epimorphisms.


  Hence a chain map $f$ in $\calx$ is a canonical weak equivalence if and only if both
  $k^+(f)$ and $k^-(f)$ are chain homotopy equivalences.
\end{example}

Unless specified otherwise, all the chain categories in the sequel will carry the canonical Waldhausen structure
and we often use Lemma~\ref{lem:canonical_Waldhausen_structure_is_Waldhausen_structure}~%
\ref{lem:canonical_Waldhausen_structure_is_Waldhausen_structure:split}
without mentioning this again.


\subsection{The Fibration Theorem}
\label{sec:The_Fibration_Theorem}

In the sequel we use the definitions and notation of Waldhausen~\cite{Waldhausen(1985)}.
Suppose that $\calc$ is a category with cofibrations
and that $\calc$  is equipped with two categories of weak equivalences, one finer than
the other, $v\calc \subseteq w \calc$. Thus $\calc$ becomes a  Waldhausen category in two ways.
Let $\calc^w$ denote the subcategory with cofibrations of $\calc$ given
by the objects $C$  in $\calc$ having the property that the map  $A \to \pt$ belongs to  $w\calc$. 
Then  $\calc^w$ inherits two Waldhausen structures
if we put $v\calc^w = \calc^w \cap v\calc $  and $w\calc^w = \calc^w \cap w\calc$.

\begin{theorem}[Fibration Theorem] \label{the:Fibration_Theorem}
Suppose that $\calc$ has a cylinder functor, and the 
category of weak equivalences $w\calc$ satisfies the cylinder axiom, saturation axiom,
and extension axiom. Then:

\begin{enumerate}

\item \label{the:Fibration_Theorem:spaces}
The square of path connected spaces
\xycomsquare{|vS.C^w|}{}{|wS.C^w| \simeq \pt}
{}{}
{|vS.C|}{}{|wS.C|}
is homotopy cartesian, and the upper right term is contractible;

\item \label{the:Fibration_Theorem:spectra}
We get a homotopy fibration of spectra
\[
\bfK(C^w,v) \to \bfK(\calc,v) \to \bfK(\calc,w).
\]
\end{enumerate}
\end{theorem}
\begin{proof}~\ref{the:Fibration_Theorem:spaces}
This is proved in~\cite[Theorem~1.6.4]{Waldhausen(1985)}.
\\[1mm]~\ref{the:Fibration_Theorem:spectra}
The functor loop space $\Omega$ commutes with homotopy pullbacks and homotopy fibrations. 
The $K$-theory spectrum $\bfK(\calc)$ is given by a sequence of maps
\[
|w\calc| \to \Omega |wS.\calc| \to \Omega \Omega  |wS.S.\calc| \to \Omega \Omega  \Omega |wS.S.S.\calc|\to \cdots
\]
where all structure maps are weak equivalences possibly except the first one,
see~\cite[page~330]{Waldhausen(1985)}. Hence 
assertion~\ref{the:Fibration_Theorem:spectra} follows from assertion~\ref{the:Fibration_Theorem:spaces}
\end{proof}


\subsection{The Approximation Theorem}
\label{sec:The_Approximation_Theorem}

 The following result is taken from~\cite[Theorem~1.6.7]{Waldhausen(1985)}.

 \begin{theorem}[Approximation Theorem] \label{the_Approximation_Theorem} Let $\calc_0$
   and $\calc_1$ be Waldhausen categories. Suppose that the weak equivalences in $\calc_0$ and
   $\calc_1$ satisfy the saturation axiom. Suppose further that $\calc_0$ has a cylinder
   functor and the weak equivalences in $\calc_0$ satisfy the cylinder axiom. Let $F
   \colon \calc_0 \to \calc_1$ be an exact functor. Suppose F has the approximation
   property, i.e., satisfies the following two conditions:

   \begin{enumerate}

   \item \label{the_Approximation_Theorem:A1} An arrow in $\calc_0$ is a weak equivalence
     in $\calc_0$ if and only if its image in $\calc_1$ is a weak equivalence in
     $\calc_1$;

   \item \label{the_Approximation_Theorem:A2} Given any object $C_0$ in $\calc_0$ and any
     map $f\colon F(C_0) \to C_1$ in $\calc_1$, there exist a cofibration $i \colon
     C_0 \to C_0'$ in $\calc_0$ and a weak equivalence $g \colon F(C_0') \to C_1$ in
     $\calc_1$ satisfying $f = g \circ F(i)$.

   \end{enumerate}
   
   Then the induced maps of spaces $|w\calc_0| \xrightarrow{\simeq} |w\calc_1|$ and
   $|wS.\calc_0| \xrightarrow{\simeq} |wS.\calc_1|$ and the map of spectra $\bfK(\calc_0)
   \xrightarrow{\simeq} \bfK(\calc_1)$ are homotopy equivalences.
 \end{theorem}


\subsection{Cisinski's version of the  Approximation Theorem}
\label{sec:Cisinskis_version_of_the_Approximation_Theorem}

The following result is a consequence of~\cite[Proposition~2.14]{Cisinki(2010inv)}.

\begin{theorem}[Cisinski's Approximation Theorem] \label{the:Cisinki_Approximation_Theorem} 
  Let $F \colon \calc_0 \to \calc_1$ be
  an exact functor of Waldhausen categories. Suppose for $k = 0,1$ that $\calc_k$ satisfy
  the saturation axiom and any morphism $f \colon C \to C''$ in $\calc_k$ factorizes as $C
  \xrightarrow{i} C' \xrightarrow{w} C''$ for a cofibration $i$ and a weak equivalence
  $w$.  Furthermore, we assume:

   \begin{enumerate}

   \item \label{the:Cisinki_Approximation_Theorem:A1} An arrow in $\calc_0$ is a weak equivalence
     in $\calc_0$ if and only if its image in $\calc_1$ is a weak equivalence in
     $\calc_1$;

   \item \label{the:Cisinki_Approximation_Theorem:A2} Given any object $C_0$ in $\calc_0$ and any
     map $f\colon F(C_0) \to C_1$ in $\calc_1$, there exists a commutative diagram in $\calc_1$
    \[
    \xymatrix{F(C_0) \ar[r]^-{f} \ar[d]_{F(u)}
    &
    C_1 \ar[d]^{v}_{\simeq}
    \\
    F(D_0) \ar[r]_-{w}^{\simeq}
    & 
    D_1
    }
    \]
    for a morphism $u \colon C_0 \to D_0$ in $\calc_0$ and weak equivalences $v \colon C_1 \to D_1$ and
   $w \colon F(D_0) \to D_1$ in $\calc_1$.
   \end{enumerate}
   
   Then  the map of spectra $\bfK(F) \colon \bfK(\calc_0)    \xrightarrow{\simeq} \bfK(\calc_1)$ 
    is a weak homotopy equivalence.
 \end{theorem}


\section{Proof of Theorem~\ref{the:homotopy_pull_back_of_calx}}
\label{sec:Proof_of_Theorem_ref(the:homotopy_pull_back_of_calx)}

This section is entirely  devoted to the proof of
Theorem~\ref{the:homotopy_pull_back_of_calx}.

In the first step of the proof of Theorem~\ref{the:homotopy_pull_back_of_calx}
we replace the additive category $\cala_\Phi[t]$ by a
larger exact category $\caly$ with equivalent $K$-theory. It is defined
as follows: An object of $\caly$ is a triple $(A, f, B)$ consisting of an object $A$ of $\cala_\Phi[t]$, an object $B$ of $\cala_\Phi[t,t\inv]$ (as opposed to $\cala_\Phi[t\inv]$ in the definition of $\calx$), and an isomorphism $f\colon j_+A\to B$ in $\cala_\Phi[t,t\inv]$. A
morphism from $(A, f, B)$ to $(C, g, D)$ is a morphism $\varphi^+\colon A\to C$ in $\cala_\Phi[t]$ and a commutative diagram
\[
\xymatrix{j_+A \ar[r]^f \ar[d]_{j_+\varphi^+} & B \ar[d]^{\varphi}\\
C \ar[r]^g & D  
}
\]
in $\cala_\Phi[t,t\inv]$. The category $\caly$ is exact in the same way as $\calx$ is.

\begin{lemma} \label{lem:caly_and_cala}
The functors
\[
u\colon \cala[t]\to \caly, \quad A\mapsto (A, \id, j_-A)
\]
and
\[
v\colon \caly\to\cala[t], \quad (A^+, f, A^-)\mapsto A^+
\]
are exact. The composite $v\circ u$ is the identity and the composite $u\circ v$ is naturally isomorphic to the identity functor.
In particular,  they induce homotopy equivalences on $K$-theory, homotopy inverse to
each other.
\end{lemma}

\begin{proof}
  It is clear that the functors are exact. Obviously $v\circ u$ is the identity. The composite
  $u\circ v$ is naturally isomorphic to the identity functor: the isomorphism in $\caly$ at the object $(A^+, f, A^-)$ is
  given by
  $(\id,f) \colon  (A^+,\id,j_+A^+)\xrightarrow{\cong}(A^+,f,A^-) $. This implies 
  $\bfK(u)\circ  \bfK(v)\simeq \id$.
\end{proof}

Denote by 
\[
k'\colon \calx\to\caly
\]
the inclusion functor, and define
\[
j'\colon \Chcat(\caly)\to \Chcat(\cala_\Phi[t,t^{-1}]), \quad (A^+,f,A^-) \mapsto A^-.
\]
 
Then  the square
\[
\xymatrix{{\Chcat(\calx)} \ar[rr]^(.4){\Chcat(k^-)} \ar[d]_{\Chcat(k')} &&
\Chcat(\cala_\Phi[t^{-1}]) \ar[d]^-{\Chcat(j^-)}\\
{\Chcat(\caly)} \ar[rr]_-{\Chcat(j')} &&   
\Chcat(\cala_\Phi[t,t^{-1}])
}
\]
is strictly commutative, and we are going to show that it induces a
homotopy pullback after applying $\bfK$. To show that the square is a
homotopy pullback on $K$-theory, we are going to show that the
horizontal homotopy fibers of $\bfK(\Chcat (k^-))$ and $\bfK(\Chcat(j'))$ and
agree.

Let $w\Chcat( \calx)$ be the subcategory of
$\Chcat(\calx)$ consisting of all chain maps which become weak equivalences
in  $\cala_\Phi[t^{-1}]$, after applying $\Chcat(k^-)$, and let $\Chcat(\calx)^w$ be the
full subcategory of $\Chcat(\calx)$ of all objects which are $w$-acyclic. In other words, an
object $(C^+, f, C^-)$ belongs to $\Chcat(\calx)^w$ if and only if $C^-$ is
contractible as an $\cala_\Phi[t,t\inv]$-chain complex. Similarly, denote
by $w\Chcat(\caly)$ the subcategory of all morphisms $f$ such that 
$\Chcat(j')(f)$ is a chain homotopy equivalence in $\Chcat(\cala_\Phi[t,t^{-1}])$, and adopt the
notation $\Chcat(\caly)^w$ for the $w$-acyclic objects.

\begin{lemma} \label{lem:certain_chain_homotopy_equivalences}
The maps
\begin{align*}
   \bfK(\Chcat(k^-))\colon & \bfK(\Chcat(\calx), w)\to \bfK(\Chcat(\cala_\Phi[t^{-1}]))\change{\quad \mathrm{and}}\\
   \bfK(\Chcat (j')) \colon & \bfK(\Chcat (\caly), w)\to \bfK(\Chcat (\cala_\Phi[t,t^{-1}]))
\end{align*}
are homotopy equivalences.
\end{lemma}
\begin{proof}
  We want to apply the Approximation Theorem~\ref{the_Approximation_Theorem}.  We give the
  details only for $\bfK(\Chcat(k^-))$, the analogous proof for $\bfK(\Chcat (j'))$ is
  left to the reader.

  It suffices to verify the assumptions appearing in the Approximation
  Theorem~\ref{the_Approximation_Theorem}.  The saturation and cylinder axioms are satisfied by 
  Lemma~\ref{lem:canonical_Waldhausen_structure_is_Waldhausen_structure}.
  The main task is to verify the
  conditions~\ref{the_Approximation_Theorem:A1} and~\ref{the_Approximation_Theorem:A2}
  appearing in the Approximation Theorem~\ref{the_Approximation_Theorem}.

  A morphism $f$ in $\Chcat(\calx)$ is by definition in $w\Chcat(\calx)$ if and only if
  $\Chcat (k^-)(f)$ is a chain homotopy equivalence in $\Chcat(\cala_\Phi[t^{-1}])$. This takes care of 
  condition~\ref{the_Approximation_Theorem:A1} for $\bfK(\Chcat(k^-))$.

  Finally we deal with condition~\ref{the_Approximation_Theorem:A2}.  Consider an object
  $(C^+, f, C^-)$ in $\Chcat(\calx)$ and a morphism $\varphi^-\colon C^-\to D^-$ in 
  $\Chcat(\cala_\Phi[t^{-1}])$. We will extend $\varphi^-$ to a morphism
  \[
   \varphi=(\varphi^+, \varphi^-)\colon (C^+, f, C^-)\to (D^+, g, D^-)
  \]
  in $\Chcat(\calx)$. 

Let $m\in \IZ$ such that $D^-_*=0$ for $*>m$. Choosing $K\gg 0$, let
$D^+$ be the following chain complex:
\[
\dots\to 0 \to D^-_m \xrightarrow{t^K\cdot \Phi^K(d_m)} \Phi^K(D^-_{m-1})
\xrightarrow{t^K\cdot \Phi^{2K}(d_{m-1})}\dots
\]
where $d_*$ is the differential of $D^-$. Notice that $D^+$ is a chain
complex  in $\cala_\Phi[t]$ provided $K$ was chosen big enough. 
Let $c_*$ be the differential of $C^+$. Enlarging
$K$ if necessary, the following diagram provides  a factorization of
$\varphi^-\circ f$ into an $\cala_\Phi[t]$-morphism $\varphi^+\colon C^+\to D^+$, 
followed by the $\cala_\Phi[t,t^{-1}]$-isomorphism $g\colon D^+\to D^-$
\[
\xymatrix{C_m^+ \ar[rr]^{t^K\cdot \varphi^-\circ f} \ar[d]^{c_m^+}
&& \Phi^K(D_m^-) \ar[rr]^{t^{-K}\cdot \id}  \ar[d]^{t^K\cdot \Phi^K(d_m^-)}
&& D_m^- \ar[d]^{d_m^-}
\\
C_{m-1}^+ \ar[rr]^{t^{2K}\cdot \varphi^-\circ f} \ar[d]^{c_{m-1}^+}
&& \Phi^{2K}(D_{m-1}^-) \ar[rr]^{t^{-2K}\cdot \id}  \ar[d]^{t^K\cdot \Phi^{2K}(d_{m-1}^-)}
&& D_{m-1}^- \ar[d]^{d_{m-1}^-}
\\
\vdots
&& \vdots
&& \vdots
}
\]
Hence $\varphi=(\varphi^+, \varphi^-)$ is a morphism in $\Chcat(\calx)$ projecting to $\varphi^-$ under $\Ch(k^-)$. Then, factoring $\varphi=\mu \circ\psi$ into a cofibration $\psi$ followed
  by a weak equivalence $\mu$ (using the mapping cylinder), we can write 
  $\varphi^-=\mu^-\circ \Chcat (k^-)(\psi)$ where $\psi$ is a cofibration 
  and $\mu^-$ is a weak equivalence, as required
  in condition~\ref{the_Approximation_Theorem:A2}. 
\end{proof}

\begin{theorem}\label{the:relative_fibration_sequences}
There are fibration sequences
\begin{align*}
   \bfK(\Chcat(\calx)^w) & \to \bfK(\calx)\to \bfK(\cala_\Phi[t^{-1}]);
   \\
   \bfK(\Chcat(\caly)^w) &\to \bfK(\caly)\to \bfK(\cala_\Phi[t,t^{-1}]).
\end{align*}
\end{theorem}
\begin{proof}
We give the
  details only for the first sequence, the analogous proof for the second one is
  left to the reader.

We apply the Fibration Theorem~\ref{the:Fibration_Theorem}~\ref{the:Fibration_Theorem:spectra} in the case
$\calc = \Ch(\calx)$, $w$ as described above and $v$ the structure of weak equivalences coming from
chain homotopy equivalences. Thus we obtain homotopy fibration  of spectra
\[
\bfK(\Ch(\calx)^w) \to \bfK(\Ch(\calx)) \to \change{\bfK(\Ch(\calx),w).}
\]
Because of Lemma~\ref{lem:certain_chain_homotopy_equivalences}
we obtain a homotopy fibration
\[
\bfK(\Ch(\calx)^w) \to \bfK(\Ch(\calx)) \to \change{\bfK(\Ch(\cala_\Phi[t^{-1}])).}
\]
Now the claim follows from Theorem~\ref{the:Gillet_Waldhausen}.
\end{proof}

\begin{lemma} \label{k'_is_homotopy_equivalence}
The functor $k'$ induces a homotopy equivalence
\[
\bfK(\Chcat(\calx)^w) \xrightarrow{\simeq} \bfK(\Chcat(\caly)^w).
\]
\end{lemma}

\begin{proof}
Again we will use the Approximation Theorem~\ref{the_Approximation_Theorem}. Let
\begin{equation}\label{diagram_defining_a_morphism_in_calx}
  \xymatrix{j_+C^+ \ar[r]^f \ar[d]_{j_+\varphi^+} & j_-C^- \ar[d]^{j_-\varphi^-}\\
 j_+D^+ \ar[r]^g & j_-D^- 
}
\end{equation}
represent a morphism in $\Chcat(\calx)^w$ which maps to a weak equivalence in
$\Chcat(\caly)^w$. Then $\varphi^+$ is a chain homotopy equivalence in
$\cala_\Phi[t]$ and $\varphi^-$ is a chain homotopy equivalence in
$\cala_\Phi[t,t^{-1}]$. By assumption, $C^-$ and $D^-$ are contractible in $\cala_\Phi[t^{-1}]$, so
$\varphi^-$ has to be an equivalence in $\cala_\Phi[t^{-1}]$. It follows that the
morphism given by~\eqref{diagram_defining_a_morphism_in_calx} is a weak equivalence
in~$\Chcat(\calx)^w$ already. This takes care of 
condition~\ref{the_Approximation_Theorem:A1}.

It remains to check condition~\ref{the_Approximation_Theorem:A2}.
Suppose now that
\begin{equation}\label{diagram_defining_a_morphism_in_caly}
  \xymatrix{j_+C^+ \ar[r]^f \ar[d]_{j_+\varphi^+} & C^- \ar[d]^{\varphi^-}\\
 j_+D^+ \ar[r]^g & D^- 
}
\end{equation}
represents a morphism in
$\Chcat(\caly)^w$, with $(C^+, f, C^-)$ in $\Chcat(\calx)^w$. We have to
factor this morphism through a map in $\Chcat(\calx)^w$ (which we may then
replace by a cofibration using the mapping cylinder) and a weak
equivalence in $\Chcat(\caly)^w$.

Notice that the morphism $\varphi^-$ is a chain homotopy equivalence in
$\cala_\Phi[t,t^{-1}]$, as both $C^-$ and $D^-$ are contractible in that
category, by assumption. We conclude from 
Lemma~\ref{lem:elementary_facts}~\ref{lem:elementary_facts:elementary} that
there is a chain isomorphism of the shape
\[
\begin{pmatrix}
     \varphi^- & y \\ x & z
\end{pmatrix}
\colon
C^- \oplus E \xrightarrow{\cong} D^- \oplus E'
\]
where $E$ and $E'$ are elementary chain complexes in
$\cala_\Phi[t,t^{-1}]$, or even in $\cala$ since both categories have the
same objects.

For large enough $K > 0$, the commutative diagram
\[
\xymatrix{j_+C^+ \ar[rrr]^f
\ar[d]_{\footnotesize\begin{pmatrix}j_+\varphi^+\\ t^K\cdot x\circ f
\end{pmatrix}}
&&& C^-
\ar[d]^{\footnotesize\begin{pmatrix}1 \\ 0 \end{pmatrix}}
\\
j_+D^+\oplus \Phi^K(i_0 E')
\ar[rrr]^{{\footnotesize\begin{pmatrix}g^{-1}\circ \varphi^- & g^{-1}\circ
y\\ t^K\cdot x & t^K\cdot z \end{pmatrix}}^{-1} }
\ar[d]_{\footnotesize\begin{pmatrix}1 & 0 \end{pmatrix}}
&&& C^-\oplus i_0 E
\ar[d]^{\footnotesize\begin{pmatrix}\varphi^- & y \end{pmatrix}}
\\
j_+D^+ \ar[rrr]^g
&&& D^-
}
\]
provides the desired  factorization of~\eqref{diagram_defining_a_morphism_in_caly}.
\end{proof}

\begin{proof}[Proof of Theorem~\ref{the:homotopy_pull_back_of_calx}]
Combine 
Lemma~\ref{lem:caly_and_cala}, Lemma~\ref{the:relative_fibration_sequences}, and Lemma~\ref{k'_is_homotopy_equivalence}.
\end{proof}


\section{Proof of Theorem~\ref{the:computing_K(calx)}}
\label{sec:Proof_of_Theorem_ref(computing_K(calx)}

\begin{notation}[Truncation for objects] \label{not:truncation_for_objects}
Let $A$ and $B$ be objects in $\cala$. Define for $a,b \in \IZ \amalg \{-\infty,\infty\}$
an object in $\cala^\kappa$ by
\[
A[a,b] = \bigoplus_{k = a}^b \Phi^{-k}(A)
\]
where $A[a,b]$ is defined to be zero if $a > b$ holds. 

Given a morphism $f \colon A \to B$ in $\cala_{\Phi}[t,t^{-1}]$ and $a_0,b_0,a_1,b_1$ 
in $\IZ \amalg \{-\infty,\infty\}$, define the $\cala^\kappa$  morphism $f\trun$
in $\cala$  to be the composite
\[
f\trun \colon A[a_0,b_0] \xrightarrow{i} A[-\infty,\infty] = i^0 A 
\xrightarrow{i^0f} i^0B  = B[-\infty,\infty] \xrightarrow{p} B[a_1,b_1],
\]
where $i$ is the obvious inclusion and $p$ the obvious projection.
\end{notation}

The morphism $f\trun \colon A[-\infty,\infty] \to B[-\infty,\infty]$ agrees with
$i^0f$ for a morphism $f \colon A \to B$ in $\cala_{\Phi}[t,t^{-1}]$.  If $f$ belongs to
$\cala_{\Phi}[t^{\pm 1}]$, we abbreviate $(j_{\pm} f)\trun$ by $f \trun$ again.

Notice that $(g \circ f)\trun$ is in general  \emph{not} equal to $g\trun \circ f\trun$
and $\id\trun$ is in general \emph {not} the identity. As a typical example, let $f\colon A\to \Phi(A)$ 
be the morphism $t=\id_{\Phi(A)}\cdot t$ and $g\colon \Phi(A)\to A$ be the morphism $t\inv=\id_A\cdot t\inv$. Then 
\[
(t\inv \circ t)\trun\colon A[-\infty,0]\to A[-\infty,0]
\]
is the identity while the map
\[
A[-\infty,0]=\bigoplus_{k= 0}^\infty \Phi^{k}(A) \xrightarrow{t\trun} (\Phi(A))[-\infty,0] = \bigoplus_{k= 1}^\infty \Phi^{k}(A)
\]
is the canonical projection and in particular is not a monomorphism. As another example,
\[
\id_A\trun\colon A[-\infty,\infty]=\bigoplus_{k=-\infty}^\infty A \to A[0,0]=A
\]
is just the projection map.

\begin{notation}[Truncation for chain complexes] \label{not:truncation_for_chain_complexes}
If $C^+$ is an $\cala_{\Phi}[t]$-chain complex and $a,b \in \IZ \amalg \{-\infty,\infty\}$,
then we obtain an $\cala^\kappa$-chain complex $C^+[a,b]$ by defining
the $n$-chain object to be $C^+_n[a,b]$ and the $n$-th differential 
to be  $c_n\trun \colon C_n^+[a,b] \to C_{n-1}^+[a,b]$ if $c_n$ is the differential of $C^+$.
(One has to check that $c_n\trun \circ c_{n+1} \trun = 0$.) 
A chain map $f \colon C^+ \to D^+$ of $\cala_{\Phi}[t]$-chain complexes induces a 
$\cala^\kappa$-chain map denoted by $f\trun \colon C^+_n[a,b] \to D^+_n[a',b']$
provided that $a'\leq a$ and $b'\leq b$.

If $C^-$ is an $\cala_{\Phi}[t^{-1}]$-chain complex and $a,b \in \IZ \amalg \{-\infty,\infty\}$,
define the $\cala^\kappa$-chain complex $C^-[a,b]$ analogously.
A chain map $f \colon C^- \to D^-$ of $\cala_{\Phi}[t^{-1}]$-chain complexes induces a 
$\cala^\kappa$-chain map denoted by $f\trun \colon C^-[a,b] \to D^-[a',b']$
provided that $a'\geq a$ and $b'\geq b$.
\end{notation}

Notice that Notation~\ref{not:truncation_for_chain_complexes}  
(in contrast to Notation~\ref{not:truncation_for_objects})
does in this generality \emph{not} make sense 
for chain complexes in $\cala_\Phi[t,t\inv]$, because of the lack of functoriality of truncation.

\begin{definition}[Global section functor]\label{def:global_section_functor}
  The \emph{global section functor} 
  \[
   \Gamma\colon \Chcat(\calx)\to\Chcat(\cala^\kappa)
   \] 
   sends an object $(C^+, f, C^-)$ to the $\cala^\kappa$-chain complex
  \[
   \Sigma^{-1}\cone\left(C^+[0,\infty] \oplus C^-[-\infty,0] \xrightarrow{(-f\trun,\id\trun)} C^-[-\infty,\infty]\right).
   \] 
   A morphism
  $(\varphi^+,\varphi^-) \colon (C^+, f, C^-)  \to (D^+, g, D^-)$ of $\Chcat(\calx)$ 
  is sent to the  morphism in $\Ch(\cala^\kappa)$ obtained by applying
  Lemma~\ref{lem:elementary_facts}~\ref{lem:elementary_facts:maps_between_mapping_cones} 
  to the commutative diagram (using the trivial homotopy)
  \[
  \xymatrix@!C= 11em{C^+[0,\infty] \oplus C^-[-\infty,0] \ar[r]^-{(-f\trun,\id\trun)} \ar[d]_{(\varphi^+\trun,\varphi^-\trun)}
   &
   C^-[-\infty,\infty] \ar[d]^{\varphi^-\trun}
   \\
   D^+[0,\infty] \oplus D^-[-\infty,0] \ar[r]_-{(-g\trun,\id\trun)} 
   &
   D^-[-\infty,\infty]
   }
   \]

\end{definition}

\begin{remark}[Comparison with global sections for modules]
  \label{rem:Comparison_with_global_sections_for_modules}
  Consider  the special case of modules over a ring $R$ with automorphism $\phi \colon R \to R$.
  The classical global section functor assigns to a triple $(M^+,f,M^-)$ consisting of a finitely
  generated free $R_{\phi}[t]$-module $M^-$, a finitely generated free
  $R_{\phi}[t^{-1}]$-module $M^-$ and an $R_{\phi}[t,t^{-1}]$-isomorphism 
  $f \colon j_+M^+   := R_{\phi}[t,t^{-1}] \otimes_{R_{\phi}[t]} M^+ 
  \xrightarrow{\cong} j_-M^- :=   R_{\phi}[t,t^{-1}] \otimes_{R_{\phi}[t^{-1}]} M^-$ 
  the finitely generated projective $R$-module
  \[
  \Gamma(f) := \{(a^+,a^-) \in i^+M^+\oplus i^-M^-) \mid f \circ j_+(a) = j_-(b)\}
  \]
  where $i^{\pm} M^{\pm}$ is the restriction to an $R$-module and 
  $j_{\pm} \colon M^{\pm}   \to R_{\phi}[t,t^{\pm 1}] \otimes_{R_{\phi}[t^{\pm 1}]} M^{\pm}$ is the obvious map. This
  can be rewritten as the kernel of the $R$-homomorphism
  \[
  i^+M^+ \oplus i^-M^-\xrightarrow{(-f \circ j_+, j_-)} i^0j_-M^-,
  \]
  where $i^0$ is again restriction to $R$.
  In the case of modules over  rings, global sections and its derived functors
  can be used to compute the $K$-theory of the projective 
  line~\cite[Theorem~3.1 in  Section~8.3 on page~59]{Quillen(1973)}. 
  In our situation, the above kernel might not
  exist since $\cala$ is not necessarily abelian, but we can replace it by the  mapping cone construction. 
  
  Such an idea and a similar strategy of proof has been used by
  \change{H\"uttemann}-Klein-Vogell-Waldhausen-Williams~\cite{Huettemann-Klein-Waldhausen-Williams(2001)}.
\end{remark}

Let $\Chcathf(\cala)\subset \Chcat(\cala^\kappa)$ be the full subcategory of homotopy
finite chain complexes, i.e., chain complexes over $\cala^{\kappa}$ which are homotopy
equivalent to a (bounded) chain complex over $\cala$. It follows from 
Lemma~\ref{lem:homotopy_finite_two_of_three} that this category is closed under pushouts along a
cofibration, so it is a Waldhausen subcategory of $\Chcat(\cala^\kappa)$. The
Approximation Theorem~\ref{the:Cisinki_Approximation_Theorem} shows that the inclusion
$\Ch(\cala)\to\Chhf(\cala)$ induces an equivalence on $K$-theory.

\begin{lemma} \label{lem:Gamma_values_in_CHhf}

\begin{enumerate}
   \item \label{lem:Gamma_values_in_CHhf:exact}
The functor $\Gamma$ is Waldhausen exact (for the canonical Waldhausen structures).

\item \label{lem:Gamma_values_in_CHhf:hf} Suppose that $\cala$ is idempotent complete. 
  Then for any object $(C^+, f, C^-)$ of
  $\Chcat(\calx)$, the chain complex $\Gamma(C^+,f,C^-)\in \Chcat(\cala^\kappa)$ is
  chain homotopy equivalent to an object in $\Chcat(\cala)$. Thus, $\Gamma$ defines a Waldhausen
  exact functor
\[
\Gamma\colon \Chcat(\calx)\to\Chcathf(\cala).
\]
\end{enumerate}
\end{lemma}

\begin{proof}~\ref{lem:Gamma_values_in_CHhf:exact} We showed in
  Section~\ref{sec:The_Gillet_Waldhausen_theorem} that the functors 
  $k^\pm\colon  \Ch(\calx)\to\Ch(\cala_\Phi[t^{\pm1}])$ are Waldhausen exact. The restriction functors
  from $\Chcat(\cala_\Phi[t])$, $\Chcat(\cala_\Phi[t^{-1}])$ and
  $\Chcat(\cala_\Phi[t,t^{-1}])$ to $\Chcat(\cala^\kappa)$ are defined on the level of
  additive categories and hence are Waldhausen exact. Taking cones and suspensions is also
  Waldhausen exact.  
  \\[1mm]~\ref{lem:Gamma_values_in_CHhf:hf} The following diagram of $\cala^{\kappa}$-chain complexes has exact rows
 \[
  \xymatrix{0 \ar[r] 
 & 
 C^-[-\infty,0] \ar[d]^{\id}_{\cong} \ar[r]^-{\tiny \bigl(\begin{matrix} 0 \\ \id \end{matrix}\bigr)}
 & 
 C^+[0,\infty] \oplus C^-[0,\infty] \ar[d]^{(-f\trun,\id\trun)} \ar[r]^-{(\id,0)}
 & 
 C^+[0,\infty ]\ar[d]^{-f\trun} \ar[r] 
 & 0
 \\
 0 \ar[r] 
 & 
 C^-[-\infty,0] \ar[r]_-{\id\trun}
 & 
 C^-[-\infty,\infty] \ar[r]_-{\id\trun}
 & C^-[1,\infty] \ar[r] 
 & 0
 }
\]
 We conclude from Subsection~\ref{subsec:Homotopy_fiber_sequences}
 \begin{eqnarray}
 \Gamma(C^+,f,C^-) &\simeq &\Sigma^{-1}\cone\bigl(-f\trun\colon C^+[0,\infty]\to C^-[1,\infty]\bigr).
 \label{Gamma_simeq_cone(-ftrun)}
 \end{eqnarray}
 
 Write $f_n\inv = \sum_{k \in \IZ} a_{n,k} \cdot  t^k$. Now choose a natural number $N$ such that 
 we have $a_{n,k} = 0$ for all $\lvert k\rvert \ge N$ and all $n$. 
 Then $f\inv\trun$ factors through
\[
\xymatrix{C^-[-\infty,\infty] \ar[r]^{f\inv\trun} \ar[d]^{\id\trun} & C^+[N,\infty]\\
 C^-[1,\infty] \ar[ru]_{\overline{f\inv}}
}
\] 
 and the composite
 \[
 C^+[N,\infty] \xrightarrow{f\trun} C^-[1,\infty] \xrightarrow{\overline{f^{-1}}} C^+[N,\infty]
 \]
 is the identity map.
 
  Hence in $\Idem(\cala^{\kappa})$, the chain complex $C^-[1,\infty]$ splits as
 \[
 C^-[1,\infty] \cong C^+[N,\infty] \oplus R.
 \]
 We argue that $R$ is actually isomorphic to a chain complex in $\cala$.
 In fact, denote by $r\colon C^-[1,\infty]\to R$ the projection and by $i\colon R\to
 C^-[1,\infty]$ the inclusion. The composite
 \[
 C^-[2N,\infty] \xrightarrow{f^{-1}\trun} C^+[N,\infty] \xrightarrow{f\trun} C^-[2N,\infty]
 \]
 is the identity, which shows that the restriction of $r$ onto $C^-[2N,\infty]$ is zero. Hence 
 $r$ factors as
 \[
 C^-[1,\infty] \xrightarrow{p}  C^-[1,2N-1] \xrightarrow{r'} R.
 \]
 The $\Idem(\cala^{\kappa})$-isomorphism
 \[
 \bigl(C^-[1,\infty], ir\bigr) \xrightarrow{pir} \bigl(C^-[1,2N-1] , pir'\bigr)
 \] 
 (with inverse $ir'$) now shows that $R=(C^-[1,\infty], ir)$ is isomorphic to an
 object in $\Idem(\cala)$, hence in $\cala$, since $\cala$ is idempotent complete. So $R$
 is isomorphic to a $\cala$-chain complex, as we claimed.
 
 Since $\cala^{\kappa}$ is a full subcategory of $\Idem(\cala^{\kappa})$,
 we obtain a exact sequence of $\cala^{\kappa}$-chain complexes
 \[
  \xymatrix{0 \ar[r] 
 & 
 C^+[N,\infty] \ar[r] \ar[d]^{\id}
 &  C^+[0,\infty] \ar[d]^{-f\trun} \ar[r] 
 & 
 C^+[0,N-1]\ar[d]^{g} \ar[r] 
 & 0
 \\
 0 \ar[r] 
 & C^+[N,\infty] \ar[r]^{-f\trun} 
 & C^-[1,\infty] \ar[r]^r 
 & R  \ar[r] 
 & 0
 }\]
 where $g$ is the induced map on the quotients. It shows that $\Sigma^{-1}\cone(-f\trun) \simeq
 \Sigma^{-1}\cone(g)$ which is isomorphic to a chain complex in $\cala$. Hence $\Gamma(C^+,f,C^-)$ belongs to
 $\Chcathf(\cala)$ because of~\eqref{Gamma_simeq_cone(-ftrun)}.
 \end{proof}

Recall that the automorphism $\Phi \colon \cala \to \cala$ extends to an automorphism,
denoted by the same letter, $\Phi \colon \cala_{\Phi}[t,t^{-1}] \to \cala_{\Phi}[t,t^{-1}]
$ by sending a morphism $\sum_{k= -\infty}^{\infty} g_k \cdot t^k \colon A \to B$ to
$\sum_{k= -\infty}^{\infty} \Phi(g_k) \cdot t^k \colon \Phi(A) \to \Phi(B)$. It induces
automorphisms $\Phi \colon \cala_{\Phi}[t^{\pm 1}] \to \cala_{\Phi}[t^{\pm 1}]$. In particular
we get for an $\cala_{\Phi}[t]$-chain complex $C$ a new $\cala_{\Phi}[t]$-chain complex
$\Phi^{-1}(C)$, an $\cala_{\Phi}[t]$-chain map $t \colon \Phi^{-1}(C) \to C$ and a
$\cala_{\Phi}[t,t^{-1}]$-chain map $t \colon \Phi^{-1}(j_+C) = j_+ \Phi^{-1}(C) \to j_+C$.

Denote by $s \colon \calx \to \calx$  the additive functor which sends the object 
$(C^+, f, C^-)$ to $(\Phi^{-1}(C^+), f\circ t , C^-)$ and the morphism $(\varphi^+,\varphi^-)$ 
to $(\Phi^{-1}(\varphi^+), \varphi^-)$.  This is well-defined since 
$j_+\Phi^{-1}(\varphi^+) = \Phi\inv(j_+\varphi^+) = t^{-1} \circ j_+(\varphi^+) \circ t$ holds in
$\cala_{\Phi}[t,t^{-1}]$. Recall that $l_0\colon \cala\to \calx$ was defined in~\eqref{def_of_l_i} to send $A$ to $(A, \id, A)$.
Put
\begin{align*}
l_i := s^i \circ l_0\colon & \Chcat(\cala)\to\Chcat(\calx);\\
   \Gamma_i := \Gamma\circ s^{-i}\colon & \Chcat(\calx)\to \Chcathf(\Idem(\cala)).
\end{align*}
Denote by $(\Chcat(\calx),w_i)$ the Waldhausen category with
underlying category the category of bounded chain complexes over
$\calx$ and the usual cofibrations, but with a new category
of weak equivalences, namely, the one consisting of those chain maps that become a weak
equivalence after applying $\Gamma_i$.  Notice that $w_0 \cap w_1$
contains all chain homotopy equivalences so that $\Chcat(\calx)$ is a
Waldhausen subcategory of $(\Chcat(\calx),w_0 \cap w_1)$.

\begin{lemma}\label{lem:comparison_of_weak_equivalences}
The map induced by inclusion of Waldhausen categories
\[
 \bfK(\Chcat(\calx))\to \bfK(\Chcat(\calx), w_0\cap w_1)
\]
is a homotopy equivalence.
\end{lemma}

\begin{proof}
  Let $v$ be the standard structure of weak equivalences in $\Ch(\calx)$.  We will show
  that if an object $(C^+,f,C^-)$ of $\Chcat(\calx)$ is $(w_0\cap w_1)$-acyclic, then $C^+$ is
  contractible in $\cala_\Phi[t]$ and $C^-$ is contractible in $\cala_\Phi[t^{-1}]$, so that
  $(C^+,f,C^-)$ is $v$-acyclic. This statement implies that $\bfK(\Chcat(\calx)^{w_0\cap w_1},v)$ 
  is contractible, from which the Lemma follows by the 
  Fibration  Theorem~\ref{the:Fibration_Theorem}.

  First we want to show that the $\cala^{\kappa}$-chain complex $C^-[-\infty,0]$ is
  contractible. The following diagram of $\cala^{\kappa}$-chain complexes commutes
  \[
  \xymatrix@!C=11em{
    C^+[0,\infty] \oplus C^-[-\infty,0] \ar[r]^-{(-f\trun,\id\trun )} \ar[d]_{\id\trun \oplus \id}
   &
   C^-[-\infty,\infty]
  \\
   C^+[-1,\infty]\oplus C^-[-\infty,0] \ar[ru]_-{\quad (-f\trun,\id\trun)}
   &
   }
   \]
   There is an obvious identification of $C[-1,\infty]$ with $\Phi(C)[0,\infty]$. Under
   this identification the map 
   $(-f\trun,\id\trun) \colon C[-1,\infty] \oplus C^-[-\infty,0] \to C^-[-\infty,\infty]$ becomes
   the map $(-f \circ t^{-1})\trun,\id\trun]) \colon \Phi(C)[0,\infty] \oplus C^-[-\infty,0] 
  \to C^-[-\infty,\infty]$. Hence the
   mapping cone of the lower horizontal arrow agrees with the mapping cone appearing in the
   definition of $\Gamma_1(C^+,f,C^-)$. The mapping cone of the  horizontal map is
   the mapping cone appearing in the definition of $\Gamma_0(C^+,f,C^-)$.  Since
   $(C^+,f,C^-)$ of $\Chcat(\calx)$ is $(w_0\cap w_1)$-acyclic by assumption, the mapping
   cone of both the upper horizontal and the lower horizontal arrow are contractible. We
   conclude from Lemma~\ref{lem:elementary_facts}~\ref{lem:elementary_facts:chain_homotopy_equivalence_cone}  
   and~\ref{lem:elementary_facts:chain_homotopy_equivalence_2_of_3}
   that the inclusion $\id\trun \colon C^+[0,\infty] \to C^+[-1,\infty]$ is an $\cala^{\kappa}$-chain equivalence.

   If we apply $\Phi^n$ for $n \in \IZ, n\ge 0$ to the inclusion above and use the obvious
   identifications $\Phi^n(C^+[0,\infty]) = C^+[-n,\infty]$ and $\Phi^n(C^+[-1,\infty]) =
   C^+[-n-1,\infty]$, we conclude that also the inclusion $\id\trun \colon
   C^+[-n+1,\infty] \to C^+[-n,\infty]$ is an $\cala^{\kappa}$-chain equivalence.  Hence
   the inclusion $\id\trun \colon C^+[0,\infty] \to C^+[-n,\infty]$ is a
   $\cala^\kappa$-chain homotopy equivalence for every $n \in \IZ, n \ge 0$.

   Next we want to show that $\id\trun \colon C^+[0,\infty] \to C^+[-\infty,\infty]$ is a
   $\cala^\kappa$-chain homotopy equivalence. Since the inclusion $\id\trun \colon
   C^+[0,\infty] \to C^+[-n,\infty]$ is levelwise split injective, the canonical
   projection from its mapping cone to $C^+[-n,\infty]/C^+[0,\infty]$ is a
   $\cala^\kappa$-chain homotopy equivalence by
   Lemma~\ref{lem:elementary_facts}~\ref{lem:elementary_facts:chain_homotopy_equivalence_2_of_3}.
   Since $\id\trun \colon C^+[0,\infty] \to C^+[-n,\infty]$ is a $\cala^\kappa$-chain
   homotopy equivalence,
   Lemma~\ref{lem:elementary_facts}~\ref{lem:elementary_facts:chain_homotopy_equivalence_cone}
   implies that the $\cala^\kappa$-chain complex $C^+[-n,\infty]/C^+[0,\infty]$ is
   contractible.

  Because of Lemma~\ref{lem:elementary_facts}~\ref{lem:elementary_facts:exact_sequence_and_contraction} 
  and~\ref{lem:elementary_facts:chain_homotopy_equivalence_2_of_3} 
  we can  find for $n \in \IZ, n \ge 0$ chain contractions $\gamma_*[-n]$ for $C^+[-n,\infty]/C^+[0,\infty]$,
  such that $\gamma_*[-n]$ and $\gamma_*[-n-1]$ are compatible with the inclusion
  $C^+[-n,\infty]/C^+[0,\infty] \to C^+[-n-1,\infty]/C^+[0,\infty]$. By inspecting the definitions of
  the various chain modules as direct sums, one sees that the $\cala$-chain complex
   $C^+[-\infty,\infty]/C^+[0,\infty]$ is the colimit  $\colim_{n \to \infty} C^+[-n,\infty]$ within the category of
   $\cala$-chain complexes. Hence $C^+[-\infty,\infty]/C^+[0,\infty]$
    inherits a chain contraction from the various chain contractions
   $\gamma_*[-n]$. We conclude from 
   Lemma~\ref{lem:elementary_facts}~\ref{lem:elementary_facts:chain_homotopy_equivalence_cone}
   and~\ref{lem:elementary_facts:chain_homotopy_equivalence_2_of_3} that
    $\id\trun \colon C^+[0,\infty] \to C^+[-\infty,\infty]$ is a $\cala^\kappa$-chain homotopy equivalence.

   The following diagram commutes
   \[
   \xymatrix@!C=11em{
    C^+[0,\infty] \oplus C^-[-\infty,0] \ar[r]^-{(-f\trun,\id\trun )}_\simeq \ar[d]_{\id\trun \oplus \id}
   &
   C^-[-\infty,\infty]
  \\
   C^+[-\infty,\infty]\oplus C^-[-\infty,0] \ar[ru]_-{\quad (-f\trun,\id\trun)}
   &
   }
   \]
   Since the vertical and the horizontal arrow are $\cala$-chain homotopy equivalences,
   $(-f,\id\trun) \colon C^+[-\infty,\infty]\oplus C^-[-\infty,0] \to C^-[-\infty,\infty]$ is a $\cala$-chain homotopy
  equivalence. Since $f \colon C^+[-\infty,\infty]  \to C^-[-\infty,\infty]$ is a $\cala$-chain homotopy equivalence,
  we conclude from 
  Lemma~\ref{lem:elementary_facts}~\ref{lem:elementary_facts:chain_homotopy_equivalence_2_of_3} 
  that $C^-[-\infty,0]$ is contractible.

   Analogously one proves that the $\cala^{\kappa}$-chain complex
   $C^+[0,\infty]$ is contractible, Namely,  choose a $\cala_{\Phi}[t,t^{-1}]$-chain homotopy
   inverse $f^{-1}$ of $f$ and consider the triple
   $(C^-,f^{-1},C^+)$ and conclude from the assumption
   that the object $(C^+,f,C^-)$ of $\Chcat(\calx)$ is $(w_0\cap w_1)$-acyclic
   that  the mapping cones of the  $\cala^{\kappa}$-chain maps
  \begin{eqnarray*}
  (-f^{-1} \trun,\id\trun]) \colon C^-[-\infty,0] \oplus C^+[0,\infty] 
  & \to & 
   C^+[-\infty,\infty]; 
   \\
   (-f^{-1}\trun,\id\trun) \colon C^-[-\infty,1] \oplus C^+[0,\infty] 
   & \to &
   C^+[-\infty,\infty],
 \end{eqnarray*}
   are contractible.   

  Finally we conclude from Lemma~\ref{lem:contractible_over_cala_Phi[t,t(-1)]_versus_over_cala} that
  $C^+$ is contractible in $\cala_\Phi[t]$ and $C^-$ is contractible in $\cala_\Phi[t^{-1}]$.
  This finishes the proof of Lemma~\ref{lem:comparison_of_weak_equivalences}.
 \end{proof}

 Observe that for all $i$ we have $\Gamma_{i-1}\circ l_i\simeq *$. In fact,
 $\Gamma_{i-1}\circ l_i(C)=\Gamma\circ s\circ l_0(C)$ is, up to a suspension, the cone of
 the chain isomorphism
\[
\id\trun \oplus \id\trun \colon C[1,\infty] \oplus C[-\infty,0] \xrightarrow{\cong} C[-\infty,\infty]
\] 
and therefore contractible.  In particular, $l_i$ induces a functor
\[
\bfK(l_i)\colon \bfK(\Chcat(\cala))\to \bfK(\Chcat(\calx^{w_{i-1}}, w_i)),
\]
where $\Chcat(\calx^{w_{i-1}}, w_i)$ is the full Waldhausen subcategory of $(\Chcat(\calx), w_i)$ of those
$\calx$-chain complexes which are $w_{i-1}$-acyclic.
\begin{lemma}\label{lem:comparison_of_K_theories}
  Suppose that $\cala$ is idempotent complete. For any $i$, the maps
\begin{eqnarray*}
\bfK(\Chcat(\cala)) 
&\to &
\change{\bfK(\Chcat(\calx^{w_{i-1}}), w_i) \quad \mathrm{and}}
\\
\bfK(\Chcat(\cala))
& \to & 
\bfK(\Chcat(\calx), w_i)
\end{eqnarray*}
 induced by $l_i$ are weak equivalences.
\end{lemma}

\begin{proof}
  Since $s^i \colon (\calx,w_0) \to (\calx,w_i)$ is an isomorphism of Waldhausen
  categories, it suffices to treat the case $i = 0$.  The proof consists in showing that
  the functors $l_0$ and $\Gamma$ are mutually inverse up to homotopy in a suitable
  sense. To make this precise, denote by $\widehat{\calx}$ the Waldhausen category 
where an object is a triple $(C^+, f, C^-)$ with $C^+$ and $C^-$ chain complexes in 
$\cala_\Phi[t]^\kappa$ and $\cala_\Phi[t\inv]^\kappa$, respectively, and $f$ is a  
$\cala_\Phi[t,t\inv]^\kappa$-chain equivalence $j_+(C^+)\to j_-(C^-)$. Morphisms and 
the Waldhausen structure structure are defined analogously as for $\Ch(\calx)$.
  
  Note that both functors $\Gamma$ and $l_0$ extend (by
  the same formula) to functors
  \[
  \Ch(\cala^\kappa) \overset{l_0}{\underset{\Gamma}{\rightleftarrows}}
  \widehat{\calx}.
  \]
  We provide the right-hand side with the category $w_0$ of weak equivalences, i.e., a
  morphism is a weak equivalence if it becomes one after applying $\Gamma$. Then both
  functors are Waldhausen exact.

  We now claim that these are mutually inverse weak equivalences of Waldhausen categories,
  i.e., both composites are related by a zigzag of natural weak equivalences to the
  respective identity functors.

  To verify our claim, we first define for $C\in\Ch(\cala^\kappa)$ a chain homotopy equivalence
  of $\cala^{\kappa}$-chain complexes, natural in $C$
\begin{multline}
T(C) \colon C \to \Gamma \circ l_0(C) 
= \Sigma^{-1}\cone\bigl((-\id\trun ,\id\trun) \colon C[0,\infty] \oplus C[-\infty,0] 
\\
\to C[-\infty,\infty]\bigr)
\label{T(C)}
\end{multline}
by  the short exact sequence of $\cala^{\kappa}$-chain complexes
\[
0 \to  C = C[0,0] \xrightarrow{\begin{pmatrix} \id\trun \\ \id\trun \end{pmatrix}}
C[0,\infty] \oplus C[-\infty,0] \xrightarrow{(-\id\trun,\id\trun)} C[-\infty,\infty] \to 0.
\]
using Lemma~\ref{lem:elementary_facts}~\ref{lem:elementary_facts:maps_between_mapping_cones}.
The chain map $T(C)$ is a chain homotopy equivalence by
Subsection~\ref{subsec:Homotopy_fiber_sequences}.  
This provides a natural weak equivalence $T\colon \id\to \Gamma\circ l_0(C)$.

Now consider an object $(C^+,f,C^-)$ in $\widehat{\calx}$. In the sequel  we abbreviate
$\Gamma(f) := \Gamma(C^+,f,C^-)$. Recalling that $\Gamma(f)$ is the
$\cala^{\kappa}$-chain complex
\[
\Gamma(f)  := \Sigma^{-1}\cone\bigl((-f\trun ,\id\trun) \colon C^+[0,\infty] \oplus C^-[-\infty,0] 
\to C^-[-\infty,\infty]\bigr),
\]
we obtain from 
Lemma~\ref{lem:elementary_facts}~\ref{lem:elementary_facts:maps_between_mapping_cones}
the following (not necessarily commutative) diagram of $\cala^{\kappa}$-chain
complexes which commutes up to a preferred chain homotopy 
$h \colon f\trun \circ \id\trun \circ \varphi^+ \simeq \id \trun \circ \varphi^- \circ \id$:
\[
\xymatrix@!C=8em{
\Gamma(f) \ar[d]_{\varphi^+}  \ar[r]^{\id}
&
\Gamma(f)    \ar[d]^{\varphi^-}
\\
C^+[0,\infty] \ar[d]_{\id\trun} 
&
C^-[-\infty,0] \ar[d]^{\id\trun} 
\\
C^+[-\infty,\infty] \ar[r]_{f\trun}^{\simeq}
&
C^-[-\infty,\infty]
}
\]

>From the adjunctions  in Lemma~\ref{lem:Adjunction_between_induction_and_restriction}
we obtain an $\cala_{\Phi}[t]^\kappa$-chain map 
\[
\psi^+ \colon i_+\Gamma(f) \to C^+
\] 
from $\varphi^+$, an $\cala_{\Phi}[t^{-1}]^\kappa$-chain map 
\[
\psi^- \colon i_-\Gamma(f) \to C^-
\]
from $\varphi^-$, and a homotopy of $\cala_{\Phi}[t,t^{-1}]^\kappa$-chain maps 
\[
H \colon f \circ j_+\psi^+ \simeq j_-\psi^-.
\]
from $h$.

Let $D$ be the mapping cylinder of the $\cala_{\Phi}[t]^\kappa$-chain map
$\psi^+ \colon i_+\Gamma(f) \to C^+$. Denote by $u \colon i_+\Gamma(f) \to D$ 
and $v \colon C^+ \to D$ the canonical inclusions and by $p \colon D \to C^+$ the
canonical projection. Notice that $j_+D$ can be identified with the mapping cylinder of
$j_+\psi^+ \colon i_0\Gamma(´f) \to j_+C^+$. Because of
Lemma~\ref{lem:elementary_facts}~\ref{lem:elementary_facts:maps_between_mapping_cylinders}  
we obtain from $H$ a
$\cala_{\Phi}[t,t^{-1}]^\kappa$-chain map $f' \colon j_+D \to j_-C^-$ so that the following
diagram of $\cala_{\Phi}[t,t^{-1}]^\kappa$-chain complexes commutes (strictly):
\[
\xymatrix{i_0\Gamma(f) \ar[r]^{\id} \ar[d]_{j_+u}
&
i_0\Gamma(f)  \ar[d]^{j_-\psi^-}
\\
j_+D \ar[r]^{f'}
&
j_- C^- 
\\
j_+C^+ \ar[u]^{j_+v} \ar[r]^{f}
&
j_-C^- \ar[u]_{\id}
}
\]
Thus  we get morphisms 
\begin{align}
(u,\psi^-) \colon \bigl(i_+\Gamma(f),\id, i_-\Gamma(f)\bigr) & \to  (D,f',C^-);
\label{(u,psi-)}
\\
(v,\id) \colon (C^+,f,C^-) & \to  (D,f',C^-),
\label{(v,id)}
\end{align}
in $\widehat{\calx}$.

The morphism~\eqref{(v,id)} is a $w_0$-equivalence as $(v,\id)$ is a weak equivalence in
$\widehat{\calx}$ and $\Gamma$ is Waldhausen exact. To show that \eqref{(u,psi-)}
is a also $w_0$-equivalence, note first that the following diagram in
$\cala_\Phi[t,t\inv]^\kappa$ commutes up to a canonical homotopy $K$, which comes from the
homotopy $\id_D\simeq v\circ p$, corresponding to the collapse of the mapping cylinder to
its bottom:
\[
\xymatrix{j_+D \ar[d]^{j^+p} \ar[r]^{f'} & j_-C^- \ar[d]^\id\\
j_+C^+ \ar[r]^f & j_-C^-
}
\]

Abbreviating $\Gamma(f') := \Gamma(D,f',C^-)$,
Lemma~\ref{lem:elementary_facts}~\ref{lem:elementary_facts:maps_between_mapping_cones}
provides us therefore a canonical map $\mu\colon \Gamma(f')\to \Gamma(f)$ which splits
$\Gamma(j_+v, \id)$ and hence is a weak equivalence. Then it follows from the definitions that the composite
\[
\Gamma(f) \underset{\simeq}{\xrightarrow{T(\Gamma(f))}} \Gamma(i_+\Gamma(f), \id, i_-\Gamma(f)) 
\xrightarrow{\Gamma(u,\psi^-)} \Gamma(f') \underset{\simeq}{\xrightarrow{\mu}} \Gamma(f)
\]
is just the identity map. Hence \eqref{(u,psi-)} is a $w_0$-equivalence. 

Thus the two maps
\eqref{(u,psi-)} and \eqref{(v,id)} provide a zigzag of natural weak equivalences between
the identity functor and $l_0\circ\Gamma$. This proves our claim that $l_0$ and $\Gamma$ are mutually inverse natural
weak equivalences $\Ch(\cala^\kappa) \overset{l_0}{\underset{\Gamma}{\rightleftarrows}}
  \widehat{\calx}$.

  Now denote by $\hatxhf\subset \widehat{\calx}$ the full Waldhausen subcategory on
  objects $(C^+, f, C^-)$ such that $C^+\in \Chhf(\cala_\Phi[t])$ and
  $C^-\in\Chhf(\cala_\Phi[t\inv])$. The functors $l_0$ and $\Gamma$ restrict to mutually
  inverse natural weak equivalences
  \[
  \Chhf(\cala) \overset{l_0}{\underset{\Gamma}{\rightleftarrows}} \hatxhf.
  \]

  It follows that there are mutually inverse equivalences of spectra
  \[
  \bfK(\Chhf(\cala)) \overset{\bfK(l_0)}{\underset{\bfK(\Gamma)}{\rightleftarrows}}
  \bfK(\hatxhf).
  \] 
   By the Approximation Theorem~\ref{the:Cisinki_Approximation_Theorem}
  the map $\bfK(\Ch(\cala))\to\bfK(\Chhf(\cala))$ induced by the inclusion is an
  equivalence. To show that similarly $\bfK(\Ch(\calx))\xrightarrow{\simeq}
  \bfK(\hatxhf)$, we apply Cisinski's version of Approximation Theorem twice:

  Denote by $\hatxf\subset \hatxhf$ the full Waldhausen subcategory on objects $(C^+, f,
  C^-)$ such that both $C^+$ and $C^-$ are finite chain complexes in $\cala_\Phi[t]$ and
  $\cala_\Phi[t\inv]$ respectively. (It differs from $\Ch(\calx)$ in that $f$ is required
  to be a chain homotopy equivalence, rather than an isomorphism.) Given such an object
  $(C^+, f, C^-)$, by Lemma~\ref{lem:elementary_facts}~\ref{lem:elementary_facts:elementary} there are elementary
  chain complexes $E^+$ and $E^-$ in $\cala$ and a chain isomorphism $F\colon j_+ C^+
  \oplus i_0 E^+ \xrightarrow{\cong} j_- C^-\oplus i_0 E^-$ in $\cala_\Phi[t,t\inv]$ such
  that
\[
f= p\circ F\circ i\colon j_+ C^+ \xrightarrow{i} j_+ C^+ \oplus i_0 E^+ \xrightarrow{F} j_- C^- \oplus i_0 E^- \xrightarrow{p} j_- C^-,
\]
where $i$ and $p$ are the inclusion and projection, respectively. (Note that any
elementary chain complex in $\cala_\Phi[t,t\inv]$ is of the form $i_0 E$.)

Now let $(D^+, g, D^-)$ be an object of $\Ch(\calx)$ and $a=(a^+, a^-)\colon (D^+, g,
D^-)\to (C^+, f, C_-)$ a morphism in $\hatxf$. Postcomposing $F$ by
\[
\id\oplus \id\cdot t^{-n}\colon C^-\oplus i_0 E^-\to C^-\oplus \Phi^{-n} i_0 E^-
\]
for large enough $n$ we may assume that in the composite map
\[
j_- D^- \xrightarrow{g\inv} j_+ D^+\xrightarrow{i\circ a^+}  j_+ C^+ \oplus i_0 E^+ \xrightarrow{F} j_- C^- \oplus i_0 E^-
\]
no positive powers of $t$ appear, i.e., is of the form $j_-(h)$. Then the commutative diagram
\[
\xymatrix{(D^+, g, D^-) \ar[rr]^{(a^+, a^-)} \ar[d]^{(i\circ a^+, h)}
 && (C^+, f, C^-) \ar[d]^{(i,\id)}_\simeq 
\\
(C^+\oplus i_+ E^+, F, C^-\oplus i_- E^-) \ar[rr]^{(\id,p)}_\simeq
 && (C^+\oplus i_+ E^+, p\circ F, C^-)
}
\]
shows that the assumptions of Cisinski's version of the Approximation 
Theorem~\ref{the:Cisinki_Approximation_Theorem}  are
satisfied. Hence $\bfK(\Ch(\calx))\xrightarrow{\simeq} \bfK(\hatxf)$.

Now let $(D^+, g, D^-)$ be an object of $\hatxf$ and $a=(a^+, a^-)\colon (D^+, g, D^-)\to
(C^+, f, C_-)$ a morphism in $\hatxhf$. By assumption there exist a
$\cala_\Phi[t\inv]^\kappa$-chain equivalence $\psi\colon D^-\to Z^-$ where $Z^-$ is in
$\Ch(\cala_\Phi[t\inv])$. Moreover the assumptions imply that there is a factorization of
$a^+$ into
\[
D^+ \xrightarrow{z^+} Z^+ \underset{\simeq}{\xrightarrow{\phi}} C^+
\]
where $\phi$ is a $\cala_\Phi[t]^\kappa$-chain equivalence and $Z^+$ is in $\Ch(\cala_\Phi[t])$. 
Then the commutative diagram
\[
\xymatrix{(D^+, g, D^-) \ar[rr]^{(a^+, a^-)} \ar[d]^{(z^+, \psi\circ a^-)}
 && (C^+, f, C^-) \ar[d]^{(\id,\psi)}_\simeq 
\\
(Z^+, j_-\psi \circ f\circ j_+\phi, Z^-) \ar[rr]^{(\phi,\id)}_\simeq
 && (C^+, j_-\psi\circ f, Z^+)
}
\]
shows that the assumptions of Cisinski's version of the Approximation 
Theorem~\ref{the:Cisinki_Approximation_Theorem}  are satisfied. 
Hence $\bfK(\hatxf)\xrightarrow{\simeq} \bfK(\hatxhf)$.

This concludes the proof that 
\[
\bfK(l_0)\colon \bfK(\Ch(\cala))\to \bfK(\Ch(\calx), w_0)
\]
is an equivalence of spectra. To show that
\[
\bfK(l_0) \colon \bfK\bigl(\Chcat(\cala)\bigr) \to \bfK\bigl(\Chcat(\calx^{w_{-1}}),w_0\bigr)
\]
is a weak equivalence as well, apply the same argument, but replacing $\hatxhf$ and
$\hatxf$ by their full Waldhausen subcategories of $w_{-1}$-acyclic objects.
\end{proof}

 Finally, we can finish the proof of Theorem~\ref{the:computing_K(calx)}.

\begin{proof}[Proof of Theorem~\ref{the:computing_K(calx)}]
  By the Fibration Theorem~\ref{the:Fibration_Theorem} there is a fibration sequence
  \[
  \bfK(\Chcat(\calx^{w_0}), w_1)\to \bfK(\Chcat(\calx), w_0\cap w_1) \to
  \bfK(\Chcat(\calx), w_0))
  \]
  where by Theorem~\ref{the:Gillet_Waldhausen} and
  Lemma~\ref{lem:comparison_of_weak_equivalences} the middle term agrees with
  $\bfK(\calx)$.  By Lemma~\ref{lem:comparison_of_K_theories}, $\bfK(\cala)$ is homotopy
  equivalent to both the left-hand and the right-hand side, using the functors $l_1$ and
  $l_0$ respectively. The fibration sequence splits as $l_0$ factors through the middle
  term.
\end{proof}


\section{Strategy of proof for 
Theorem~\ref{the:BHS_decomposition_for_connective_K-theory}%
~\textnormal{\ref{the:BHS_decomposition_for_connective_K-theory:Nil}}}
\label{sec:Strategy_of_proof_for_the_connective_version_NIL}

In this section we present the details of the formulation and then the basic strategy of
proof of Theorem~\ref{the:BHS_decomposition_for_connective_K-theory}%
~\ref{the:BHS_decomposition_for_connective_K-theory:Nil}.

\begin{definition}[Nilpotent morphisms and Nil-categories] 
\label{def:Nilpotent_morphisms_and_Nil-categories}
  Let $\cala$ be an additive category and $\Phi$ be an automorphism of
  $\cala$.
  \begin{enumerate}
  \item A morphism $f\colon \Phi(A)\to A$ of $\cala$ is called \emph{$\Phi$-nilpotent} if
    for some $n \ge 1 $, the $n$-fold composite
    \[
    f^{(n)}:=f \circ \Phi(f) \circ \cdots \circ \Phi^{n-1}(f) \colon \Phi^n(A)\to A.
    \]
    is trivial;
 
  \item The category $\Nil(\cala, \Phi)$ has as objects pairs $(A, \phi)$ where $\phi\colon
    \Phi(A)\to A$ is a $\Phi$-nilpotent morphism in $\cala$. A morphism from $(A, \phi)$ to
    $(B, \mu)$ is a morphism $u\colon A\to B$ in $\cala$ such that the following diagram is
    commutative:
    \[
    \xymatrix{{\Phi(A)} \ar[rr]^{\phi} \ar[d]^{\Phi(u)} && A \ar[d]^u\\
      {\Phi(B)} \ar[rr]^{\mu} && B }
     \]
  \end{enumerate}
\end{definition}

The category $\Nil(\cala, \Phi)$ inherits the structure of an exact category from $\cala$,
a sequence in $\Nil(\cala,\Phi)$ is declared to be exact if the underlying sequence in $\cala$ is (split) exact.

There is a functor
\[
\chi\colon \Nil(\cala,\Phi) \to \Ch(\cala_\Phi[t\inv])
\]
sending $\phi\colon \Phi(A)\to A$ to the 1-dimensional chain complex
$A\xrightarrow{t\inv-i_-\phi} \Phi(A)$. (See
section~\ref{sec:On_the_Nil_category_(Version} for more details.) Using the
Gillet-Waldhausen Theorem~\ref{the:Gillet_Waldhausen}, this leads to a map
\[
\bfK(\chi)\colon \bfK(\Nil(\cala,\Phi))\to \bfK(\cala_\Phi[t]).
\]

The key ingredient in the proof Theorem~\ref{the:BHS_decomposition_for_connective_K-theory}%
~\ref{the:BHS_decomposition_for_connective_K-theory:Nil} is the following theorem
whose proof is deferred to Section~\ref{sec:On_the_Nil_category_(Version}.

\begin{theorem}[Fiber sequence for the Nil]
\label{the:fiber_sequence}
Suppose that $\cala$ is idempotent complete. 
The following is a homotopy fiber sequence, natural in $(\cala,\Phi)$:
\[
\bfK(\Nil(\cala,\Phi)) \xrightarrow{\bfK(\chi)} \bfK(\cala_{\Phi}[t^{-1}])
\xrightarrow{\bfK(j_-)} \bfK(\cala_{\Phi}[t,t^{-1}]).
\]
\end{theorem}

In the remainder of this section we explain how Theorem~\ref{the:BHS_decomposition_for_connective_K-theory}%
~\ref{the:BHS_decomposition_for_connective_K-theory:Nil} follows from 
Theorem~\ref{the:fiber_sequence}.  Define spectra 
\begin{eqnarray*}
\bfE_0(\cala,\Phi) 
& := & 
\hofib\bigl(\bfK(i_+) \colon \bfK(\cala)  \to \bfK(\cala_{\Phi}[t]) \bigr);
\\
\bfE_1(\cala,\Phi) 
& := & 
\hofib\bigl(\bfK(i_+ \circ \Phi^{-1}) \vee \bfK(i_+) \colon \bfK(\cala) \vee \bfK(\cala)   \to \bfK(\cala_{\Phi}[t]) \bigr);
\\
\bfE_2(\cala,\Phi) 
& := & 
\hofib\bigl(\bfK(\cala_{\Phi}[t^{-1}])\xrightarrow{\bfK(j_-)} \bfK(\cala_{\Phi}[t,t^{-1}])\bigr).
\end{eqnarray*}
The inclusion to the second summand $\bfK(\cala)  \to  \bfK(\cala) \vee \bfK(\cala)$ induces a map of spectra
\[
\bfw \colon \bfE_0 \to \bfE_1
\]
and the projection onto the first summand  $\bfK(\cala) \vee \bfK(\cala)   \to \bfK(\cala)$ induces 
a map of spectra
\[
\bfx \colon \bfE_1 \to \bfK(\cala),
\]
such that the following is a fibration sequence of spectra
\begin{equation}\label{eq:fibration_sequence_E0_E1_KA}
\bfE_0(\cala,\Phi)  \xrightarrow{\bfw} \bfE_1(\cala,\Phi)  \xrightarrow{\bfx} \bfK(\cala).
\end{equation}

>From the diagram~\eqref{ho-pullback_auxiliary} 
 we obtain a weak equivalence of spectra, natural in $(\cala,\Phi)$.
\[
\bfy \colon \bfE_1(A,\Phi) \xrightarrow{\simeq} \bfE_2(A,\Phi).
\]
Let
\[\bfz \colon \bfK(\Nil(\cala,\Phi)) \to \bfE_2(\cala,\Phi)
\]
be  the in $(\cala,\Phi)$ natural weak homotopy equivalence associated to the 
homotopy fiber sequence of Theorem~\ref{the:fiber_sequence}. 
Define $\bfE(\cala,\Phi)$ to be the homotopy pullback
\[
\xymatrix{\bfE(\cala,\Phi)  \ar[r]^-{\overline{\bfy \circ \bfw}} \ar[d]_{\overline{\bfz}}^{\simeq}
& 
 \bfK(\Nil(\cala,\Phi)) \ar[d]^{\bfz}_{\simeq}
\\
\bfE_0(\cala,\Phi)   \ar[r]_-{\bfy \circ \bfw}
& 
\bfE_2(\cala,\Phi)
}
\]
It follows from the sequence \eqref{eq:fibration_sequence_E0_E1_KA} that
\[
\bfE(\cala,\Phi) \xrightarrow{\overline{\bfy\circ \bfw}} \bfK(\Nil(\cala,\Phi))
\xrightarrow{\bfx\circ \bfy\inv\circ \bfz} \bfK(\cala)
\] 
is a fibration sequence of spectra.

Now the inclusion
$i\colon \cala \to \Nil(\cala,\Phi)$ sending $A$ to $(A,0)$ induces a map of spectra
\[
\bfK(i) \colon \bfK(\cala) \to \bfK(\Nil(\cala,\Phi)).
\]

\begin{lemma}
In the homotopy category, the following composite agrees with $-\bfK(\Phi\inv)$: 
\[
\bfK(\cala) \xrightarrow{\bfK(i)} \bfK(\Nil(\cala,\Phi))\bigr)
\xrightarrow{\bfz} \bfE_2(A,\Phi)
\xrightarrow{\bfy^{-1}} 
\bfE_1(A,\Phi)
\xrightarrow{\bfx} \bfK(\cala).
\]
\end{lemma}

\begin{proof}
  As in section~\ref{sec:Proof_of_Theorem_ref(the:homotopy_pull_back_of_calx)} (but
  interchanging the roles of $t$ and $t\inv$) we denote by
  $\Ch(\cala_\Phi[t\inv])^w\subset \Ch(\cala_\Phi[t\inv])$ the full Waldhausen subcategory
  of chain complexes which are contractible over $\cala_\Phi[t,t\inv]$, and by
  $\Ch(\calx)^w\subset \Ch(\calx)$ the full Waldhausen subcategory or complexes whose
  plus-part is contractible. Then the results of that section imply
  $\bfE_2(\cala,\Phi)\simeq \bfK(\Ch(\cala_\Phi[t\inv])^w)$; moreover
  $\bfE_1(\cala,\Phi)\simeq \bfK(\Ch(\calx^w))$ if we use the equivalence
  \[
  \bfK(l_1)\vee \bfK(l_0)\colon \bfK(\cala) \vee \bfK(\cala) \xrightarrow{\simeq}
  \bfK(\calx)
  \] 
   from Theorem~\ref{the:computing_K(calx)}.
 
  Under these identifications, the composite $\bfz\circ \bfK(i)$ corresponds to the map
  induced by the functor
  \[
   F_1\colon \cala\to\Ch(\cala_\Phi[t\inv])^w, \quad A\mapsto
  \cone(\Phi(A)\xrightarrow{t\inv} A).
  \] 
   But this is the image of the functor
  \[
   F_2\colon \cala\to\Ch(\calx)^w, \quad A\mapsto
  \cone\bigl(l_1(\Phi\inv(A))\xrightarrow{(t\inv,\id)} l_0(A)\bigr)
  \] 
  under the projection   $\calx\to \cala_\Phi[t\inv]$. Applying the Additivity theorem to the
  cylinder-cone-sequence shows that in $\bfK(\calx)$, we have
\[
\bfK(F_2)\simeq \bfK(l_0)- \bfK(l_1\circ \Phi\inv),
\]
which is the image under $\bfK(l_1)\vee \bfK(l_0)$ of the map 
\[
(-\bfK(\Phi\inv),\id)\colon \bfK(\cala)\to \bfK(\cala)\vee \bfK(\cala).
\]
Now project to the first variable.
\end{proof}

Hence the fibration sequence splits and we obtain a weak equivalence
\begin{eqnarray}
\bfK(\bfi) \vee \bfK(\overline{\bfy \circ \bfw}) \colon 
\bfK(\cala) \vee \bfE(\cala,\Phi)\xrightarrow{\simeq} \bfK(\Nil(\cala,\Phi)).
\label{K(cala)_vee_K(E)_simeq_K(Nil)}
\end{eqnarray}
As $\bfE_0(\cala,\Phi)$ is the homotopy fiber of 
$\bfK(i_+) \colon \bfK(\cala)  \to \bfK(\cala_{\Phi}[t])$ and
$\bfNK(\cala_{\Phi}[t])$ is the homotopy fiber of $\bfK(\bfev_0^+) \colon \bfK(\cala_{\Phi}[t]) \to \bfK(\cala)$
and $\bfK(\bfev_0^+) \circ \bfK(i_+)$ is the identity, we obtain a weak equivalence
of spectra, natural in $(\cala,\phi)$
\[
\bfu \colon \bfE_0 \to \Omega \bfNK(\cala_{\Phi}[t]).
\]
Thus we obtain a weak equivalence of spectra, natural in $(\cala,\Phi)$,
\begin{eqnarray}
\bfu \circ \overline{\bfz} \colon \bfE 
& \xrightarrow{\simeq} &
\Omega \bfNK(\cala_{\Phi}[t]).
\label{K(cala)_vee_E_simeq_K(cala)_vee_OmegaNK}
\end{eqnarray}
Now Theorem~\ref{the:BHS_decomposition_for_connective_K-theory}%
~\ref{the:BHS_decomposition_for_connective_K-theory:Nil} follows 
from~\eqref{K(cala)_vee_K(E)_simeq_K(Nil)} and~\eqref{K(cala)_vee_E_simeq_K(cala)_vee_OmegaNK}, 
provided that Theorem~\ref{the:fiber_sequence} holds.

 
 \section{On the Nil-category}
 \label{sec:On_the_Nil_category_(Version}

 Lemma~\ref{lem:caly_and_cala}  and Theorem~\ref{the:relative_fibration_sequences} 
 imply that there is homotopy fiber sequence
 \begin{equation}\label{eq:fundamental_fiber_sequence}
 \bfK(\Ch(\cala_\Phi[t^{-1}])^w) \to \bfK(\cala_\Phi[t^{-1}])\to \bfK(\cala_\Phi[t,t^{-1}])
 \end{equation}
 where
 $\Ch(\cala_\Phi[t^{-1}])^w)$ denotes the category of bounded chain complexes over
 $\cala_\Phi[t^{-1}]$ which are contractible as chain complexes over
 $\cala_\Phi[t,t^{-1}]$. The main goal
 of this section is to see that the first term of this sequence can be described in terms
 of the $K$-theory of the  twisted $\Nil$-category $\Nil(\cala,\Phi)$.

 If $(A,\phi)$ is an object of $\Nil(\cala,\Phi)$, then there is an associated chain
 complex
 \[
 \Phi(A) \xrightarrow{t^{-1} -i_- \phi} A
 \]
over $\cala_\Phi[t^{-1}]$,  concentrated in dimension $0$ and $1$. It is contractible over $\cala_\Phi[t,t^{-1}]$ since
 \[
  t^{-1} -i_- \phi = (1-\phi \cdot t)\circ t^{-1}
  \] 
  and both $1-\phi \cdot t$ and  $t^{-1} $ are invertible in
 $\cala_\Phi[t,t^{-1}]$:  inverses are $\sum_{i=0}^{n-1} (\phi \cdot t)^i$ and $t$
 respectively, where $n$ is such that $\phi^{(n)}=0$. This induces a  functor of Waldhausen categories
 \[
 \chi\colon \Nil(\cala,\Phi)\to \Ch(\cala_\Phi[t^{-1}])^w.
 \]
 
 The goal of this section is to prove the following result:
 
 \begin{theorem}\label{the:relative_term_is_Nil}
  Suppose that $\cala$ is idempotent complete. 
 Then the  induced map on connective $K$-theory
 \[
 \bfK(\chi)\colon \bfK(\Nil(\cala,\Phi))\to \bfK(\Ch(\cala_\Phi[t^{-1}])^w)
\] 
 is a homotopy equivalence.
 \end{theorem}
 
 \begin{proof}[Proof of Theorem~\ref{the:fiber_sequence} using Theorem~\ref{the:relative_term_is_Nil}]
 Immediate from the fiber sequence \eqref{eq:fundamental_fiber_sequence}.
 \end{proof}


 \subsection{The characteristic sequence}
 \label{subsec:The_characteristic_sequence}

 The first step in the proof of Theorem~\ref{the:relative_term_is_Nil} is to relate chain complexes over
$\cala_\Phi[t\inv]$ with chain complexes over $\cala$ equipped with an endomorphism. 
This relation is a consequence of the characteristic sequence, which we recall now.

 Recall that $\cala^\kappa$ is obtained from $\cala$ by adjoining countable direct sums.  
 We have defined induction and restriction functors 
 $i_- \colon \cala^\kappa \to \cala_{\Phi}[t^{-1}]^\kappa$ and $i^- \colon \cala_{\Phi}[t^{-1}]^\kappa \to \cala^{\kappa}$ 
 in Subsections~\ref{subsec:induction}
 and~\ref{subsec:restriction}.  Consider an object $A \in \cala_{\Phi}[t^{-1}]^\kappa$. Let 
 \[
 e \colon i_-i^-A \to A
 \]
 be the morphism in $\cala_{\Phi}[t^{-1}]^\kappa$ which is the adjoint of 
 $\id\colon i^-A \to i^-A$ under the adjunction of 
 Lemma~\ref{lem:Adjunction_between_induction_and_restriction}.
 We have the  morphism $\id_A \cdot t^{-1} \colon \Phi (A) \to  A$ 
 in $\cala_{\Phi}[t^{-1}]^\kappa$. Applying the composite $i_-i^-$ yields a morphism 
 $i_-i^- (\id_A \cdot t^{-1})  \colon i_-i^-\Phi(A) \to i_-i^-A$  in 
 $\cala_{\Phi}[t^{-1}]^\kappa$. We also have the morphism 
 $\id_{i_-i^-A} \cdot t^{-1} \colon i_-i^-\Phi(A) \to i_-i^-A$. 
 We will abbreviate $\id_{i_-i^-A} \cdot t^{-1}$ and $\id_A \cdot t^{-1}$ by $t^{-1}$.
 The difference of the two morphisms above yields the  morphism in $\cala_{\Phi}[t^{-1}]^\kappa$
 \[
 t^{-1} - i_-i^-t^{-1} :=  \id_{i_-i^-A} \cdot t^{-1} - i_-i^- (\id_A \cdot t^{-1}) \colon 
  i_-i^- \Phi(A) \to i_-i^-A.
 \]
 
 \begin{lemma}\label{lem:characteristic_sequence}
 Let $A$ be an object in $\cala_{\Phi}[t^{-1}]^\kappa$. Then the so called characteristic sequence 
 \[
 0 \to  i_-i^-\Phi(A) \xrightarrow{t^{-1} - i_-i^-t^{-1}}  i_-i^-A \xrightarrow{e}  A \to  0
 \]
 in $\cala_\Phi[t^{-1}]^\kappa$ is (split) exact and natural in $A$.
 \end{lemma}
 \begin{proof}
 We have $i_-i^-\Phi (A)  = \bigoplus_{i = -\infty} ^{-1} \Phi^{-i} (A)$ and 
 $i_-i^-A  = \bigoplus_{i = -\infty} ^{0} \Phi^{-i} A$. Under
 this identification the short sequence under consideration becomes the following sequence in
 $\cala_{\Phi}[t]^{\kappa}$:
 \begin{multline*}
 0 
 \to 
 \bigoplus_{i = -\infty}^{-1} \Phi^{-i} (A) 
 \xrightarrow{\begin{pmatrix}
 t^{-1} & 0 & 0 & 0 &\cdots
 \\ 
 -\id  & t^{-1} & 0 & 0 & \cdots
 \\
 0 & -\id & t^{-1} & 0 & \cdots 
 \\
 0 & 0 & -\id & t^{-1} & \cdots
 \\
 \vdots & \vdots& \vdots & \vdots & \ddots & 
 \end{pmatrix}
 } 
 \bigoplus_{i = -\infty}^{0} \Phi^{-i} (A) 
 \\
  \xrightarrow{\begin{pmatrix} \id&  t^{-1} &  t^{-2} &  \cdots
 \end{pmatrix}}
 A 
 \to 
 0.
\end{multline*}
If we view this as a $2$-dimensional chain complex, we obtain  an $\cala_{\Phi}[t]^{\kappa}$-chain contraction by
 \[ 
 \bigoplus_{i = -\infty}^{-1} \Phi^{-i} (A) 
 \xleftarrow{
 \begin{pmatrix}
 0  & -\id  & -t^{-1}  & -t^{-2}  & -t^{-3} & \cdots
 \\ 
 0 & 0 & -\id  & -t^{-1} &  -t^{-2} & \cdots
 \\
 0 & 0 & 0 & -\id & -t^{-1} & \cdots 
 \\
 0 & 0 & 0 & 0 & -\id & \cdots
 \\
 0 & 0 & 0 & 0 & 0  & \cdots
 \\
 \vdots & \vdots& \vdots & \vdots & \vdots & \ddots & 
 \end{pmatrix}
 } 
 \bigoplus_{i = - \infty}^{0} \Phi^{-i} (A) 
 \xleftarrow{\begin{pmatrix} \id \\ 0 \\ 0 \\ 0 \\ 0 \\ \vdots
 \end{pmatrix}}
 A
 \]
 \end{proof}

 Given an $\cala_\Phi[t^{-1}]^\kappa$-chain complex $C$, we obtain from
 Lemma~\ref{lem:characteristic_sequence} that $C$ can be resolved by an in $C$ natural
 short  exact sequence
 \begin{eqnarray}
  &
 0 \to  i_-i^-\Phi (C) \xrightarrow{t^{-1} - i_-i^-t^{-1}}  i_-i^- C \xrightarrow{e}  C \to 0
 &
 \label{characteristic_resolution}
 \end{eqnarray}
 of chain complexes in the image of the co-unit $i_- i^-$ of the adjunction.
 
 \begin{lemma}\label{lem:characteristic_sequence_over_cala}
   Consider a morphism $\phi \colon \Phi(A) \to A$ in $\cala^{\kappa}$. Then we obtain a
   (split) exact and in $\Phi$-natural exact sequence of
   $\cala^{\kappa}$-modules 
\begin{eqnarray*}
  &
 0 \to  i^-i_-\Phi (A) \xrightarrow{i^-(t^{-1} - i_-\phi)}  i^-i_-A \xrightarrow{e'}  A \to 0.
 &
 \end{eqnarray*}
\end{lemma}
\begin{proof}
The proof is analogous to the one of Lemma~\ref{lem:characteristic_sequence}, the role of
$t^{-1}$ is now played by $\phi$. Namely, we have $i^-i_- \Phi (A)  = \bigoplus_{i = -\infty} ^{-1} \Phi^{-i} (A)$ and 
 $i^-i_-A  = \bigoplus_{i = -\infty} ^{0} \Phi^{-i} (A)$. Under
 this identification the short sequence under consideration becomes the following sequence in
 $\cala^{\kappa}$:
 \begin{multline*}
 0 
 \to 
 \bigoplus_{i = -\infty}^{-1} \Phi^{-i} (A) 
 \xrightarrow{\begin{pmatrix}
  \phi & 0 & 0 & 0 &\cdots
 \\ 
 -\id  & \Phi(\phi) & 0 & 0 & \cdots
 \\
 0 & -\id & \Phi^{2}(\phi) & 0 & \cdots 
 \\
 0 & 0 & -\id & \Phi^{3}(\phi)  & \cdots
 \\
 \vdots & \vdots& \vdots & \vdots & \ddots & 
 \end{pmatrix}
 } 
 \bigoplus_{i = -\infty}^{0} \Phi^{-i} (A) 
 \\
  \xrightarrow{\begin{pmatrix} \id&  \phi &  \phi^{(2)}  &  \cdots
 \end{pmatrix}}
 A 
 \to 
 0,
\end{multline*}
where $\phi^{(k)} := \phi \circ \Phi^{1}(\phi) \circ \cdots \circ \Phi^{(k-1)}(\phi) \colon \Phi^{k} (A) \to A$. 
 If we view this as a $2$-dimensional chain complex, we obtain  a $(\cala^{\kappa})^{\kappa}_{\Phi}$-chain contraction by
 \[ 
 \bigoplus_{i = -\infty}^{-1} \Phi^{-i} (A) 
 \xleftarrow{
 \begin{pmatrix}
 0  & -\id  & -\Phi(\phi)  & -\Phi(\phi^{(2)})  & -\Phi(\phi^{(3)}) & \cdots
 \\ 
 0 & 0 & -\id  & -\Phi^2 (\phi) &  -\Phi^2(\phi^{(2)}) & \cdots
 \\
 0 & 0 & 0 & -\id & -\Phi^3(\phi) & \cdots 
 \\
 0 & 0 & 0 & 0 & -\id & \cdots
 \\
 0 & 0 & 0 & 0 & 0  & \cdots
 \\
 \vdots & \vdots& \vdots & \vdots & \vdots & \ddots & 
 \end{pmatrix}
 } 
 \bigoplus_{i = - \infty}^{0} \Phi^{-i} (A) 
 \xleftarrow{\begin{pmatrix} \id \\ 0 \\ 0 \\ 0 \\ 0 \\ \vdots
 \end{pmatrix}}
 A 
\qedhere 
\]
\end{proof}

 \begin{notation}[$\End(\Ch(\cala),\Phi)$]
 \label{not:End(Ch(calakappa),Phi(-1))}\ \\
 Denote by $\End(\Ch(\cala),\Phi)$ the Waldhausen category of
 $\Phi$-twisted endomorphisms of $\cala$-chain complexes. An object is a
 pair $(C, \phi)$, where $C$ is a chain complex in $\cala$, and 
 $\phi \colon \Phi(C)\to C$ is a chain map. A morphism $u \colon (C, \phi)\to (D, \psi)$ is 
 an $\cala$-chain map  $u \colon C\to D$ such that $u \circ \phi = \psi \circ \Phi(u)$. 
 It is a cofibration or weak equivalence respectively if the underlying 
 $\cala$-chain map $u$ has this property.
 \end{notation}
 
 Define functors  of Waldhausen categories 
\begin{multline}
 \chi_{\cala^\kappa} \colon \End(\Ch(\cala^\kappa), \Phi)   \to   \Ch(\cala_\Phi[t^{-1}]^\kappa),   
  \\
 (C,\phi)   \mapsto \cone \bigl(i_-\Phi(C)\xrightarrow{t^{-1} - i_-\phi} i_-C\bigr); 
 \label{chi_on_End}
\end{multline}
\begin{eqnarray}
 N_{\cala^\kappa} \colon \Ch(\cala_\Phi[t^{-1}]^\kappa)   \to \End(\Ch(\cala^\kappa, \Phi)),  & & 
 D \mapsto (i^-D,i^-t^{-1}).
 \label{N_on_End}
\end{eqnarray}
 
 \begin{lemma}\label{lem:chi_and_N}
The functors $\chi_{\cala^\kappa}$ and  $N_{\cala^\kappa}$ are inverse equivalences of Waldhausen categories. 
\end{lemma}
 
 \begin{proof}
   We obtain from the sequence~\eqref{characteristic_resolution} 
 using Subsection~\ref{subsec:Homotopy_fiber_sequences}
 for any object $D$ in $\Ch(\cala_\Phi[t^{-1}]^\kappa)$ a weak equivalence in 
 $\Ch(\cala_\Phi[t^{-1}]^\kappa)$ 
 \[
 T(D) \colon \chi_{\cala^{\kappa}} \circ N_{\cala^{\kappa}}(D) = \cone\bigl(i_-i^-\Phi (D) 
 \xrightarrow{t^{-1} - i_-i^-t^{-1}}  i_-i^-D  \bigr) \xrightarrow{\simeq} D,
 \]
 and thus a natural weak equivalence 
 $T \colon \chi_{\cala^{\kappa}} \circ N_{\cala^{\kappa}} \xrightarrow{\simeq} \id$.
 
 Conversely, for an object
 $(C,\phi)$ in $\End(\Ch(\cala^\kappa))$, we obtain from Lemma~\ref{lem:characteristic_sequence_over_cala}
 the short exact sequence in $\Ch(\cala^{\kappa})$
 \begin{eqnarray*}
  &
 0 \to  i^-i_-\Phi(C) \xrightarrow{i^-(t^{-1} - i_-\phi)}  i^-i_-C \xrightarrow{e'}  C \to 0
 &
 \end{eqnarray*}
  such that the following diagram commutes:
 \[
 \xymatrix@!C=9em{
 i^-\Phi (i_-\Phi (C)) \ar[r]^--{i^-(t^{-1} - i_-\phi)} \ar[d]^{i^-t^{-1}}
 & 
 i^-\Phi(i_-C) \ar[r]^{e'} \ar[d]^{i^-t^{-1}}
 & \Phi(C) \ar[d]^{\phi}
 \\
 i^-i_-\Phi (C) \ar[r]^-{i^-(t^{-1} - i_-\phi)} 
 & 
 i^-i_-C \ar[r]^{e'}
 & C
 }
 \]
 Thus we obtain using Subsection~\ref{subsec:Homotopy_fiber_sequences} 
  a natural weak equivalence in $\End(\Ch(\cala^\kappa)$
 \[
 S(C,\phi)  \colon N_{\cala^{\kappa}} \circ \chi_{\cala^{\kappa}}(C,\phi)   \xrightarrow{\simeq} (C,\phi).
 \]
 This finishes the proof of Lemma~\ref{lem:chi_and_N}.
 \end{proof}


\subsection{Homotopy nilpotent endomorphisms}

The goal of this subsection is to restrict the equivalences from Lemma~\ref{lem:chi_and_N} to equivalences on suitable subcategories.
 
Recall that a $\cala^\kappa$-chain complex $C$ is called homotopy finite if it is chain
equivalent to a chain complex in $\cala$. Recall as well that a morphism $f\colon
\Phi(A)\to A$ of $\cala$ is called \emph{$\Phi$-nilpotent} if for some $n \ge 1 $ the
$n$-fold composite
   \[
   f^{(n)}:=f \circ \Phi(f) \circ \cdots \circ \Phi^{n-1}(f) \colon \Phi^n(A)\to A
   \]
   is trivial.

\begin{definition}[Homotopy $\Phi$-nilpotent]\label{def:nilpotent}
   A chain map  $f\colon \Phi(C)\to C$ of chain complexes in $\cala$ is called \emph{$\Phi$-homotopy nilpotent} if for
   some $n \ge 1 $ the $n$-fold composite $f^{(n)}$ is $\cala$-chain homotopic to the trivial chain map.
 \end{definition}

 \begin{lemma}\label{lem:endos_and_finiteness}
     Let $C$ be bounded $\cala$-chain
     complex. Let $D$ be an $\cala^{\kappa}$-chain complex which is homotopy equivalent to a
     bounded $\cala$-chain complex. Let $u\colon (C,\phi)\to (D, \psi)$ be a morphism in
     $\End(\Ch(\cala^\kappa), \Phi)$.
 
     Then there exists a commutative diagram in
     $\End(\Ch(\cala^\kappa), \Phi)$
     \[
     \xymatrix{(C,\phi) \ar[r]^u \ar@{>->}[d] & (D, \psi) \ar@{>->}[d]^\simeq\\
       (E, \mu) \ar[r]^\simeq & (F,\sigma) }
     \]
     where the arrows labelled by $\simeq$ are weak equivalences, 
    the vertical arrows are cofibrations, and $E$ is a bounded
     $\cala$-chain complex as well.
 \end{lemma}

 \begin{proof}
We begin with two reductions.

\emph{Reduction 1:} It is enough to consider the special case where $u\colon C\to D$ is a cofibration. 

In fact, put $D' = \cyl(u)$.
 Let $u' \colon C \to D'$ be the canonical inclusion and $p \colon D' \to D$ be the canonical projection. 
 Since $\psi \circ u = u \circ \phi$ holds, we conclude from 
 naturality of the mapping cylinder construction
 that there exists a chain map $\psi' \colon \Phi(D') \to D'$ such that the following 
is a sequence in $\End(\cala^\kappa,\Phi)$
 \[
 (C,\phi) \xrightarrow{u'} (D',\psi') \xrightarrow{p} (D,\psi).
\]

 Applying now the special case ($u'$ is a cofibration), we obtain the left square in the following diagram, 
where $E$ is a bounded $\cala$-chain complex:
  \[
 \xymatrix{(C,\phi) \ar[r]^{u'} \ar@{>->}[d]^{x}
 &
 (D',\psi') \ar[r]^{p}_{\simeq}  \ar@{>->}[d]^{y}_{\simeq} 
 & (D,\psi) \ar@{>->}[d]^{\overline{y}}_{\simeq} 
 \\
 (E,\mu) \ar[r]^{z}_{\simeq} 
 &
 (F',\sigma'') \ar[r]^{\overline{p}}_{\simeq} 
 & (G,\tau)
 }
 \]
  The right square is obtained by applying the pushout construction to the pair $(p,y)$.
 All vertical maps are cofibrations and all maps marked with $\simeq$ are chain homotopy equivalences.
 Now the outer square is the desired diagram in $\End(\Ch(\cala^{\kappa}))$.

\emph{Reduction 2:} It is enough to construct 
\begin{enumerate} 
\item  a zig-zag in $\End(\cala^\kappa,\Phi)$
\[
\xymatrix{(E,\mu) & (C,\phi) \ar@{>->}[l]_j \ar@{>->}^u[r] & (D,\psi), }
\]
where $E$ is a bounded $\cala$-chain complex and $j$ is a cofibration;
\item a chain homotopy equivalence $v\colon E\to D$ such that $v\circ j=u$;

\item  a chain homotopy
\[
H\colon \psi\circ v \simeq v\circ \Phi(\mu)\colon \Phi(E)\to D
\]
which is stationary over $\Phi(C)$ (i.e., $H\circ \Phi(j)=0$).
\end{enumerate}

In fact, suppose we are given this data. By 
Lemma~\ref{lem:elementary_facts}~\ref{lem:elementary_facts:maps_between_mapping_cylinders},
 there exists a chain map 
\[
\rho=F_{(\mu,\psi, H)}\colon \Phi(\cyl(v))\to \cyl(v)
\]
such that the inclusions of the top and bottom end into the mapping cylinder give rise to a zig-zag
\[
\xymatrix{(E,\mu) \ar@{>->}[r]^{i(E)} &(\cyl(v),\rho) & (D,\psi) \ar@{>->}[l]_{i(D)}.
}
\]
Explicitly, 
\[
\rho_n  \colon
  \Phi(E)_{n-1} \oplus \Phi(E)_n \oplus \Phi(D)_n  \xrightarrow{{\footnotesize\begin{pmatrix} \mu_{n-1} & 0 & 0 
 \\ 
 0 & \mu_n & 0 \\ H_n & 0 & \psi_n\end{pmatrix}}} E_{n-1} \oplus E_n \oplus D_n.
 \]

Thus we get a (non-commutative)
diagram in $\End(\Ch(\cala^\kappa))$:
\begin{eqnarray}
 &
 \xymatrix{(C,\phi) \ar[r]^-{u} \ar@{>->}[d]_{j} 
 & 
 (D,\psi)  \ar@{>->}[d]^{i(D)}_{\simeq}
 \\
 (E,\mu) \ar[r]_-{i(E)}^-{\simeq}
 & (\cyl(v),\rho)
 }
 &
 \label{non-comm_diagram_in_End}
 \end{eqnarray}

 In order to obtain a commutative diagram, we have to pass to a quotient of $\cyl(v)$ as
 follows.  Let $k \colon \cyl(u) \to \cyl(v)$ be the obvious cofibration induced
 by $j(C) \colon C \to E$.  Let $i(C)\colon C\to \cyl(u)$ and $\pr \colon \cyl(u) \to D$ be the canonical
 inclusion and projection, respectively.
 Define an $\cala^{\kappa}$-chain complex $F$ by the pushout of
 $\cala^{\kappa}$-chain complexes
 \[
 \xymatrix{\cyl(u) \ar[r]^-{\pr}_-{\simeq} \ar@{>->}[d]_{k}^\simeq 
 & D \ar@{>->}[d]_{\overline{k}}^\simeq
 \\
 \cyl(v) \ar[r]_-{\overline{\pr}}^{\simeq} 
 & 
 F
 }
 \]
 All arrows are chain homotopy equivalences: This is always true for the projection 
 $\pr$ in the mapping cylinder and follows for $\overline{\pr}$ from
 Subsection~\ref{subsec:Homotopy_fiber_sequences} since $k$ is a cofibration. The morphism
$k$ is a chain homotopy equivalence, as both domain and target are canonically homotopy equivalent
to $D$.

As $u\colon (C,\phi)\to (D,\psi)$ is a morphism in $\End(\Ch(\cala^\kappa),\Phi)$,
we obtain by naturality an induced map $\rho'\colon \Phi(\cyl(u))\to \cyl(u)$ for
which the projection $\pr$ becomes a morphism in the endomorphism category. Moreover
$\rho$ restricts to $\rho'$ under $k$. This follows from the explicit form of $\rho$
as displayed above, together with the fact that $\mu$ restricts to $\phi$ and $H$
restricts to 0 by assumption.

Hence we may define $(F,\sigma)$ to be the pushout in $\End(\Ch(\cala^\kappa),\Phi)$
in the right square of the following diagram:
 \[
 \xymatrix{(C,\phi) \ar@{>->}[d]^j \ar[r]^{i(C)} 
  & (\cyl(u),\rho')\ar[r]^-{\pr}_-{\simeq} \ar@{>->}[d]_{k}^\simeq 
 & (D,\psi) \ar@{>->}[d]_{\overline{k}}^\simeq
 \\
 (E,\mu) \ar[r]^\simeq_{i(E)}
 & (\cyl(v),\rho) \ar[r]_-{\overline{\pr}}^-{\simeq}
 &  (F,\sigma)
 }
 \]
 The outer square provides then the conclusion of the Lemma. 

This finishes the proof of Reduction 2. We are left to show that the hypotheses of Reduction 2 are satisfied.

 Choose a bounded $\cala$-chain
 complex $D'$ together with an $\cala^{\kappa}$-chain homotopy equivalence $f \colon D' \to D$.  
 Choose a homotopy inverse $f^{-1} \colon D \to D'$. Consider $f^{-1} \circ u \colon C \to D'$. 
 Then we can choose a homotopy $f \circ (f^{-1} \circ u) \simeq u$. 
 Let $E = \cyl(f^{-1} \circ u)$ and write $e$ for its
 differential. Notice that $E$ is a bounded $\cala$-chain complex. We obtain from
 Lemma~\ref{lem:elementary_facts}~\ref{lem:elementary_facts:maps_between_mapping_cylinders}
 an $\cala^{\kappa}$-chain map $v \colon E \to D$ such that 
 $v \circ j(C) = u$ and $v \circ j(D') = f$, where $j:=j(C) \colon C \to E$ and 
 $j(D') \colon D' \to \cyl(f^{-1} \circ u)$ denote the canonical inclusions.

 Since $j(D')$ and $f$ are chain homotopy equivalences, the same is true for $v$.
 From Lemma~\ref{lem:Chain_homotopy_equivalences_and_cofibrations} we obtain a chain map
 $w \colon D \to E$ with $w \circ u = j(C)$ and a chain homotopy $h \colon v \circ w \simeq \id_D$
 satisfying 
 \begin{eqnarray}
 h \circ u & = & 0.
 \label{h_circ_u_is_0}
 \end{eqnarray}
 Define $\mu \colon \Phi(E) \to E$ to be $w \circ \psi \circ\Phi(v)$. 
 Since 
 \begin{eqnarray*}
 \mu \circ \Phi(j(C)) 
 & = & 
 w \circ \psi \circ \Phi(v) \circ \Phi(j(C)) 
 \\
 & = & 
 w \circ \psi \circ \Phi(u)
 \\
 & = & 
 w \circ u  \circ \phi 
 \\
 & = & 
 j(C)  \circ \phi,
 \end{eqnarray*}
 we obtain a morphism $j(C) \colon (C,\phi) \to (E,\mu)$ in $\End(\Ch(\cala),\Phi)$.
 
 Consider the (not necessarily commutative) diagram of $\cala^{\kappa}$-chain complexes
 \[
 \xymatrix{\Phi(E) \ar[r]^-{\mu} \ar[d]_{\Phi(v)}^{\simeq} &  E\ar[d]^{v}_{\simeq}
       \\
       \Phi(D) \ar[r]_-{\psi} & D}
 \]
 It commutes  up to the chain homotopy
 \[
 H := h \circ \psi \circ \Phi(v).
 \]
 Then $H$ is stationary over $\Phi(C)$: In fact, we compute
 \begin{eqnarray}
 \label{hprime_circ_j(C)_is_0}
 H \circ \Phi(j(C))
 & = &
 h \circ \psi \circ \Phi(v) \circ \Phi(j(C))
 \\
 & = & 
 h \circ \psi \circ \Phi(u)
 \nonumber
 \\
 & = & 
 h \circ u \circ \phi
 \nonumber
 \\
 & \stackrel{\eqref{h_circ_u_is_0}}{=} &
 0 
 \nonumber
 \end{eqnarray}

This concludes the proof that the hypotheses of Reduction 2 are satisfied, and thus the proof of the Lemma.
 \end{proof}

 \begin{lemma}\label{lem:comparing_finiteness_conditions} 
 For $(C, \phi)\in \End(\Ch(\cala^{\kappa}, \Phi))$, the following are equivalent:
 
 \begin{enumerate}
 
 \item \label{lem:comparing_finiteness_conditions:(C,phi)}
   $C$ is homotopy finite  and $\phi$ is homotopy nilpotent;
 
 \item \label{lem:comparing_finiteness_conditions:chi(C,phi)}
  $\chi_{\cala^{\kappa}}(C, \phi)$ is homotopy finite and contractible in
  $\cala_{\Phi}[t,t^{-1}]^\kappa$.

 \end{enumerate}
 \end{lemma}
 \begin{proof}~\ref{lem:comparing_finiteness_conditions:(C,phi)}~%
 $\implies$~\ref{lem:comparing_finiteness_conditions:chi(C,phi)}
    If we apply Lemma~\ref{lem:endos_and_finiteness} to the morphism $(0,0) \to (C, \phi)$, 
  we obtain  a zigzag of weak equivalences 
   $(E,\mu)\xrightarrow{\simeq} (F,\sigma) \xleftarrow{\simeq} (C,\phi)$
   in  $\End(\Ch(\cala^\kappa), \Phi)$
   such that $E$ is a $\cala$-chain complex. This implies by Subsection~\ref{subsec:Homotopy_fiber_sequences}
   that   $\chi(C,\phi)$ is $\cala_{\Phi}[t^{-1}]^\kappa$-chain homotopy equivalent to
   $\chi(E,\sigma)$.  Since $E$ is a $\cala$-chain complex, $\chi_{\cala}(E,\sigma)$ 
   is a $\cala_{\Phi}[t^{-1}]$-chain complex.

   Over $\cala_\Phi[t,t^{-1}]^\kappa$ we may split $i_-\phi-t^{-1}=(1-\phi \cdot t)\circ
   t^{-1}$ where $t^{-1}$ is an isomorphism and $1-\phi \cdot t$ is a homotopy equivalence
   if $\phi $ is homotopy nilpotent (with homotopy inverse $\sum_{i\ge 0} (\phi \cdot
   t)^i$). In this case $i_-\phi-t^{-1}$ is a chain homotopy equivalence and hence its
   cone $\chi_{\cala^{\kappa}}(C, \phi)$ is $\cala_\Phi[t,t^{-1}]^\kappa$-contractible by
   Lemma~\ref{lem:elementary_facts}~\ref{lem:elementary_facts:chain_homotopy_equivalence_cone}.
   \\[2mm]~\ref{lem:comparing_finiteness_conditions:chi(C,phi)}~$\implies$~%
\ref{lem:comparing_finiteness_conditions:(C,phi)} Let $D=\chi(C,
   \phi)$. Lemma~\ref{lem:chi_and_N} implies that $N_{\cala^{\kappa}}(D) \simeq (C, \phi)$
   holds in $\End(\Ch(\cala^\kappa), \Phi)$.  So it suffices to show that the
   $\cala^{\kappa}$-chain complex $i^-D$ is homotopy finite and that $i^-t^{-1} \colon i^-
   \Phi(D) \to i^-D$ is homotopy nilpotent.  By assumption $D$ is homotopy finite; by
   homotopy invariance we may assume that $D$ is actually finite, i.e., in
   $\Ch(\cala_\Phi[t\inv])$.
   
   Let $H$ be a null-homotopy for the $\cala_{\Phi}[t,t^{-1}]$-chain complex $j_-D$.  Let
   $M$ be a large enough natural number so that all coefficients in $H$ for $t^k$ are
   zero, provided $\lvert k\rvert \geq M$.  Then the collection of maps
\[
H_n\trun\colon D_n[-\infty, -M]\to D_{n+1}[-\infty,0]
\] 
(introduced in Notation~\ref{not:truncation_for_objects})  
   is a null-homotopy for   the inclusion $D[-\infty,-M] \to D[-\infty,0] $. 
   This inclusion agrees with the $\cala^{\kappa}$-chain map 
   $(i^-t^{-1})^{(M)} \colon  i^- \Phi^M(D) = (\Phi^M(D))[-\infty,0] \to i^- D = D[-\infty,0]$. 
   We conclude that $i^-t\inv$ is $\Phi$-homotopy nilpotent. Moreover
   the exact sequence
   \[
   0\to D[-\infty,-M]  \to D[-\infty,0] \to D[-\infty,0]/D[-\infty,-M]\to 0
   \] 
   shows using Subsection~\ref{subsec:Homotopy_fiber_sequences} 
   that we obtain $\cala^{\kappa}$-chain homotopy equivalences
   \begin{multline*}
   D[-\infty,0]/D[-\infty,-M] \simeq \cone\bigl(D[-\infty,-M]  \to D[-\infty,0]\bigr) 
   \\
   \simeq D[-\infty,0] \oplus \Sigma D[-\infty,-M]. 
   \end{multline*}
   We conclude  that $D[-\infty,0]$ is a homotopy retract of the (finite) $\cala$-chain complex $D[-\infty,0]/D[-\infty,-M]$.
   As $\cala$ was assumed to be idempotent complete, we conclude from
   Lemma~\ref{lem:criterion_for_finitely_dominated} that $i^-D = D[-\infty,0]$ is $\cala$-homotopy finite.
 \end{proof}

\begin{notation}[$\HNil(\Chhf(\cala), \Phi)$ and ${\Chhf(\cala_{\Phi}[t^{-1}])^w}$]
 \label{not:HNil(Chhf(cala))_and_Chhf(cala_Phi[t])w)}\ \\
 Let $\HNil(\Chhf(\cala), \Phi)$ be the full subcategory of $\End(\Ch(\cala^{\kappa},\Phi))$
 consisting of objects $(C,\phi)$ for which the $\cala^{\kappa}$-chain complex $C$ 
 is homotopy finite and $\phi$ is homotopy $\Phi$-nilpotent. We let a cofibration (weak equivalence) 
 in $\HNil(\Chhf(\cala),\Phi)$ be a morphism $f\colon (C,\phi)\to(D,\psi)$ 
  whose underlying map is a cofibration (or weak equivalence, respectively). 
 
 Let $\Chhf(\cala_{\Phi}[t^{-1}])^w$ be the full subcategory of $\Ch(\cala_{\Phi}[t^{-1}]^\kappa)$
 consisting of those objects $C$ which are homotopy finite
 and become contractible in $\cala_{\Phi}[t,t^{-1}]^\kappa$. A morphism in this category is a 
 cofibration or weak equivalence if it is such in $\Ch(\cala_\Phi[t\inv]^\kappa)$.
\end{notation}

\begin{lemma} \label{lem:waldhausen_structure_on_HNil_and_End} This notion of cofibration
  and weak equivalence defines a structure of Waldhausen category on both
  $\HNil(\Chhf(\cala),\Phi)$ and $\Chhf(\cala_\Phi[t,t\inv])^w$.
\end{lemma}

\begin{proof}
  As both categories contain the zero object, we are left to check that they are closed
  under pushouts along a cofibration. For $\Chhf(\cala_\Phi[t,t\inv])^w$ this follows (as
  usual) from Lemma~\ref{lem:homotopy_finite_two_of_three} and the fact that the class of
  weak equivalences in $\cala_\Phi[t,t\inv]^\kappa$ satisfies the glueing lemma.

  Given a pushout diagram $(C_1,\phi_1)\rightarrow (C_0,\phi_0) \leftarrow (C_2,\phi_2)$
  in $\HNil(\Chhf(\cala),\Phi)$, denote by $(C,\phi)$ its pushout in
  $\End(\Ch(\cala^\kappa),\Phi)$. By Lemma~\ref{lem:comparing_finiteness_conditions} each
  $\chi_{\cala^\kappa}(C_i,\phi_i)$ (where $i=0,1,2$) is homotopy finite and contractible
  in $\cala_\Phi[t,t\inv]^w$. The functor $\chi_{\cala^\kappa}$ commutes with pushouts, so
  the first part of the proof implies that $\chi_{\cala^\kappa}(C,\phi)$ is also homotopy
  finite and contractible in $\cala_\Phi[t,t\inv]^w$. Hence, again by
  Lemma~\ref{lem:comparing_finiteness_conditions}, $(C,\phi)$ is an object of
  $\HNil(\Chhf(\cala),\Phi)$.
\end{proof}

\begin{proposition}\label{pro:reexpressing_nil}
  The functor $\chi_{\cala^{\kappa}}$ appearing in Lemma~\ref{lem:chi_and_N} induces an
  equivalence of Waldhausen categories
  \[
  \chi(\cala) \colon \HNil(\Chhf(\cala), \Phi) \xrightarrow{\simeq}
  \Chhf(\cala_\Phi[t^{-1}])^w,
  \]
  In particular we obtain a homotopy equivalence
  \[
  \bfK\bigl(\chi(\cala)\bigr) \colon \bfK\bigl(\HNil(\Chhf(\cala), \Phi)\bigr) \xrightarrow{\simeq}
  \bfK\bigl(\Chhf(\cala_\Phi[t^{-1}])^w\bigr).
  \]
\end{proposition}
\begin{proof}
  We conclude from Lemma~\ref{lem:comparing_finiteness_conditions} that the functor $\chi_{\cala}$
  defined in~\eqref{chi_on_End} restricts to a functor of Waldhausen categories 
  \[
   \chi(\cala) \colon \HNil(\Chhf(\cala), \Phi) \xrightarrow{\simeq} \Chhf(\cala_\Phi[t^{-1}])^w.
  \]
  Because of Lemma~\ref{lem:chi_and_N} and again by Lemma~\ref{lem:comparing_finiteness_conditions}, 
an inverse up to natural equivalence of Waldhausen
  categories is given by the restriction of the functor $N_{\cala^\kappa}$.
\end{proof}

 \subsection{Homotopy nilpotence vs.~strict nilpotence}

 \begin{lemma}\label{lem:endos_and_strictly_nilpotent}
   Let $C$ and $D$ be a bounded $\cala$-chain complexes. Consider a morphism 
   $u\colon   (C,\phi)\to (D, \psi)$ in $\End(\Ch(\cala), \Phi)$.  Suppose that $\phi$ is
   $\Phi$-nilpotent and $\psi$ is homotopy $\Phi$-nilpotent.
 
   Then there exists a commutative diagram in $\End(\Ch(\cala), \Phi)$
   \[
    \xymatrix{(C,\phi) \ar[r]^u \ar@{>->}[d] &  (D, \psi) \ar@{>->}[d]^\simeq\\
     (E, \mu) \ar[r]^\simeq & (F,\sigma) 
   }
   \]
   where the arrows labelled by $\simeq$ are weak equivalences, the vertical arrows are
   cofibrations, and $\mu$ is $\Phi$-nilpotent.
 \end{lemma}

 \begin{proof}
The same argument as in the proof of Lemma~\ref{lem:endos_and_finiteness} shows that we can make the following two reductions.

\emph{Reduction 1:} It is enough to consider the special case where $u\colon\Phi(D)\to D$ is a cofibration.

\emph{Reduction 2:} It is enough to construct (i) a zig-zag in $\End(\cala,\Phi)$
\[
\xymatrix{(E,\mu) & (C,\phi) \ar@{>->}[l]_j \ar@{>->}^u[r] & (D,\psi), }
\]
where $\mu$ is $\Phi$-nilpotent and $j$ is a cofibration, (ii) a chain homotopy equivalence $v'\colon D\to E$ 
satisfying $v'\circ u=j$, and (iii) a chain homotopy
\[
H'\colon \mu\circ \Phi(v') \simeq v'\circ \psi\colon \Phi(E)\to D
\]
which is stationary over $\Phi(C)$ (i.e., $H'\circ \Phi(u)=0$).

We now proceed to show that the assumptions of Reduction 2 are fulfilled. As a first step we prove that we can choose  an integer $n$ satisfying
\begin{eqnarray}
\phi^{(n)} & = & 0,
\label{phi(n)_is_zero}
\end{eqnarray}
and a chain homotopy
 \[
h(\psi) \colon \psi^{(n)} \simeq 0
 \]
 satisfying 
 \begin{eqnarray}
 h(\psi) \circ \Phi^n(u) & = & 0.
 \label{h(psi)_circ_Phi(u)_is_0}
 \end{eqnarray}

Let $E$ be the cokernel of the cofibration $u \colon C \to D$.
Let $\overline{\psi} \colon E \to E$ be the $\cala$-chain map induced by
$\psi$. Because of Lemma~\ref{lem:waldhausen_structure_on_HNil_and_End}
we can choose an integer $m$ such that $\phi^{(m)} = 0$ and there exists a nullhomotopy 
$\overline{H} \colon \overline{\psi}^{(m)} \simeq 0$. Since $u \colon C \to D$ is a cofibration, 
we can assume $D_k = C_k \oplus E_k$ and that the differential of $D$ looks like
\[
d_k = \begin{pmatrix} c_k & x_k \\ 0 & e_k \end{pmatrix} \colon 
D_k = C_k \oplus E_k \to D_{k-1} = C_{k-1} \oplus E_{k-1},
\]
if $c$ and $e$ denote the differentials of $C$ and $E$, and $\psi^{(m)}$ looks like
\[
\psi^{(m)}_k = \begin{pmatrix} 0 & y_k \\ 0 & \overline{\psi}_k^m \end{pmatrix} \colon 
\Phi^m(D_k) = \Phi^m(C_k) \oplus \Phi^m(E_k) \to D_k = C_k \oplus E_k.
\]
Define a homotopy 
\[
H_k = \begin{pmatrix} 0 &  0 \\ 0 & \overline{H}_k\end{pmatrix} \colon 
\Phi^m(D_k) = \Phi^m(C_k) \oplus \Phi^m(E_k) \to D_{k+1} = C_{k+1} \oplus E_{k+1}.
\]
We have
\begin{eqnarray*}
\lefteqn{d_{k+1} \circ H_k + H_{k-1} \circ \Phi^m(d_k)}
\\
& = & 
\begin{pmatrix} c_{k+1} & x_{k+1} \\ 0 & e_{k+1} \end{pmatrix}  
\circ \begin{pmatrix} 0 &  0 \\ 0 & \overline{H}_k \end{pmatrix} +
\begin{pmatrix} 0 &  0 \\ 0 & \overline{H}_{k-1} \end{pmatrix} 
\circ \begin{pmatrix} \Phi^m(c_k) & \Phi^m(x_k) \\ 0 & \Phi^m(e_k) \end{pmatrix} 
\\
& = & 
\begin{pmatrix} 0  & x_{k+1} \circ \overline{H}_k 
\\ 
0 & e_{k+1} \circ \overline{H}_k + \overline{H}_{k-1} \circ \Phi^m(e_k)\end{pmatrix} 
\\
& = & 
\begin{pmatrix} 0  & x_{k+1} \circ \overline{H}_k \\ 0 & \overline{\psi}^{(m)}_k\end{pmatrix} 
\end{eqnarray*}
Hence, if we put $z_k = y_k - x_{k+1} \circ \overline{H}_k$ and define $\omega \colon \Phi^m(D) \to D$ by
\[
\omega_k  = \begin{pmatrix} 0 & z_k \\ 0 & 0 \end{pmatrix} \colon 
\Phi^m(D_k) = \Phi^m(C_k) \oplus \Phi^m(E_k) \to D_k = C_k \oplus E_k,
\]
then $\omega$ is a chain map  and $H$ is a chain homotopy $\psi^{(m)} \simeq \omega$. 

It is easy to verify that if $f\simeq f'\colon C\to D$ are two chain maps which are
homotopic via a chain homotopy $H$, and $g\simeq g'\colon D\to E$ are homotopic via $K$,
then $g\circ f\simeq g'\circ f'$ via the chain homotopy $g\circ H + K\circ f'$.

In our situation 
\[
\psi^{(2m)}=\psi^{(m)}\circ \Phi^m\psi^{(m)}\simeq \omega \circ \Phi^m(\omega) = 0
\]
via the homotopy $h(\psi):=\psi^{(m)} \circ \Phi^m(H) + H \circ \Phi^m(\omega)$. Since both 
$\Phi^m(H)$ and $\Phi^m(\omega)$ are zero when restricted along $u\colon C\to D$, the same 
is true for $h(\psi)$. This establishes~\eqref{h(psi)_circ_Phi(u)_is_0},
with $n=2m$.

We get from 
 Lemma~\ref{lem:elementary_facts}~\ref{lem:elementary_facts:projection_cyl_to_D} 
 $\cala$-chain homotopies
 \begin{eqnarray*}
 h(C) \colon \id_{\cyl(\phi)} & \simeq & l(C) \circ \pr(C);
 \\
 h(D) \colon \id_{\cyl(\psi)} & \simeq & l(D) \circ \pr(D).
 \end{eqnarray*}
 satisfying
 \begin{eqnarray*}
 \pr(C) \circ h(C) & = & 0;
 \\
 \pr(D) \circ h(D) & = & 0;
 \\
 h(D) \circ \overline{u} & = & \overline{u} \circ h(C),
 \end{eqnarray*}
 where $l(C) \colon C \to \cyl(\phi)$ and $l(D) \colon C \to \cyl(\psi)$ are the canonical inclusions,
 $\pr(C) \colon \cyl(\phi) \to C$ and $\pr(D) \colon \cyl(\psi) \to D$ the canonical projections,
 and $\overline{u}$ is the chain map $\cyl(\phi) \to \cyl(\psi)$ given by
 $\overline{u}_n = \Phi(u_{n-1}) \oplus \Phi(u_n) \oplus u_n \colon 
\Phi(C_{n-1}) \oplus \Phi(C_n) \oplus C_n \to \Phi(D_{n-1}) \oplus \Phi(D_n) \oplus D_n$.
Denote by $C'$ and $D'$ the iterated mapping cylinders. 
 \begin{eqnarray*}
 C' & := & \Phi^{n-2}(\cyl(\phi))\cup_{\Phi^{n-2}(C)} \Phi^{n-3}(\cyl(\phi)) \cup_{\Phi^{n-3}(C)} \dots \cup_{\Phi(C)} \cyl(\phi);
 \\
 D' & := & \Phi^{n-2}(\cyl(\psi))\cup_{\Phi^{n-2}(D)} \Phi^{n-3}(\cyl(\psi)) \cup_{\Phi^{n-3}(D)} \dots \cup_{\Phi(D)} \cyl(\psi).
 \end{eqnarray*}
 Denote by
 \begin{eqnarray*}
 i(C) \colon C & \to & C';
 \\
 i(D) \colon D & \to & D';
 \\
 i(\phi^{n-1}(C)) \colon \Phi^{n-1}(C) & \to & C';
 \\
 i(\phi^{n-1}(D)) \colon \Phi^{n-1}(D) & \to & D',
 \end{eqnarray*}
 the obvious inclusions. The various chain maps $\Phi^i(\pr(C))$ and $\Phi^i(\pr(D)$ fit together to projections
 \begin{eqnarray*}
 p(C) \colon C' & \to & C;
 \\
 p(D) \colon D' & \to & D.
 \end{eqnarray*}
 We have
 \begin{eqnarray*}
 p(C) \circ i(C) & = & \id_C;
 \\
 p(D) \circ i(D) & = & \id_D;
 \\
 p(C) \circ i(\Phi^{n-1}(C)) & = & \phi^{(n-1)};
 \\
 p(D) \circ i(\Phi^{n-1}(D)) & = & \psi^{(n-1)}.
 \end{eqnarray*}
 Since $\psi \circ \Phi(u) = u \circ \phi$, the various maps $\Phi^i(\overline{u})$ fit together to a cofibration
 \[
 u' \colon C' \to D'
 \]
satisfying
\begin{eqnarray*}
u' \circ i(C) & = & i(D) \circ u;
\\
p(D) \circ u' & = & u \circ p(C).
\end{eqnarray*}
The various chain homotopies  $\Phi^i(h(C))$ and $\Phi^i(h(D)$ fit together to chain homotopies
 \begin{eqnarray*}
 g(C) \colon  \id_{C'} & \simeq & i(C) \circ p(C);
 \\
 g(D) \colon  \id_{D'} & \simeq & i(D) \circ p(D),
 \end{eqnarray*}
 satisfying
 \begin{eqnarray*}
 p(C) \circ g(C) & = & 0;
 \\
 p(D) \circ g(D) & = & 0;
 \\
 u' \circ g(C) & = & g(D) \circ u.
 \end{eqnarray*}
 
 Next we define a chain maps $\phi' \colon \Phi(C') \to C'$ and $\psi' \colon \Phi(D') \to D'$. 
 The morphism $\psi'$ is constructed as follows: 
For $i<n-2$, on $\Phi\bigl(\Phi^i(\cyl(\psi))\bigr)$ it is given by the inclusion
 \[
 \Phi(\Phi^i(\cyl(H))) = \Phi^{i+1}(\cyl(H)) \to D'.
 \] 
 It remains to define $\psi'$ on $\Phi\bigl(\Phi^{n-2}(\cyl(\psi))\bigr)$.
 Consider the following (not necessarily commutative) diagram of $\cala$-chain complexes
 \[
 \xymatrix{\Phi^n(C) \ar[d]_{\Phi^{n-1}(\phi)} \ar[rrd]^0  \\
  \Phi^{n-1}(C) \ar[rr]_{i(\Phi^{n-1}(D))} && C'
 }
 \hfill \text{and} \hfill
 \xymatrix{\Phi^n(D) \ar[d]_{\Phi^{n-1}(\psi)} \ar[rrd]^0  \\
  \Phi^{n-1}(D) \ar[rr]_{i(\Phi^{n-1}(D))} && D'
 }
 \]
 where $i(\Phi^{n-1}(C))$ and $i(\Phi^{n-1}(D))$ are the obvious inclusions. Using~\eqref{phi(n)_is_zero},
we obtain  explicit chain homotopies
  of chain maps $\Phi^n(C) \to C'$ and $\Phi(D) \to D'$
 \begin{eqnarray*}
 k(C) = g(C) \circ   i(\Phi^{n-1}(C)) \circ \Phi^{n-1}(\phi)   \colon i(\Phi^{n-1}(C)) \circ \Phi^{n-1}(\phi) &\!\! \simeq 0;
 \\
 k(D) = g(D) \circ i(\Phi^{n-1}(D)) \circ \Phi^{n-1}(\psi)   + i(D) \circ h(\psi) \colon i(\Phi^{n-1}(D)) \circ \Phi^{n-1}(\psi) &\!\! \simeq  0,
 \end{eqnarray*}
 satisfying because of~\eqref{h(psi)_circ_Phi(u)_is_0}
 \begin{eqnarray*}
 k(D) \circ \Phi^n(u) & = & u' \circ k(C);
 \\
 p(C) \circ k(C) & = & 0.
 \end{eqnarray*}
 We obtain from
 Lemma~\ref{lem:elementary_facts}~\ref{lem:elementary_facts:maps_between_mapping_cones}
 chain maps $\Phi^{n-1}(\cyl(\phi)) \to C'$ and $\Phi^{n-1}(\cyl(\psi)) \to D'$ which will
 be declared to be  the restrictions of $\phi'$ and $\psi'$ to $\Phi^{n-1}(\cyl(\phi))$
 and $\Phi^{n-1}(\cyl(\psi)) \to D'$. This finishes the construction of the chain maps
 \begin{eqnarray*}
 \phi' \colon \Phi(C') & \to & C';
 \\
 \psi' \colon \Phi(D') & \to & D'.
 \end{eqnarray*}
 One easily checks
 \begin{eqnarray*}
 (\phi')^{(n)} & = & 0;
 \\
 (\psi')^{(n)} & = & 0;
 \\
 u' \circ \phi' & = & \psi' \circ \Phi(u');
 \\
 p(C) \circ \phi'  & = & \phi \circ \Phi(p(C));
 \\
 p(D) \circ \psi'  \circ \Phi(i(D)) & = & \psi.
 \end{eqnarray*}
 We define $(E,\mu)$ by the pushout in $\End(\Ch(\cala),\Phi)$
 \[
 \xymatrix{(C',\phi') \ar@{>->}[r]^{u'} \ar[d]_{p(C)}^{\simeq} 
 & (D',\psi') \ar[d]^{\overline{p(C)}}_{\simeq}
 \\
 (C,\phi) \ar@{>->}[r]_{j} & (E,\mu)
 }
 \]
 Since $\psi'$ is $\Phi$-nilpotent and $p(C)$ and hence $\overline{p(C)}$ are split surjective,
 also $\mu$ is $\Phi$-nilpotent.

 Letting $v':=\overline{p(C)}\circ i(D)$, we obtain an explicit chain homotopy of chain maps $\Phi(D) \to E$
 \[
 H' := \overline{p(C)} \circ g(D) \circ \psi' \circ \Phi(i(D))   
 \colon \mu  \circ \Phi(v') 
 \simeq v' \circ \psi.
 \]
 The following computation shows that $H'$ is stationary over $\Phi(C)$:
 \begin{eqnarray*}
 H' \circ \Phi(u) 
 & = & 
 \overline{p(C)} \circ g(D) \circ \psi' \circ \Phi(i(D)) \circ \Phi(u)
 \\
 & = & 
 \overline{p(C)} \circ g(D) \circ \psi' \circ \Phi(u') \circ \Phi(i(C)) 
 \\
 & = & 
 \overline{p(C)} \circ g(D) \circ u' \circ \phi'  \circ \Phi(i(C)) 
 \\
 & = & 
 \overline{p(C)} \circ   u' \circ g(C ) \circ \phi'  \circ \Phi(i(C)) 
 \\
 & = & 
 j \circ   p(C)  \circ g(C ) \circ \phi'  \circ \Phi(i(C)) 
 \\
 & = & 
 j \circ   0\circ \phi'  \circ \Phi(i(C)) 
 \\
 & = & 
 0,
 \end{eqnarray*}
This completes the verification of the assumptions of Reduction 2, and therefore completes the proof of the Lemma.
 \end{proof}

Now we have all the results available to conclude the proof.

\begin{proof}[Proof of Theorem~\ref{the:relative_term_is_Nil}]
By Lemma~\ref{lem:canonical_Waldhausen_structure_is_Waldhausen_structure}~%
\ref{lem:canonical_Waldhausen_structure_is_Waldhausen_structure:axioms}, the Waldhausen categories 
$\Ch(\cala_\Phi[t^{-1}])$ and $\Ch(\cala_\Phi[t\inv]^{\kappa})$ and hence also $\Ch(\cala_\Phi[t^{-1}])^w$  and $\Chhf(\cala_\Phi[t^{-1}])^w$ satisfy
the saturation and the cylinder axiom.
Cisinski's  Approximation Theorem~\ref{the:Cisinki_Approximation_Theorem}   implies that the inclusion
 \[
 \Ch(\cala_\Phi[t^{-1}])^w\to \Chhf(\cala_\Phi[t^{-1}])^w
 \]
induces an equivalence on $K$-theory. 

We claim that also the inclusion
 \[
 \Nil(\cala, \Phi)\to \HNil(\Chhf(\cala), \Phi).
 \]
 induces an equivalence on $K$-theory. In fact, $\Nil(\cala, \Phi)\to \HNil(\Chhf(\cala), \Phi)$ can be split into a
   sequence of inclusions
 \[ 
 \Nil(\cala, \Phi)\xrightarrow{I_1} \Nil(\Ch(\cala), \Phi) \xrightarrow{I_2}  \HNil(\Ch(\cala), \Phi) 
 \xrightarrow{I_3} \HNil(\Chhf(\cala), \Phi).
 \]
 We will show that each of the three inclusions $I_1$, $I_2$ and $I_3$ induce homotopy
 equivalence on $K$-theory.
 
The morphism $I_1$ induces a homotopy equivalence on $K$-theory
because of the Gillet-Waldhausen
Theorem~\ref{the:Gillet_Waldhausen} and Example~\ref{exa:twisted_Nil_category}
using the identity of Waldhausen categories $\Nil(\Ch(\cala), \Phi)  = \Ch(\Nil(\cala, \Phi))$.
 
 The maps $I_2$ and $I_3$ induce equivalences on $K$-theory by Cisinski's 
 Approximation Theorem~\ref{the:Cisinki_Approximation_Theorem}. We have to check the various assumptions
appearing in Theorem~\ref{the:Cisinki_Approximation_Theorem}.
The categories $\Ch(\cala^{\kappa},\Phi))$, $\Ch(\cala_{\Phi^{-1}}[t^{-1}]^{\kappa})$
 and $\Nil(\Ch(\cala) , \Phi)  = \Ch(\Nil(\cala, \Phi))$ satisfy the saturation axiom and the cylinder axiom
because of Lemma~\ref{lem:canonical_Waldhausen_structure_is_Waldhausen_structure}~%
\ref{lem:canonical_Waldhausen_structure_is_Waldhausen_structure:axioms}. 
We conclude 
that $\End(\Ch(\cala^{\kappa},\Phi))$ 
and the full Waldhausen subcategories
$\HNil(\Ch(\cala), \Phi)$ and $\HNil(\Chhf(\cala), \Phi)$
satisfy the saturation axiom and the cylinder axiom.
The inclusion functors $I_2$ and $I_3$  reflect weak equivalences by 
Lemma~\ref{lem:F_reflects_then_CH(F)_reflects}. 
The second approximation property appearing in 
Theorem~\ref{the:Cisinki_Approximation_Theorem} was shown to hold in
 Lemma~\ref{lem:endos_and_finiteness} and  Lemma~\ref{lem:endos_and_strictly_nilpotent}.

 Hence, $\bfK(\Nil(\cala, \Phi))\xrightarrow{\simeq} \bfK( \HNil(\Chhf(\cala), \Phi))$ as we claimed. 
Now Proposition~\ref{pro:reexpressing_nil} concludes the proof.
\end{proof}


\section{Passing to non-connective algebraic $K$-theory}
\label{sec:Passing_to_non-connective_algebraic_K-theory}

Given the Bass-Heller-Swan decomposition for connective $K$-theory, one may adapt Bass's
contracting functor approach to the setting of spectra. This yields a definition of  non-connective $K$-theory and Nil-spectra such that the Bass-Heller-Swan decomposition
automatically extends to the non-connective setting. This definition of nonconnective
$K$-theory agrees with the other definitions in the literature.

The details of the definitions of the non-connective versions of the $K$-groups
appearing in Theorem~\ref{the:BHS_decomposition_for_non-connective_K-theory}
and the argument  how Theorem~\ref{the:BHS_decomposition_for_non-connective_K-theory}  
can be deduced from the connective version,
are presented in~\cite[Section~6]{Lueck-Steimle(2013delooping)}.

\typeout{-----------------------  References ------------------------}



\end{document}